\documentclass[a4paper, 10pt]{article}

\usepackage{geometry}
\usepackage{graphicx}
\usepackage{amsmath, amsthm, amssymb,amsfonts}
\usepackage[utf8]{inputenc}
\usepackage[english, frenchb]{babel}
\usepackage[T1]{fontenc}
\usepackage{url}
\usepackage{verbatim}
\usepackage{indentfirst}
\usepackage{xypic}
\usepackage{braket}
\usepackage{fancyhdr}
\usepackage{enumerate}
\usepackage{xcolor}
\usepackage[colorlinks=true, linktoc=page, citecolor=blue, linkcolor=blue, urlcolor=blue]{hyperref}

\usepackage[titles]{tocloft}

\usepackage[nottoc]{tocbibind}

\usepackage{mathptmx}
\usepackage{mathrsfs}
\usepackage{lipsum}
\usepackage{colortbl}
\newcommand*\separateur{\begin{center} \includegraphics[scale=0.2]{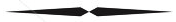}\end{center} }
\def\Fin#1{\leavevmode\unskip\nobreak\quad\hspace*{\fill}{#1}}

\DeclareFontFamily{OT1}{pzc}{}
\DeclareFontShape{OT1}{pzc}{m}{it}{<-> s * [1.1] pzcmi7t}{}
\DeclareMathAlphabet{\mathpzc}{OT1}{pzc}{m}{it}  

\usepackage{setspace}
\usepackage{titlesec}
\renewcommand\thepart{\textbf{Part \Roman{part}.}}
\titleformat{\part}
  {\sc  \center}
  {\thepart}{1em}{}

\titleformat{\section}
  {\bf \normalsize \center}
  {\thesection.}{1em}{}

\titleformat{\subsection}
  {\bf \normalsize}
  {\thesubsection.}{1em}{}

\titleformat{\subsubsection}
  {\bf \normalsize}
  {\thesubsubsection.}{1em}{}  %

\newtheoremstyle{defn}{}{}{}{}{\bfseries}{ : }{0cm}{}
\newtheoremstyle{theoreme}{6pt}{6pt}{\itshape}{}{\bfseries}{.}{3pt}{}

\theoremstyle{defn}

\theoremstyle{theoreme} 
\newtheorem{thm}{Theorem}[section]
\newtheorem{lem}[thm]{Lemma}
\newtheorem{hyp}[thm]{Assumption}

\newtheorem{prop}[thm]{Proposition}

\newtheorem{cor}[thm]{Corollary}

\theoremstyle{remark}
\newtheorem{rem}[thm]{Remark} 
\newtheorem{construction}[thm]{Construction}

\numberwithin{equation}{section}

\newcommand{\R}{\mathbb{R}}
\newcommand{\C}{\mathbb{C}}
\newcommand{\Z}{\mathbb{Z}}
\newcommand{\N}{\mathbb{N}}

\newcommand{\g}{\mathfrak{g}}
\newcommand{\ka}{\mathfrak{k}}
\newcommand{\pe}{\mathfrak{p}}
\newcommand{\norme}[1]{\left\Vert #1\right\Vert}

\newcommand{\ind}{\text{Ind}}

\newcommand{\ad}{\text{Ad}}

\title{On the analogy between real reductive groups and Cartan motion groups. II: Contraction of irreducible tempered representations}
\author{Alexandre Afgoustidis\\ \small \emph{CEREMADE, Université Paris-Dauphine}}
\date{}

\begin{document}
\selectlanguage{english}

\maketitle

\begin{abstract}

\noindent Attached to any reductive Lie group $G$ is a ``Cartan motion group'' $G_0$ $-$ a Lie group with the same dimension as $G$, but a simpler group structure. A natural one-to-one correspondence between the irreducible tempered representations of $G$ and the unitary irreducible representations of $G_0$, whose existence had been suggested by Mackey in the 1970s, has recently been described by the author.  In the present article, we use the existence of a family of groups interpolating between $G$ and $G_0$ to realize the bijection as a deformation: for every irreducible tempered representation $\pi$ of G, we build, in an appropriate Fr\'echet space, a family of subspaces and evolution operators that contract $\pi$ onto the corresponding representation of $G_0$.

 \end{abstract}

\tableofcontents

\section{Introduction}

Suppose $G$ is a reductive Lie group. Fix a maximal compact subgroup $K$ in $G$ and write $\g = \ka \oplus \pe$ for the corresponding Cartan decomposition of the Lie algebra of $G$.  The \emph{Cartan motion group} attached to $G$ (and $K$) is the semidirect product $G_0 = K \ltimes \pe$ associated with the adjoint action of $K$ on $\pe$.

The Lie groups $G$ and $G_0$ have the same dimension, but very different algebraic structures; their representation theories are, accordingly, quite different. The motion group $G_0$ is a split extension of a compact group by an abelian group: as a result, the classification of its unitary representations is among the simplest applications of Mackey's theory of induced representations \cite{MackeyPNAS49}. The representation theory of $G$ is more difficult: in general, much about the unitary dual $\widehat{G}$ remains mysterious, and the complete description of the tempered dual $\widetilde{G}$ achieved by Knapp and Zuckerman \cite{KnappZuckerman, KnappZuckerman2} is a very deep result which draws on much of Harish-Chandra's work.

Nevertheless, Mackey remarked in the early 1970s that his own description of the unitary dual $\widehat{G_0}$ featured parameters formally similar to those that appeared in Harish-Chandra's then-ongoing study of $\widetilde{G}$; he went on to wonder whether the analogy could give rise to a natural one-to-one correspondence between large subsets of $\widehat{G_0}$~and~$\widetilde{G}$.

His conjecture seems to have been met with skepticism. But later on, the Baum-Connes-Kasparov conjecture in operator algebra theory renewed interest in Mackey's observations (see \cite{ConnesLivreENG}, Ch. II, or \cite{BaumConnesHigson}, \S4); Nigel Higson eventually strengthened Mackey's rather tentative suggestion into a precise guess that  \emph{exactly the  same} parameters can be used to describe $\widehat{G_0}$ and $\widetilde{G}$ \cite{Higson2008}. That guess has recently turned out to be accurate: we constructed in \cite{AAMackey} a {natural bijection} between $\widehat{G_0}$ and $\widetilde{G}$, using Vogan's notion of lowest $K$-types (see \S \ref{subsec:correspondance}). \\

\noindent The existence of a common parametrization for the reduced duals of such different groups would perhaps seem incomprehensible, were it not for the existence of a continuous family $(G_t)_{t \in [0,1]}$ of Lie groups that interpolates between $G_0$ and $G=G_1$. The group $G$ can be viewed as the isometry group of the riemannian symmetric space $G/K$, while $G_0$ acts by affine isometries on the flat tangent space to $G/K$ at the identity coset. The  (constant) curvature of $G/K$ is negative and can be set to $(-1)$ with  suitable normalizations; a change of scale in that metric (akin to zooming-in on a neighborhood of the origin in $G/K$) yields, for each $t>0$, a symmetric space with curvature $-t^2$ whose isometry group $G_t$ is isomorphic with $G$. The motion group $G_0$ appears as a limiting case of the construction as the curvature parameter $t$ goes to zero (see \S \ref{subsec:defo}).

The relationship between $G$ and $G_0$ is thus of a kind often encountered in physics, when the structure constants determining the symmetries of a physical system are changed and brought to a limiting value (for instance when the speed of light becomes infinite): the motion group $G_0$ is a \emph{contraction} of $G$, in the sense of  \.{I}n\"{o}n\"{u}  and Wigner (\cite{InonuWignerContractions}; see also \cite{Segal}, \S 6). Given a Lie group $\Gamma$ and a contraction $\Gamma_0$, it is a classical problem in physics to try to relate the unitary representations of $\Gamma$ and $\Gamma_0$ (to be thought of as carrier Hilbert spaces for the states of quantum systems with different but related symmetries) through a limiting process, studying the behavior of representation spaces and operators as the physical parameter of interest becomes singular.  Mackey's conjecture had in fact been motivated by the fact that an infinite space with negative but very small curvature is as plausible a model for physical space as a flat one $-$ whence the idea that the representations of the corresponding symmetry groups ought to ``resemble'' one another.

In most cases of physical interest, the unitary representations of $\Gamma$ and $\Gamma_0$ do not correspond exactly under the contraction $-$ far from it. The most famous example is obtained in the transition from the Poincaré group to the Galilei group \cite{InonuWignerContractions}: the ``physically meaningful'' unitary irreducible representations do not  in any reasonable sense ``contract'' to unitary representations of the Galilei group. 

From this point of view, the existence of the Mackey-Higson bijection appears to involve a very unusual rigidity phenomenon in the deformation from $G$ to $G_0$. The fact that the representations paired by the bijection can look either very similar (as in the case of minimal principal series representations, where the rigidity has long been known \cite{DooleyRice, HelgasonTangentSpace}) or quite different (as in the case where $G$ admits discrete series representations, whose counterparts in $\widehat{G_0}$ are then much more trivial finite-dimensional representations), and the importance for several aspects of the representation theory of $G$ of the large-scale geometry of $G/K$ (while $G_0$ is a first-order approximation of $G$ near $K$), add to the mystery.

The mere existence of the deformation $(G_t)_{t \in [0,1]}$ does not, then, make the Mackey-Higson bijection any less unexpected. Yet, one can wonder whether the persistence phenomenon now established at the level of parameters is reflected in the behavior of representation spaces as the contraction is performed, and it is natural to ask whether the ``coincidence of parametrizations'' discovered by Mackey can be established by a deformation procedure.

That is the problem to be addressed here: when $G$ is a linear connected reductive group, we will indicate how every tempered irreducible representation $\pi$ of $G$ can be ``contracted'' $-$ at the level of geometrical realizations $-$ onto the representation $\pi_0$ of $G_0$ that corresponds to it in the Mackey-Higson bijection.
\vspace{-0.2cm}
\separateur
\vspace{-0.2cm}
Of course if one is to associate with $\pi$ a family $(\pi_t)_{t>0}$, where $\pi_t$ is a representation of $G_t$ for each $t>0$, and to try to prove that ``$\pi_t$ goes to $\pi_0$ as $t$ goes to zero'', one needs to give a meaning to the convergence. There does not seem to be a universally acknowledged method for doing so; the various notions of ``contraction of representations'' that have been used in mathematical physics are somewhat disparate in nature, and often formulated at the infinitesimal (Lie algebra) level; see \cite{InonuWignerGalilee, Hermann, Mukunda, MickelssonNiederle} for classical examples.

Our strategy  can be described, in somewhat imprecise terms, as follows (we postpone a precise statement of our results to  \S \ref{subsec:programme}: we will then have introduced enough notation to give a more formal description). 

\vspace{-0.1cm}Suppose $\pi$ is an irreducible tempered representation of $G$. Drawing from the existing constructions of geometrical realizations for $\pi$, we will attach to $\pi$ a Fréchet space $\mathbf{E}$, together with a family $(\pi_t)_{t>0}$ of representations in which each $\pi_t$, $t>0$, acts on a subspace $\mathbf{V}_t$ of $\mathbf{E}$.

This will generate a natural notion of \emph{evolution} in $\mathbf{E}$. Since each $G_t$, $t>0$, is isomorphic with $G$, the vectors in subspaces $\mathbf{V}_t$ that correspond to different values of $t$ will lie in distinct spaces that carry irreducible representations of $G$. These $G$-representations will be equivalent, or at least can be turned into equivalent representations if an appropriate (and classical) renormalization of their continuous parameters is performed (see \S \ref{subsec:programme}). Schur's lemma will determine  a way of following a vector in the subspace $\mathbf{V}_1$, and in fact any vector of $\mathbf{E}$, through the contraction. 

We will then be able to watch vectors and operators deform according to that natural evolution. What we will prove is that as $t$ goes to zero, the initial carrier space and operators for $\pi$ do ``converge'', in a precise sense, to a subspace of $\mathbf{E}$ and operators yielding a linear representation of $G_0$ on that subspace; furthermore, the equivalence class of the $G_0$-representation obtained in this way corresponds to $\pi$ in the Mackey-Higson bijection.

 We should mention that except in the case of minimal principal series, the settings we will use (for instance the description of discrete series representations using Dirac induction, or the realization of real-infinitesimal-character representations in Dolbeault cohomology spaces) force upon us a choice of Fréchet space whose topology is quite different from the Hilbert space topology inherited from the unitary structure of the representations. Thus, our geometrical analysis of Mackey's analogy is made possible only by the deep work of many mathematicians (among whom Helgason, Parthasarathy, Atiyah, Schmid, Vogan, Zuckerman, Wong), ranging from the early 1970s to the late 1990s, on the geometric realization of irreducible tempered $(\g, K)$-modules.

The most suggestive and important examples are technically much simpler to deal with than the general case: before embarking on the contraction of an arbitrary tempered representation, we will take the time  to inspect the contraction of a discrete series representation to its (unique) lowest $K$-type (see \S \ref{subsec:serie_discrete}), as well as the contraction of minimal principal series representations (see \S \ref{sec:serie_principale}) to generic representations of $G_0$ $-$ for the latter case, our results are closely related to earlier treatments in the physical literature and in \cite{DooleyRice}. Our general procedure, relying on successive reductions of the general case to that of real-infinitesimal character representations of connected semisimple groups, then to that of irreducible constituents of real-infinitesimal-character principal series representations of quasi-split groups, is more technical (especially in \S \ref{sec:reduction}); but most of the important ideas and many useful simple lemmas are already present in the two extreme examples. Section \ref{sec:real_inf_char} contains our description of the contraction process for real-infinitesimal-character representations. Section \ref{sec:contraction_generale} reduces the general case to that one.

 The analytical methods to be used here are hardly the final word on the relationship between Mackey's analogy and deformation theory. In the time elapsed since part of these results (corresponding roughly to \S \ref{sec:discreteprincipale} and \S \ref{sec:contraction_generale}) began circulating in preprint form in 2015, Bernstein, Higson and Subag have suggested that the methods of algebraic geometry may provide a framework for studying (algebraic) deformations of groups and families of representations in a broader context. Their recent work \cite{BHSPub, BHS2Pub, Subag, SubagH, BBS}, as well as Shilin Yu's recent attempts to obtain an algebraic-geometric realization of Mackey's correspondence using deformations of algebraic $\mathcal{D}$-modules \cite{ShilinSL2, Shilin}, make it reasonable to hope that insight will be gained in this way on Mackey's analogy. To mention only the simplest related issue, it will soon be obvious that the analytic methods and evolution operators in Fréchet space of the present article, in which no passage to subquotients is needed at any stage, cannot be well-suited to studying the relationship between deformation theory and the natural correspondence between the admissible duals introduced in the last section of \cite{AAMackey}: it seems reasonable to turn to more algebraic methods for that study. 
 
 As one referee pointed out, it would certainly bring further light on our subject to determine whether there exists a deformation procedure that can be constructed independently of the extensive knowledge of the tempered dual, geometrical realizations, and  lowest $K$-types, on which this paper entirely rests. The algebraic methods mentioned above may prove useful in that direction.

\paragraph{Acknowledgments.} \small Most of the material in \S \ref{sec:prelim}, \ref{sec:discreteprincipale} and \ref{sec:contraction_generale} appeared in preliminary form in my Ph.D. thesis \cite{AAthese}. I am profoundly indebted to Daniel Bennequin's guidance and support. I am happy to thank Michel Duflo, Nigel Higson and Michèle Vergne for their help and advice at various stages of this investigation, as well as François Rouvière and David Vogan, whose reading of my thesis led to major corrections and improvements.  The ideas in \S\ref{sec:reduction} took form only after Nicolas Prudhon pointed out \cite{Wong99} to me, and I may not have completed the present paper without the frequent conversations I had with~Jeremy~Daniel.

With the exception of \S \ref{sec:real_inf_char}, the results below were conceived around 2015; it was a happy time at Université Paris-7 and the Institut de Mathématiques de Jussieu $-$ Paris Rive Gauche. Section \ref{sec:real_inf_char}, which is crucial to this paper, is much more recent: I am grateful in many ways to my colleagues at Université Paris-Dauphine and CEREMADE. 

I thank the referees for their careful comments, which helped improve this paper. 
\normalsize

\section{Preliminaries} \label{sec:prelim}

\subsection{The Mackey-Higson bijection} \label{subsec:correspondance}

\noindent In \cite{AAMackey}, \S 2.2, we called \emph{Mackey parameter} any pair $(\chi, \mu)$ where $\chi$ is an element of the vector space dual $\pe^\star$ of $\pe$, and $\mu$ is an irreducible representation of the stabilizer $K_\chi$ of $\chi$ in $K$ (for the coadjoint action of $K$ on $\pe^\star$). 

Attached to any Mackey parameter $(\chi, \mu)$ is a unitary irreducible representation of $G_0$: we form the centralizer $L_0^\chi = K_\chi \ltimes \pe$ of $\chi$ in $G_0$ and consider the induced representation
\begin{equation} \mathbf{M}_0(\chi, \mu)=\ind_{L^0_\chi}^{G_0}(\sigma) = \ind_{K_\chi \ltimes \pe}^{G_0}\left(\mu\otimes e^{i\chi}\right).\end{equation}
It will be useful to record a construction of $\mathbf{M}_0(\chi, \mu)$ and three elementary remarks.

\begin{construction} Fix a carrier space $V$ for $\mu$ and a $\mu(K)$-invariant inner product on $V$. Equip the Hilbert space
\begin{equation} \label{espaceG0} \mathbf{H}= \left\{ f \in \mathbf{L}^2(K, V), \ \forall m \in K_{\chi}, \forall k \in K, \ f(km)=\mu(m^{- 1})f(k) \right\}\end{equation}
with the $G_{0}$-action in which 
\begin{equation} \label{actionG0} g_{0} = (k,v) \text{ acts through } \pi_0(k,v): f \mapsto \left[ u \mapsto e^{i \langle Ad^\star(u)\chi, v\rangle} f(k^{- 1} u)\right].\end{equation}
Then the representation $(\mathbf{H}, \pi_0)$ of $G_0$ is equivalent with $\mathbf{M}_0(\chi,\mu)$.\end{construction}

\begin{rem} Suppose $\chi$ is zero. Then $K_\chi$ equals $K$, $\mu$ is an irreducible representation of $K$, and the representation $\mathbf{M}_0(\chi, \mu)$ of $G_0$ is the trivial extension of $\mu$ where $\pe$ acts by the identity on $V$. Thus $\mathbf{M}_0(\chi,\mu)$ is finite-dimensional in that case.\end{rem}

\begin{rem} \label{remarque_fourier} For arbitrary $\chi$, suppose $\mu$ is the trivial representation of $K_\chi$. Then another geometric realization for $\mathbf{M}_0(\chi, \mu)$ may be obtained by considering the space $\mathscr{E}$ of tempered distributions on $\pe$ whose Fourier transform (a tempered distribution on $\pe^\star$) has support contained in the orbit $\ad^\star(K) \cdot \chi$; since that orbit is compact,  $\mathscr{E}$ is a space of smooth functions on $\pe$. The functions in $\mathscr{E}$ are ``combinations'' of Euclidean plane waves on $\pe$ whose wavevectors lie in the orbit $\ad^\star(K) \cdot \chi$. The realization \eqref{espaceG0}-\eqref{actionG0} of $G_0$ can be viewed as transferred from the quasi-regular action of $G_0$ on $\mathscr{E}$ by taking Fourier transforms.\end{rem}

\begin{rem} \label{remarque_regulier} Fix a maximal abelian subspace $\mathfrak{a}$ of $\pe$, and view the vector space dual $\mathfrak{a}^\star$ as a subspace of $\pe^\star$. Then every Mackey parameter $(\chi, \mu)$ is conjugate under $K$ to one in which $\chi$ lies in $\mathfrak{a}^\star$. If $\chi$ lies in $\mathfrak{a}^\star$ and is a regular element of $\mathfrak{a}^\star$ (\cite{KnappPrincetonBook}, \S V.3), then $K_\chi$ is equal to the group $M = Z_K(\mathfrak{a})$. 
\end{rem}

 We can also attach to any Mackey parameter $(\chi, \mu)$ a tempered irreducible representation of $G$. Form the centralizer $L_\chi$ of $\chi$ in $G$ (for the coadjoint action). There exists a parabolic subgroup $P_\chi = L_\chi N_\chi$ of $G$ with Levi factor $L_\chi$. Write $P_\chi = M_\chi A_\chi N_\chi$ for the Langlands decomposition of $P_\chi$: the  group $M_\chi$ admits $K_\chi$ as a maximal compact subgroup, is usually disconnected, but is linear reductive with abelian Cartan subgroups. The representation $\mu$ of $K_\chi$ determines, through Vogan's theory of lowest $K$-types \cite{Vogan81}, a unique irreducible tempered representation of $M_\chi$: among the representations of $M_\chi$ with real infinitesimal character, there is exactly one that admits $\mu$ as a lowest $K_\chi$-type (and then $\mu$ is its only lowest $K_\chi$-type):  see \cite{VoganBranching}, Theorem 1.3, or \cite{AAMackey}, \S 3.1. We write $\mathbf{V}_{M_\chi}(\mu)$ for it and define a representation of $G$ as 
\begin{equation} \label{inductionG} \mathbf{M}(\chi, \mu) = \ind_{M_\chi A_\chi N_\chi}(\mathbf{V}_{M_\chi}(\mu) \otimes e^{i\chi} \otimes \mathbf{1})\end{equation}
(for a construction, see \S \ref{sec:serie_principale}).

It is a classical result of Mackey \cite{MackeyPNAS49} that every unitary irreducible representation of $G_0$ is equivalent with $\mathbf{M}_0(\chi, \mu)$ for some Mackey parameter $(\chi, \mu)$; the corresponding tempered representation $ \mathbf{M}(\chi, \mu)$ is always irreducible, and its class in $\widetilde{G}$ depends only on the class of $\mathbf{M}_0(\chi, \mu)$ in $\widehat{G_0}$.  The above thus determines a map  $\mathcal{M}: \widehat{G_0} \to \widetilde{G}$.

\begin{thm}[\cite{AAMackey}, Theorem 3.2] \label{correspondance} The map $\mathcal{M}: \widehat{G_0} \to \widetilde{G}$ is a bijection between the unitary dual of $G_0$ and the tempered dual of $G$. \end{thm} 

\subsection{Deformation to the motion group} \label{subsec:defo}

\subsubsection{A family of Lie groups} \label{subsec:famille}

\noindent For every nonzero real number $t$, we can use the global diffeomorphism
\begin{align*} \varphi_t: K \times \pe & \rightarrow G \\ (k,v) & \mapsto \exp_G(tv) k \end{align*}
to endow $K \times \pe$ with a group structure which makes it isomorphic with $G$; the idea dates back to \cite{DooleyRice}. We write $G_t$ for that group, so that the underlying set of $G_t$ is $K \times \pe$, and the product law reads
\[ (k_1, v_1) \cdot_t (k_2, v_2)  = \varphi^{- 1}_{t} \left(\varphi_{t}(k_{1},v_{1}) \cdot \varphi_{t}(k_{2},v_{2})\right).\]

The Cartan motion group $G_0$ also has underlying set $K \times \pe$, and its product law reads $(k_1, v_1) \cdot_0 (k_2, v_2) = (k_{1}k_{2},v_{1}+\text{Ad}(k_{1})v_{2})$. The product law for $G_t$ goes to the product law for $G_0$ as $t$ goes to zero:

\begin{lem} \label{rouviere} {Suppose $(k_{1}, v_{1})$ and $(k_{2}, v_{2})$ are two elements of $K \times \pe$; then 
 $(k_1, v_1) \cdot_t (k_2, v_2)$ goes to $(k_1, v_1) \cdot_0 (k_2, v_2)$ when $t$ goes to zero. The convergence is uniform on compact subsets of  $K \times \pe$.}\end{lem}

This is quite well-known, but I have not been able to locate a precise justification in the literature. Given that  the very existence of the present article depends on Lemma \ref{rouviere}, it is perhaps reasonable to include a proof: this one was communicated to me by François Rouvière. 

 Write $(k_{1}, v_{1}) \cdot_{t} (k_{2}, v_{2})$ as $(k(t), v(t))$, where $k(t)$ is in $K$ and $v(t)$ in $\pe$; what we have to prove is that when $t$ goes to zero, $k(t)$ goes to $k_{1}k_{2}$ and $v(t)$ goes to $v_{1}+\text{Ad}(k_{1})v_{2}$. We start from the fact that
 \begin{equation} \label{compo} e^{tv(t)}k(t) = \left(e^{tv_{1}} k_{1}\right) \left(e^{tv_{2}}k_{2}\right) = e^{tv_{1}} e^{t \text{Ad}(k_{1}) v_{2}} k_{1} k_{2}. \end{equation}
$\bullet$ Write \vspace{-0.5cm}
\begin{align} \label{cartan} u :\pe&  \to G/K  \\ X & \mapsto \exp_{G}(X)K \nonumber\end{align} for the global diffeomorphism from the Cartan decomposition, and $(g,X) \mapsto g \cdot x = u^{-1}\left(g\exp_G(X)K\right)$ for the corresponding action of $G$ on $\pe$; then we deduce from \eqref{compo} that $u[tv(t)]= e^{tv_{1}} \cdot u\left[t\text{Ad}(k_{1})v_{2}\right]$. When $t$ goes to zero, the right-hand side of the last equality goes to the origin $1_{G}K$ in $G/K$, so that $tv(t)$ must go to the origin of $\pe$; taking up \eqref{compo} we deduce that $k(t)$ goes to $k_{1}k_{2}$.

\noindent $\bullet$ Let us turn to the convergence of $v(t)$. Apply the Cartan involution $\theta$ with fixed-point-set $K$ to both sides of \eqref{compo}; we get $e^{- tv(t)} k(t) = e^{- v_{1}} e^{- t \text{Ad}(k_{1}) v_{2}}k_{1}k_{2}$; taking inverses and multiplying by \eqref{compo} yields 
\begin{equation} \label{compo2} e^{2tv(t)} = e^{tv_{1}}e^{2t \text{Ad}(k_{1})v_{2}}e^{tv_{1}}. \end{equation}
From the Campbell- Hausdorff formula, we know that for small enough $t$, the right-hand side can be written as $e^{2t\left( v_{1}+\text{Ad}(k_{1})v_{2} + r(t)\right)}$, where $r(t)$ is an element of $\mathfrak{g}$ (the sum of a convergent Lie series) and  $r(t) = O(t)$  in the Landau ``big $O$'' notation. If $t$ is close enough to zero, then both $2tv(t)$ and $2t\left( v_{1}+\text{Ad}(k_{1})v_{2} + r(t)\right)$ lie in a neighborhood of the origin $\g$ over which $\exp_{G}$ is injective, so that \eqref{compo2} yields $v(t) =  v_{1}+\text{Ad}(k_{1})v_{2} + r(t)$. We deduce that for small enough $t$, $r(t)$ lies in $\pe$, and of course since $r(t)=O(t)$ we see that $v(t)$ goes to $v_{1} + \text{Ad}(k_{1})v_{2}$ as $t$ goes to zero. The last statement (the fact that the convergence is uniform on compact subsets) is now clear from the existence of explicit formulae for $r(t)$ and $k(t)(k_1k_2)^{-1}$. \qed 


\subsubsection{A family of actions and metrics on $\pe$} \label{subsec:cadregeom}

\noindent In this section, we describe a family of actions and a family of metrics on $\pe$. We shall use them in \S \ref{sec:discreteprincipale} to describe the contraction process for \emph{discrete series} and \emph{spherical principal series} representations: these two cases will introduce us to the main ideas and lemmas to be used for the general case in sections \ref{sec:real_inf_char} and \ref{sec:contraction_generale}.
\begin{center}
\begin{figure}
\begin{center}

\includegraphics[width=0.57\textwidth]{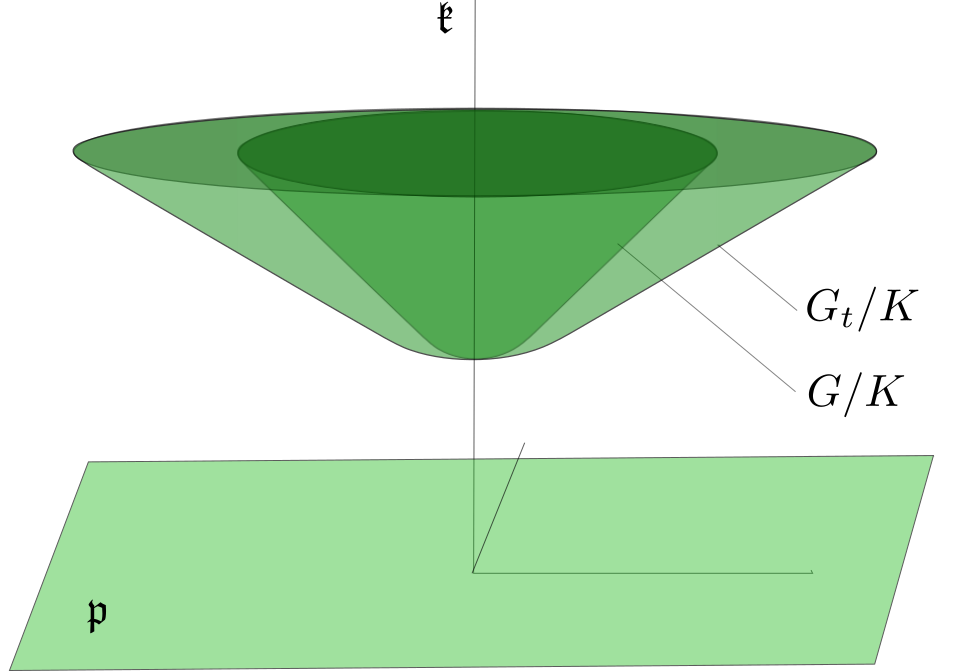} 
\caption{This is a picture of (co)-adjoint orbits for $G$ and $G_t$ when $G$ is $SL(2,\R)$. The horizontal plane is the space $\pe$ of trace-zero symmetric matrices, the vertical axis is the line $\ka$ of antisymmetric matrices; the drawn manifolds are the $G$-adjoint orbit and the $G_{t}$-coadjoint orbit of the point on the vertical axis. It is tempting, but incorrect, to interpret the diffeomorphism $u_t^{-1}$ from \eqref{cartanGt} as the vertical projection; however, the vertical projection can be recovered as $u_t^{-1} \circ \tau$, where $\tau: x \mapsto \frac{\sinh(\norme{x}_B)}{\norme{x}_B} \text{Rotation}_{\frac{\pi}{2}}(x)$. The map $\tau$ accounts for the fact that the geodesics in $(\pe, \eta_t)$ and $(\pe, \eta_0)$ do not spread at the same speed.  }
\end{center}
\end{figure}
\end{center}
 Suppose $t$ is a real number (we do not prohibit $t=0$ here). Then $G_t$  acts on the symmetric space $G_t/K$; because of the Cartan decomposition,  the map
\begin{equation} \label{cartanGt} u_t:  \pe \overset{\exp_{G_{t}}}{\longrightarrow} G_{t} \overset{\text{quotient}}{\twoheadrightarrow} G_{t}/K\end{equation}
is a global diffeomorphism. This yields a transitive action of $G_t$ on $\pe$: if $\gamma$ is an element of $G_t$, we write the action of $\gamma$ on $\pe$ as $x \mapsto \gamma \cdot_t x \overset{\text{def}}{=} u_t^{-1}\left[ \text{mult}_{G_t}\!\left(\gamma, \exp_{G_t}(x)\right) \ K \right]$, where $\text{mult}_{G_t} : G_t \times G_t \to G_t$ is the group law for $G_t$.

The action of $G_t$ on $\pe$ induces a $G_t$-invariant riemannian metric on $\pe$: we choose a fixed $K$-invariant inner product $B$ on $\pe$ and write $\eta_t$ for the (uniquely determined) $G_t$-invariant metric on $\pe$ which coincides with $B$ at the origin of $\pe$. The metric $\eta_t$ has constant scalar curvature, equal to $-t^2$: thus $(\pe, \eta_t)$ is negatively curved for every nonzero $t$, while $(\pe, \eta_0)$ is Euclidean.

In the proof of Lemma \ref{rouviere}, we also used the Cartan decomposition of $G$ to define an action $(g, x) \mapsto g \cdot x$ of $G$ on $\pe$. Since each $G_t$, $t \neq 0$, is obtained from $G$ by a simple rescaling in the Cartan decomposition, there must be a simple relationship between actions and metrics we just defined on $\pe$. The precise relationship can be phrased using the  isomorphism $\varphi_t: G_t \rightarrow G$ from \S \ref{subsec:famille}, and the dilation \begin{align*} z_{t} \ : \pe& \rightarrow \pe \\ \ x & \mapsto  {x}/{t}.\end{align*}
\begin{lem}   \label{conjactions} \begin{enumerate}[(i)]
\item For every $x$ in $\pe$ and $g$ in $G$, we have $\varphi_t^{-1}(g) \cdot_t z_t(x) =z_t(g \cdot x)$.
\item   The metrics $\eta_t$ and $\eta_1$ are related through $z_{t}^\star \eta_{t} = t^{- 2} \cdot \eta_{1}$.
\end{enumerate}\end{lem}
\begin{proof} Fix $(g,x)$ in $G \times \pe$ and observe the image of both sides of (i) by the diffeomorphism $u_t$: on~the~left-hand~side, \[ u_{t}\left( \varphi_{t}^{- 1}(g) \cdot_{t} z_{t}(x) \right)  = \varphi_{t}^{- 1}(g) \exp_{G_{t}}(z_{t}x) K
= \varphi_{t}^{- 1}(g \exp_{G}(x))K ; \]
on the right-hand side, we see from the definitions of the actions that
\[ u_{t}\left( z_{t} \left(g \cdot x\right) \right)  =  \exp_{G_{t}}\left[z_{t}(g \cdot x) \right] K = \varphi_{t}^{- 1}( \exp_{G}(g \cdot x))K  = \varphi_{t}^{- 1} \left( g \exp_{G}(x) \right)K. \]
This proves (i). For (ii), we know from (i) that $z_{t}^\star \eta_{t}$ is a $G_1$-invariant metric on $\pe$. But so is $\eta_1$: since the derivative of $z_{t}$ at zero is multiplication by $t^{- 1}$, we know that $z_{t}^\star \eta_{t}$ and $t^{- 2} \cdot \eta_{1}$ coincide at zero, so they must be equal. 
\end{proof}
We now give a precise statement for the idea that as $t$ goes to zero, the $G_t$-action on $\pe$ ``goes to'' the natural action of $G_0$ on $\pe$ by rigid motions. Each $G_t$, $t \in \R$, has underlying set $K \times \pe$, so the action is encoded by a map
\begin{align*} \mathcal{A}_{t} \ : (K \times \pe) \times \pe & \rightarrow \pe \\ \ \left( (k,v), x\right) & \mapsto (k,v) \cdot_t x. \end{align*}
\begin{lem} \label{convactions} For the topology of uniform convergence on compact subsets,  $\mathcal{A}_t$ goes to $\mathcal{A}_0$ when $t$ goes to zero.  \end{lem}
\begin{proof} For every $\left((k,v), x \right)$ in $(K \times \pe) \times \pe$ and $t$ in $\R$, we first write
 \begin{align*} (k,v) \cdot_{t} x &  = \left(\varphi_{t}^{- 1} \left[ \varphi_1(k, tv) \right]\right) \cdot_{t} x  =  \frac{1}{t} \left[ \varphi_1(k, tv) \right] \cdot (tx)  \ \ \ \text{by lemma \ref{conjactions}(i).} \end{align*}
What we need to prove is that this goes to $v + \text{Ad}(k)x$ when $t$ goes to zero. Apply the diffeomorphism $u$: we get
\small \begin{align*} u\left( \frac{1}{t} \varphi_1(k, tv) \cdot \left(tx\right) \right) & = \exp_{G}\left(\frac{1}{t} u^{- 1} \left[ \varphi_1(k, tv) \exp_{G}(tx) K \right] \right) K \\
  & = \exp_{G}\left(\frac{1}{t} u^{- 1} \left[ \exp_{G}(tv) k \exp_{G}(tx) K \right] \right) K \\
    & = \exp_{G}\left(\frac{1}{t} u^{- 1} \left[ \exp_{G}(tv) \exp_{G}(t\text{Ad}(k)[x]) K \right] \right) K.\end{align*} \normalsize
In the proof of Lemma \ref{rouviere}, we saw that $e^{tv}e^{t\text{Ad}(k)[x])}$ reads $e^{t \beta(t)}\kappa(t)$, where $\kappa(t)$ lies in  $K$ and $\beta(t)$ in $\pe$, with $\beta(t) = v + \text{Ad}(k)[x] + r(t)$, $r(t) \in \pe$, $r(t) = O(t)$. We deduce that $\exp_{G}(tv) \exp_{G}(t\text{Ad}(k)[x]) K = \exp_{G}(t\beta(t))K = u\left[t\beta(t)\right]$ and that 
   \[ u\left( \frac{1}{t} \varphi_1(k, tv) \cdot \left(tx\right) \right)  = \exp_{G}\left(\frac{1}{t} \left[ t\beta(t)\right] \right)K = u\left(v + \text{Ad}(k)[x] + r(t) \right), \]
  so $ \frac{1}{t} \varphi_1(k, tv) \cdot \left(tx\right)$ does go to $v+\text{Ad}(k)x$ as $t$ goes to zero. The fact that the convergence is uniform on compact subsets is immediate from Lemma \ref{rouviere}.\end{proof}

\subsection{Fréchet contractions and a program for \S 3-5}
\label{subsec:programme}

\noindent Suppose $\pi$ is an irreducible representation of $G$ and $\pi_0 = \mathcal{M}^{-1}(\pi)$ is the corresponding representation of $G_0$ in the Mackey-Higson bijection of  \S \ref{subsec:correspondance}. We wish to give a meaning to the phrase: ``associated to the deformation $(G_t)_{t \in \R}$ is a continuous family $(\pi_t)_{t \in \R}$ of representations, so that $\pi_1=\pi$ and $\pi_t$ goes to $\pi_0$ as $t$ goes to zero''. Of course a difficulty is that as $t$ varies, the representations $\pi_t$ usually will act on different spaces, and we need to choose an appropriate notion of ``deformation of representations''.\\

\noindent For every irreducible tempered representation $\pi$ of $G$, we shall describe
 \begin{itemize}
\item[$\bullet$] a Fréchet space $\mathbf{E}$, \Fin{(A)}
\item[$\bullet$] a family $(\mathbf{V}_t)_{t>0}$ of linear subspaces of $\mathbf{E}$, \Fin{(B)}
\item[$\bullet$] for each $t>0$, an irreducible representation $\pi_t: G_t \rightarrow \text{End}(\mathbf{V}_t)$ acting on $\mathbf{V}_t$, \Fin{(B')}
\item[$\bullet$]  a family $(\mathbf{C}_t)_{t>0}$ of endomorphisms of $\mathbf{E}$, to be called the \emph{contraction operators}, \Fin{(C)}
\end{itemize}
with the folllowing properties.\\
\begin{itemize}
\item[$\bullet$] \emph{For every $f$ in $\mathbf{E}$ and every $t>0$, write $f_t$ for $\mathbf{C}_t f$.}

\begin{itemize}
\item \emph{If $f$ lies in $\mathbf{V}_1$, then for every $t>0$, $f_t$ lies in $\mathbf{V}_t$,} \Fin{(Vect1)}
\item \emph{For every $f$ in $\mathbf{E}$, there is a limit (in $\mathbf{E})$ to $f_t$ as $t$ goes to zero.}\Fin{(Vect2)}\\
\end{itemize}
\item[$\bullet$]  \emph{Set $\mathbf{V}_0 = \left\{ \underset{t \to 0}{\text{\emph{lim}}}(f_t), \ f \in \mathbf{V}_1 \right\}$.}\\
\vspace{-0.1cm}
\begin{itemize}
\item \emph{Fix $(k,v)$ in $K \times \pe$. Fix $F$ in $\mathbf{V}_0$ and $f \in \mathbf{V}_1$ such that $F= \underset{t \rightarrow 0}{\text{lim}}(f_t)$. } \Fin{(Op1)}\\
\emph{ Then there is a limit to $\pi_t(k,v) \cdot  f_t$ as $t$ goes to zero; the limit depends only on $F$. }
\item \emph{Write $\pi_0(k,v)$ for the  map from $\mathbf{V}_0$ to itself obtained from (Op1). \Fin{\emph{(Op2)}}\\
Then $\pi_0(k,v)$ is linear for every $(k,v)$ in $K \times \pe$; the map $\pi_0: K \times \pe \rightarrow \text{\emph{End}}(\mathbf{V}_0)$ defines a morphism from $G_0$ to $\text{\emph{End}}(\mathbf{V}_0)$. The $G_0$-representation $(\mathbf{V}_0, \pi_0)$ is unitary irreducible. Its equivalence class is that which corresponds to $\pi$ in the Mackey-Higson bijection.} \\
\end{itemize}
\end{itemize}

\noindent For each $t>0$, our contraction operator $\mathbf{C}_t$ will be uniquely determined (up to a multiplicative constant) by a natural equivariance condition, together with Schur's lemma. To state the equivariance condition, we simplify the notation by identifying $G$ with $G_1$ through the isomorphism $\varphi_1: G_1 \to G$ of \S \ref{subsec:defo}, and $\pi$ with the representation $\pi_1 \circ \varphi_1^{-1}$ of $G$ acting on $\mathbf{V}_1$. Our requirement is most easily phrased in a special case:
\begin{equation} \tag{$\text{Nat}_{\text{RIC}}$} \text{\emph{ If $\pi$ has real infinitesimal character, we require that $\mathbf{C}_t$ intertwine $\pi$ and $\pi_t \circ \varphi_t^{-1}$.}} \end{equation}

\noindent If $\pi$ does not have real infinitesimal character, we require that $\mathbf{C}_t$ intertwine $\pi_t \circ \varphi_t^{-1}$, not quite with $\pi$, but with a representation obtained from $\pi$ by renormalizing the imaginary part of the infinitesimal character. We bring in the renormalization map $\mathcal{R}^{1/t}_{{G}}: \widetilde{G} \to \widetilde{G}$ of \cite{AAMackey}, \S 4:  starting from an irreducible tempered representation of $\pi$ of $G$, we use the  Knapp-Zuckerman classification to write $\pi$ as $\text{Ind}_{M_{P}A_{P}N_{P}}^G\left( \sigma \otimes e^{i\nu} \right)$ where $M_{P}A_{P}N_{P}$ is a cuspidal parabolic subgroup of $G$, $\sigma$ is a discrete series or nondegenerate limit of discrete series of $G$, and $\nu$ is an element of $\mathfrak{a}_P^\star$; then we define $\mathcal{R}^{1/t}_{{G}}(\pi)$ as the equivalence class of the irreducible representation $\text{Ind}_{M_{P}A_{P}N_{P}}^G\left( \sigma \otimes e^{i\frac{\nu}{t}} \right)$.
\begin{equation} \tag{$\text{Nat}_{\text{gen}}$} \text{ \emph{For arbitrary $\pi$,  we require that $\mathbf{C}_t$ intertwine $\mathcal{R}^{1/t}_{{G}}(\pi) $ and $\pi_t \circ \varphi_t^{-1}$.}} \end{equation}
Of course we shall explain why the renormalization in ($\text{Nat}_{\text{gen}}$) is natural: see Proposition \ref{renorm2}, Remark \ref{rem_identification} and Remark \ref{spreading}.

 In \S 3-5 below, we will implement this program for every irreducible tempered representation of $G$. 

When we will have succeeded in carrying it out for a representation $\pi$, we shall say that the quadruple $\left( \mathbf{E}, (\mathbf{V}_t)_{t>0}, (\pi_t)_{t >0}, (\mathbf{C}_t)_{t>0} \right) $ describes a \emph{Fréchet contraction of $\pi$ onto $\pi_0$. }

\addtocontents{toc}{\setcounter{tocdepth}{2}}
\section{Examples of Fréchet contractions: the discrete series and the spherical principal series}
\label{sec:discreteprincipale}

\noindent We here illustrate the program outlined in \S \ref{subsec:programme} by implementing it for two particular classes of irreducible tempered representations of $G$. 

There are at least three good reasons for beginning with these extreme and opposite examples before turning to more general results in sections \ref{sec:real_inf_char} and \ref{sec:contraction_generale}, even though the later results will imply those of this section: 

\begin{itemize}
\item[$\bullet$] All of the most important ideas in \S \ref{sec:real_inf_char} and \S \ref{sec:contraction_generale} will appear in these two examples, but they will be easier to discern without the technicalities to be met later.
\item[$\bullet$] In our two examples, the representations admit geometric realizations as spaces of functions (or sections of a finite-rank vector bundle) on $\pe$: the elementary results of \S \ref{subsec:cadregeom} will make the proofs particularly simple.
\item[$\bullet$] Most of the technical lemmas to be proven along the way will be used again in \S \ref{sec:real_inf_char} and \S \ref{sec:contraction_generale}.
\end{itemize} 
\subsection{Contraction of discrete series representations} \label{subsec:serie_discrete}

\noindent Throughout \S \ref{subsec:serie_discrete}, we assume that $G$ is a linear connected semisimple Lie group with finite center and  that $G$ and $K$ have equal ranks, so that $G$ has a nonempty discrete series. We fix a discrete series representation $\pi$ of $G$ and recall that the representation of $G_0$ that corresponds to $\pi$ in the Mackey-Higson bijection is the (unique) lowest $K$-type of $\pi$. We thus aim at describing a Fréchet contraction of $\pi$ onto its lowest $K$-type $\mu$.

\subsubsection{Square-integrable solutions of the Dirac equation}\label{subsec:dirac}

\noindent We here recall  results of Parthasarathy, Atiyah and Schmid \cite{AtiyahSchmid, Parthasarathy} which provide a geometrical realization~for~$\pi$. 

Fix a maximal torus $T$ in $K$. We will write $\Lambda \subset i\mathfrak{t}^\star$ for the set of linear functionals that arise as the derivative of a unitary character of $T$, $\Delta_{c}$ for the set of roots of $(\ka_{\mathbf{C}}, \mathfrak{t}_{\mathbf{C}})$, and $\Delta$ for the set of roots of $(\g_{\mathbf{C}}, \mathfrak{t}_{\mathbf{C}})$; of course $\Delta_{c} \subset \Delta$. If $\Delta^+$ is a system of positive roots for $\Delta$, we can consider the associated half-sum $\rho$ of positive roots, write $\rho_{c}$ for the half- sum of positive compact roots, and $\rho_{n}$ for the half-sum of positive noncompact roots, so that $\rho_{c}+\rho_{n} = \rho$. We will denote the $\Delta_c^+$-highest weight of $\mu$ by $\vec{\mu}$.

We now call in the  spin double cover $\text{Spin}(\pe)$ of the special orthogonal group $SO(\pe)$ associated with the inner product $B$ chosen in \S \ref{subsec:cadregeom}, and its spin representation $S$. Recall that the condition that $G$ and $K$ have equal ranks guarantees that $\dim(G/K)$ is an even integer, say $2q$;  the $\text{Spin}(\pe)$-module $S$ has dimension $2^{q}$ and splits into two irreducible $2^{q- 1}$-dimensional $\text{Spin}(\pe)$-submodules $S^+$ and $S^-$, with $\rho_{n}$ a weight of $S^+$ (for the action induced by the natural map from $\mathfrak{k}$ to $\text{Spin}(\pe)$). 

Suppose $V_{\mu^\flat}$ is the carrier space for an irreducible $\ka_{\C}$-module with highest weight $\mu^\flat~=~\vec{\mu}~-~\rho_{n}$. 
Then we can consider the tensor product $V_{\mu^\flat} \otimes S^{\pm}$, and although neither $V_{\mu^\flat}$ nor $S^{\pm}$ need be a $K$-module if $G$ is not simply connected, the action of $\ka$ on $V_{\mu^\flat} \otimes S^{\pm}$ does lift to $K$ (the half-integral $\rho_{n}$-shifts in the weights do compensate). So we can consider the equivariant bundle $\mathcal{E} = G \otimes_{K} \left( V_{\mu^{\flat}} \otimes S \right)$ over $G/K$, and its sub-bundles $\mathcal{E}^{\pm} = G \otimes_{K}~\left( V_{\mu^\flat}~\otimes~S^{\pm}~\right)$.

Now, the $G$-invariant metric that $G/K$ inherits from the Killing form of $\g$, together with the built-in spin structure of $\mathcal{E}$, make it possible to define a first-order differential operator $D$ acting on smooth sections of $\mathcal{E}$, the Dirac operator. We will need a few immediate consequences of its definition in the next subsection, so let us briefly spell it out, referring to \cite{Parthasarathy} for details.

We need the Clifford multiplication map $\mathbf{c}: \pe_{\C} \to \text{End}(S)$. Whenever $X$ is in $\pe$, the endomorphism $\mathbf{c}(X)$ sends $S^\pm$ to $S^{\mp}$; in addition, the element $X$ of $\g$ defines a left-invariant vector field on $G/K$, which yields a differential operator $X^\mathcal{E}$ acting (componentwise in the natural trivialization associated to the action of $G$ on $G/K$) on sections of $\mathcal{E}$. If $(X_{i})_{i = 1.. 2q}$ is an orthonormal basis of $\pe$, then the Dirac operator reads \begin{equation} \label{dirac} D s = \sum \limits_{i=1}^{2q} \mathbf{c}(X_{i}) X^\mathcal{E}_{i} s \end{equation}
for every section $s$ of $\mathcal{E}$. It splits as $D = D^+ + D^-$, with $D^\pm$ sending sections of $\mathcal{E}^{\pm}$ to sections of $\mathcal{E}^{\mp}$.

Write $\mathbf{H}$ for the space of smooth, square integrable sections of $\mathcal{E}$ which are annihilated by $D$. Since $D$ is an essentially self-adjoint elliptic operator, $\mathbf{H}$ is a closed subspace of the Hilbert space of square-integrable sections of $\mathcal{E}$. Because $D$ is $G$-invariant, $\mathbf{H}$ is invariant under the natural action of $G$ on sections of $\mathcal{E}$. 

\begin{thm}[\textbf{Parthasarathy, Atiyah \& Schmid}]
The Hilbert space $\mathbf{H}$ carries an irreducible unitary representation of $G$, whose equivalence class is that of $\pi = \mathbf{V}_{G}(\mu)$.
\end{thm}

\noindent We must point out that Atiyah and Schmid's proof shows that the solutions to the Dirac equation that are gathered in $\mathbf{H}$ do not explore the whole fibers, but are actually sections of a sub-bundle of $\mathcal{E}$ whose fiber, a $K$-module, is irreducible and of class $\mu$. Let $W$ denote the isotypical $K$-submodule of $V_{\mu^{\flat}} \otimes S^+$ for the highest weight $\vec{\mu} = \mu^{\flat} + \rho_{n}$; the $K$-module $W$ is irreducible. Let $\mathcal{W}$ denote the equivariant bundle on $G/K$ associated to $W$.

\begin{prop}[\textbf{Atiyah \& Schmid}] \label{rq_atiyah} {If a section of $\mathcal{E}$ is a square-integrable solution of the Dirac equation, then it is in fact a section of $\mathcal{W}$.}\end{prop}

\noindent Atiyah and Schmid comment on this result as follows:
\begin{center}
\begin{minipage}{0.95\textwidth}
\emph{We should remark that the arguments leading up to $[$the fact that the cokernel of the Dirac operator is zero$]$ are really curvature estimates, in algebraic disguise. The curvature properties of the bundles and of the manifold G/K force all square-integrable, harmonic spinors to take values in a certain sub-bundle of $[\mathcal{E}]$, namely the one that corresponds to the K-submodule of highest weight $[\mu^\flat + \rho_{n}]$ in $[V_{\mu^\flat} \otimes S^+]$.}
\end{minipage}
\end{center}


\subsubsection{Contraction to the lowest $K$-type}\label{subsec:contraction_discrete}

\noindent We are now in a position to observe how, for every real number $t \neq 0$, the above construction yields a $G_t$-invariant Dirac equation on $\pe$, and study the behaviour of solutions when the parameter $t$ goes to zero.

Fix a real number $t$ and bring in from \S \ref{subsec:cadregeom} the diffeomorphism $u_{t}: \pe \to G_{t}/K$, the action of $G_t$ and the metric $\eta_{t}$ on $\pe$.  Using the action of $K$ on $V_{\mu} \otimes S$, we use the results of the previous subsection to build a $G_t$-equivariant spinor bundle $\mathcal{E}_{t}$ over $G_t/K$. The vector bundle $u_{t}^\star \mathcal{E}_{t}$ over $\pe$ can be trivialized using the $G_t$-action: this yields a vector bundle isomorphism, say  $T_{t}$, between $u_{t}^\star \mathcal{E}_{t}$ and the trivial bundle $\pe \times (V_{\mu^\flat} \otimes S)$.

Using the metric $ \eta_{t}$ to build a $G_t$-equivariant Dirac operator acting on sections of $u_{t}^\star \mathcal{E}_t$, and trivializing using $T_{t}$, we obtain a first-order differential operator $D'_t$ acting on $\mathbf{C}^\infty(\pe, V_{\mu^\flat} \otimes S)$. Proposition \ref{rq_atiyah} makes it tempting to build from $D'_{t}$ a differential operator acting on
\begin{equation} \tag{A} \mathbf{E}= \mathbf{C}^\infty(\pe, W).\end{equation} We equip $\mathbf{E}$ with the usual Fréchet topology (see \cite{Treves}, Chapter 10: we shall spell out semi-norms defining the topology of $\mathbf{E}$ in the proof of Lemma \ref{lipschitz_seriediscrete}). We then define  the differential operator
\[ \Delta_t := \text{Proj}_{W} \circ\left[ (D'_t)^2 \big|_{\mathbf{C}^\infty(\pe, W)}\right] \]
where $\text{Proj}_{W}$ is the (orthogonal) isotypical projection $V_{\mu^\flat} \otimes S \rightarrow W$. 

Applying the above construction for each $t$ (including $t=0$), we obain a second-order $G_{0}$-invariant differential operator $\Delta_{0}$ on the flat space $(\pe, \eta_{0})$, as well as $G_{t}$-invariant differential operators $\Delta_{t}, t \neq0$, for the negatively-curved spaces $(\pe, \eta_{t})$. These operators do fit together:
\begin{lem} \label{continuite_delta} The family $(\Delta_{t})_{t \in [0,1]}$ of differential operators on $\mathbf{E}$ is weakly continuous: if $f$ lies in $\mathbf{E}$, then $(\Delta_{t}f )_{t \in [0,1]}$ is a continuous path in $\mathbf{E}$.\end{lem}

\begin{proof} Fix an orthonormal basis $(X_i)_{i=1..2q}$ of $\pe$ as in \eqref{dirac} and use the $2q$ corresponding Cartesian coordinates on $\pe$; given \eqref{dirac}, the operator $D'_{t}\big|_{\mathbf{C}^\infty(\pe, W)}$ reads
\[ D'_{t}\big|_{\mathbf{C}^\infty(\pe, W)} = \sum \limits_{i=1}^{2q} A_{t}^{i} \partial_{i}+ K_{t} \]  
where $A_{t}^{i}$, $i = 1... 2q$ and $K_{t}$ are continuous maps from $\pe$ to the space $\mathcal{L}(W, V_{\mu^{\flat}} \otimes S)$ of linear maps $W \to V_{\mu^{\flat}} \otimes S$. Given the definitions of $u_{t}$ and $T_{t}$, each of the vector fields $(T_{t} \circ u_{t})^\star X^{\mathcal{E}_{t}}_{i}$ defines a continuous endomorphism of $\mathcal{C}^\infty(\pe \times \R)$ when this space is equipped with its usual Fr\'echet topology; given \eqref{dirac} we deduce that multiplication by $(x, t) \mapsto A^{i}_{t}(x)$ and $(x, t) \mapsto K_{t}(x)$ define continuous endomorphisms of $\mathcal{C}^\infty(\pe \times \R)$; Lemma \ref{continuite_delta} follows. \end{proof}

\noindent We know from the previous subsection that the $\mathbf{L}^2$ kernel of $\Delta_{t}$, $t >0$, carries ``the'' discrete series representation of $G_{t}$ with lowest $K$-type $\mu$. For every $t>0$, we set 
\begin{equation} \tag{B} \mathbf{V}_t = \left\{ f \in \mathbf{C}^\infty(\pe, W) \ \ | \ \ \Delta_{t} f = 0 \quad \text{and}\quad  f \in \mathbf{L}^2(\eta_{t}, W) \right\}. \end{equation}
and call in the action of $G_t$ on $\mathbf{E}= \mathbf{C}^\infty(\pe, W)$ induced by the action $\cdot_t$ on $\pe$ (of \S \ref{subsec:cadregeom}) and the action $\mu$ of $K$ on $W$: if $(k,v)$ is an element of $K \times \pe$ $-$ interpreted as an element $\gamma$ of $G_t$ $-$ and $f$ is an element of $\mathbf{E}$, we set
\begin{equation} \label{operat_disc} \pi_t(k,v) f : x \mapsto \mu(k) \cdot f(  \gamma^{-1} \cdot_t x) \end{equation}
where the inverse for $\gamma$ is taken in $G_t$. From the $G_t$-invariance of $\Delta_t$  (and the fact that $D'_{t}$ and its square have the same $\mathbf{L}^2$ kernel) we know that 
 \begin{align} &\text{The subspace $\mathbf{V}_t$ of $\mathbf{E}$ is $\pi_t$-stable; the representation $(\mathbf{V}_t, \pi_t)$ of $G_t$ is irreducible.}\nonumber \\ &\text{It is a discrete series representation of $G_t$ with lowest $K$-type $\mu$.} \tag{B'} \end{align}
Recall from Lemma \ref{conjactions} that the dilation 
\[ z_{t} \ : \ x \mapsto  \frac{x}{t}\]
intertwines the actions of $G$ and $G_{t}$ on $\pe$. We can use $z_{t}$ to transform functions on $\pe$, setting
\begin{equation} \tag{C} \mathbf{C}_{t} f := x \mapsto f(t \cdot x) \end{equation}
for every $f$ in $\mathbf{E}$.
As a consequence of Lemma \ref{conjactions}(ii) and the definition of the Dirac operator in \eqref{dirac}, we get
\[ \mathbf{C}_{t}^{- 1} \Delta_{t} \mathbf{C}_{t} = t^{ 4} \cdot \Delta_{1}. \]
Together with the fact that $\mathbf{C}_{t} f$ is square-integrable with respect to $\eta_{t}$ as soon as $f$ is square- integrable with respect to $\eta_{1}$, this means that 
\begin{equation} \tag{Vect1} \text{If $f$ lies in $\mathbf{V}_{1}$, then for each $t>0$, the map $\mathbf{C}_{t} f$ lies in $\mathbf{V}_{t}$.}\end{equation}  
We remark that our contraction operators $(\mathbf{C}_t)_{t>0}$ satisfy the naturality property ($\text{Nat}_{\text{RIC}}$) of \S \ref{subsec:programme}, and are (almost) uniquely determined by that constraint:
\begin{lem} \label{naturalite_SD} Fix $t>0$.  \begin{enumerate}[(i)]
\item The contraction operator $\mathbf{C}_{t}$ intertwines the representations $(\mathbf{V}_1, \pi_1 \circ \varphi_1)$ and $(\mathbf{V}_t, \pi_t \circ \varphi_t)$ of $G$.
\item Our contraction operator $\mathbf{C}_{t}$ is the only map from $\mathbf{V}_1$ to $\mathbf{V}_t$ that satisfies (i) and preserves the value of functions at zero, in that  $(\mathbf{C}_{t} f)(0) = f(0)$ for every $f$ in $\mathbf{E}$. \end{enumerate}  \end{lem}
\noindent Part (i) is immediate from Lemma \ref{conjactions}, and part (ii) from the irreducibility of $\pi_1$ and $\pi_t$ and Schur's lemma.\qed\\

It should be mentioned that the idea of using of a change of variable as an intertwining  operator for contractions of representations is not new; see for example \cite{SBBM}.
\separateur
Let us turn to the convergence of the contraction process. Fix an $f$ in $\mathbf{V}_1$, and set $f_t = \mathbf{C}_t f$ for each $t>0$. We first make the trivial but crucial observation that in the usual Fréchet topology on $\mathbf{E}$,
\begin{equation} \tag{Vect2} \text{When $t$ goes to zero, $f_t$ goes to the constant function on $\pe$ with value $f(0) \in W$}. \end{equation}
In our current setting, more can be said. We know that $\mathbf{V}_1$ splits as a direct sum according to the decomposition of $(\pi_1)_{|K}$ into $K$-types, so that we can write $f$ as a Fourier series
\[ f = \sum \limits_{\lambda \in \widehat{K}} f_\lambda \]
where, for each $\lambda$ in $\widehat{K}$, $f_\lambda$ belongs to a finite-dimensional subspace $\mathbf{V}_1^\lambda$ of $\mathbf{V}_1$ on which $(\pi_1)|_{K}$ restricts as a direct sum of copies of $\lambda$. Of course a parallel decomposition holds for $\mathbf{V}_t$, $t \neq 0$, and $f_t$, too, has a Fourier series
\[ f_t = \sum \limits_{\lambda \in \widehat{K}} f_{t,\lambda}. \]
For every $\lambda$ in $\widehat{K}$, the dimension of the $\lambda$-isotypical subspace $\mathbf{V}_\lambda^t$ is independent of $t$, and the support of the above Fourier series does not depend on $t$.
\begin{lem} \label{Fourierseries} The Fourier component $f_{t, \lambda}$ is none other than $\mathbf{C}_t f_\lambda$. \end{lem}
\begin{proof} For each $\lambda \in \widehat{K}$, we define an endomorphism $P_\lambda$ of $\mathbf{C}^\infty(\pe, W)$, as follows:
\[  \text{For all $f$ in $\mathbf{C}^\infty(\pe, W)$, }  P_\lambda(f) \text{ is the map }  \left[ x \mapsto \int_K {\xi^\star_\lambda}(k) \  \mu(k) \cdot f\left(\text{Ad}(k^{-1}) x\right) dk \right]\]
where $\xi_\lambda$ is the global character of $\lambda$ $-$ a continuous function from $K$ to $\C$ $-$ and the star is complex conjugation. Now, if $f$ is an element of $\mathbf{V}_t$, we can view $K$ as a subgroup of $G_t$: since the action of $K$ on $\pe$ induced from the $G_t$-action $\cdot_t$ is none other than the adjoint action, the formula for $P_\lambda f$ is exactly that which defines the isotypical projection from $\mathbf{V}_t$ to $\mathbf{V}^\lambda_t$. From Lemma \ref{naturalite_SD}, we deduce that  $P_\lambda$  commutes with $\mathbf{C}_t$. \end{proof}

\begin{lem} \label{Fourier2} If $\lambda \in \widehat{K}$ is different from the lowest $K$-type $\mu$, then $f_{\lambda}(0) = 0$. \end{lem}

\begin{proof} The origin of $\pe$ is a fixed point for the action of $K$ on $\pe$; so 
\begin{equation} \label{schur} f_\lambda(0) = (P_\lambda f)(0) =  \int_K {\xi^\star_\lambda}(k) \  \mu(k) \cdot f(0) dk .\end{equation}
Recall that $f(0)$ is in $W$, which is an irreducible $K$-module of class $\mu$: now,~\eqref{schur} is the formula for the orthogonal projection of $f(0)$ onto the isotypical component of $W$ corresponding to $\lambda \in \widehat{K}$: this is zero whenever $\lambda \neq \mu$.
 \end{proof}
Lemmas \ref{Fourierseries} and \ref{Fourier2}, together with property (Vect2) above, prove that each Fourier component of $f$ goes to zero as the contraction is performed, except for that which corresponds to the lowest $K$-type. We can thus refine  (Vect1)-(Vect2) into the following statement. Set
\[ \mathbf{V}_0 = \left\{ \text{constant functions on $\pe$ with values in $W$}\right\}.\]
\begin{prop} \leavevmode 
\begin{enumerate}[(i)]
\item For every $f \in \mathbf{E}$, there is a limit $f_0$ to $\mathbf{C}_{t} f$ as $t$ goes to zero (the convergence is for the topology of $\mathbf{E}$).
\item The limit $f_0$ is the element of $\mathbf{V}_{0}$ with value $f(0)$. 
\item Suppose $f$ is in $\mathbf{V}_1$; write $f_{min}$ for the orthogonal projection of $f$ onto  the isotypical subspace $\mathbf{V}_1^\mu$ of $\mathbf{V}_1$ for the lowest $K$-type $\mu$. Then $\mathbf{C}_{t}(f-f_{min})$ goes to zero as $t$ goes to zero.
\item The correspondence $f \to f_0$  induces a $K$-equivariant isomorphism between $\mathbf{V}_1^\mu$ and $\mathbf{V}_0$. 
\end{enumerate}
\end{prop}
We now observe the asymptotic behavior of the operators $\pi_t(k,v)$, when  $(k,v)$ is a fixed element of $K \times \pe$ $-$ interpreted for each $t$ as an element of $G_t$. Fix an element $f$ of $\mathbf{V}_1$ and consider the associated family $(f_t)_{t>0}$. To study $\pi_t(k,v)f_t$, remark that  
\begin{equation} \label{convop} \pi_{t}(k,v) f_{t} = \pi_{t}(k,v)(f_{t} - f_{0}) + \pi_{t}(k,v) f_{0}, \end{equation}
and insert the following observation:

\begin{lem} \label{lipschitz_seriediscrete} Fix $(k,v)$ in $K \times \pe$. The topology of $\mathbf{E}$ can be defined by a distance with respect to which each of the $\pi_{t}(k,v)$, $t \in [0,1]$, is $1$- Lipschitz. \end{lem}
\begin{proof} Whenever $A \subset \pe$ is compact, the subset $\Pi(A) = \left\{ (k,v) \cdot_{t} A \ | \ t \in [0,1] \right\}$ is compact too (see Lemma \ref{convactions}). So there is an increasing family $(A_{n})_{n \in \N}$ of compact subsets of $\pe$ such that $\Pi(A_{n}) \subset A_{n+1}$ for each $n$  and $\cup_{n \in \N} A_{n} = \pe$.  For each $(n,K)$ in $\N^2$ and $f$ in $\mathbf{E}$, set $\norme{f}_{n,K} = \sup \limits_{|\alpha|_1\leq K} \sup \limits_{x \in A_n} \norme{\partial^{\alpha}f(x)}_W$, where $\alpha \in \N^K$ are multi-indices and $|\cdot|_1$ is the $\ell^1$ height function on $\N^K$. For all $f_1$, $f_2$ in $\mathbf{E}$, we have $\norme{\pi_{t} f_1 - \pi_{t} f_2}_{n,K} \leq \norme{f_1 - f_2}_{n+1,K}$. A distance whose associated topology is that of $\mathbf{E}$ is $d(f,f') = \sum \limits_{n} \frac{\norme{f - f'}_{n,K}}{2^{n+K}(1+\norme{f - f'}_{n,K})}$; the distance $d/2$ has the desired property. \end{proof}
Taking up \eqref{convop}, we can deduce that $\pi_{t}(k,v)(f_{t} - f_{0})$ goes to zero as $t$ goes to zero from the fact that $f_t -f_0$ does and $\pi_{t}(k,v)$ does not increase the distance of Lemma \ref{lipschitz_seriediscrete}. As for the second term $ \pi_{t}(k,v) f_{0}$ in \eqref{convop}, it is equal to $\mu(k) f_{0}$ independently of $t$. In conclusion:
\begin{equation} \tag{Op1} \pi_{t}(k,v) \cdot f_{t} \quad \text{goes to \quad $\mu(k) \cdot f_0$\quad as $t$ goes to zero.} \end{equation}
Since property (Op2) of \S \ref{subsec:programme} is immediate when we set $\pi_0(k,v)=\mu(k)$, we have proved:
\begin{thm} \label{contraction_discrete} The quadruple $\left(\mathbf{E}, (\mathbf{V}_t)_{t>0}, (\pi_t)_{t>0}, (\mathbf{C}_t)_{t>0}\right)$ describes a contraction from the discrete representation $\pi$ to its lowest $K$-type $\mu$.  \end{thm}
\separateur
We close our discussion of the discrete series with a side remark. It may seem annoying, in view of Lemma \ref{continuite_delta}, to find that the space $\mathbf{V}_t$ of $\mathbf{L}^2$-solutions to the $G_t$-invariant Dirac equation $\Delta_t f = 0$ should converge in such a natural way as indicated by Theorem \ref{contraction_discrete} to the $G_0$-module $\mathbf{V}_0$, but to note that $\mathbf{V}_0$ is not the $\mathbf{L}^2$ kernel of $\Delta_{0}$, since that kernel is zero. A simple way to diminish the annoyance is to consider \emph{extended kernels}, setting, for each $t$ in $\R$ (including $t=0$),
\begin{equation*} \tilde{\mathbf{V}}_{t} = \left\{ f \in \mathbf{C}^\infty(\pe, W) \ | \ \Delta_{t} f = 0 \text{ and there is a constant } c \in W \text{ such that } f + c \in \mathbf{L}^2(\eta_{t}, W) \right\}. \end{equation*}
Because the riemannian metric $\eta_t$ has infinite volume, when $f$ is in $\tilde{\mathbf{V}}_{t}$, there can be only one constant $c$ such that $f+c$ is square- integrable.
\begin{lem} \label{extended}
\begin{enumerate}[(i)] 
\item For $t \neq 0$, the extended kernel  $\tilde{\mathbf{V}}_{t}$ is none other than the $\mathbf{L}^2$ kernel ${\mathbf{V}}_{t}$. 
\item For $t=0$, the extended kernel $\mathbf{H}_{0}$ is the space of constant $W$-valued functions on $\pe$.
\end{enumerate}  \end{lem}
\begin{proof} 
Let us come back to $G/K$ and the Dirac operator $D$ defined in \S \ref{subsec:dirac}. Parthasarathy \cite{Parthasarathy} expresses $D^2$ in terms of the Casimir operator $\Omega$ acting on sections of $\mathcal{E}$: there is a scalar $\sigma$ such that
\[ D^2 := D^- D^+ = - \Omega + \sigma. \]
Suppose a $G$-invariant trivialization of $\mathcal{E}$ is chosen, so that $D^2$ is viewed as acting on functions from $G/K$ to $V_{\mu^\flat} \otimes S$, and suppose $D^2 g = 0$, where $g$ reads $f+C$ with $f \in \mathbf{L}^2(G/K, V_{\mu^\flat} \otimes S)$ and $C$ a constant in $V_{\mu^\flat} \otimes S$. Then 
\begin{equation}\label{Parta} \Omega f = \sigma f + \sigma C. \end{equation}
What we need to show is that this can only happen if $C$ is zero. To that end, we use Helgason's Fourier transform for functions on $G/K$ (see \cite{HelgasonGASS}, Chapter 3; we shall meet Helgason's analogue of plane waves in \S \ref{subsec:contraction_helgason}). The Fourier transform of a smooth function with compact support on $G/K$ is a function on $\mathfrak{a}^\star \times K/M$; the transform defines a linear isometry $\mathbf{F}$ between $\mathbf{L}^2(G/K)$ and $\mathbf{L}^2(\mathfrak{a}^\star \times K/M)$ for a suitable measure on $\mathfrak{a}^\star \times K/M$. There is a notion of tempered distributions on $G/K$ and $\mathfrak{a}^\star \times K/M$; when $f$ is a smooth and square-integrable function, both $f$ and $\Omega f$ are tempered distributions. The equality $ \Omega f = \sigma f + \sigma C$ therefore is turned into an equality of tempered distributions on $\mathfrak{a}^\star \times K/M$:
\[ \mathbf{F}(\Omega f) - \sigma \mathbf{F}(f) = \sigma C  \delta_{(0, 1M)} \]
where $\delta_{(0, 1M)}$ is the Dirac distribution at the point $(0, 1M)$. Given the transformation properties of $\mathbf{F}$ with respect to the $G$-invariant differential operators, $\mathbf{F}(\Omega f)$ is actually the product of $\mathbf{F}(f)$ $-$ an element of $\mathbf{L}^2(\mathfrak{a}^\star \times K/M)$ $-$ with a smooth function on $\mathfrak{a}^\star$. Thus, if $f$ is a smooth, square-integrable solution of \eqref{Parta}, then $\sigma C \delta_{0}$ is the product of an element in $\mathbf{L}^2(\mathfrak{a}^\star \times K/M)$ with a smooth function on the same space. That can only happen if $C$ is zero. \end{proof}


\subsection{On the principal series and generic spherical representations}\label{sec:serie_principale}

{\noindent In this section, we consider the representations associated with Mackey parameters of the form $\delta = (\chi, \mu)$ where $\chi$ is a \emph{regular} element of $\mathfrak{a}^\star$. The corresponding representations of $G$ are (minimal) principal series representations, and several existing results can be understood as giving flesh to Mackey's analogy at the level of carrier spaces. 

There are several very well-known function spaces carrying a representation of $G$ with class $\mathbf{M}(\delta)$ $-$ see for instance \S VII.1. in \cite{KnappPrincetonBook}. We will use two of these standard realizations: in the first, the functions are defined on $K/M$ (where $M = Z_K(\mathfrak{a})$ is the compact group of Remark \ref{remarque_regulier}), which has the same meaning in $G$ and $G_{0}$;  in the second, to be considered here only when $\mu$ is trivial, the functions are defined on $G/K$, or equivalently on $\pe$, and the geometrical setting in \S \ref{subsec:cadregeom} will prove helpful.  }


\subsubsection{A contraction of (minimal) principal series representations in the compact picture}
\label{subsec:contraction_compact}

\noindent We first assume $\delta = (\chi, \mu)$ to be an \emph{arbitrary} Mackey parameter and recall the ``compact picture'' realization for the $G$-representation $\mathbf{M}(\delta)$ from \eqref{inductionG}. 

To simplify the notation, we write $P = MAN$ for the cuspidal parabolic subgroup $P_\chi$ to be induced from. Suppose $V_{\sigma}$ is the carrier space for a tempered irreducible $M$-module of class $\sigma = \mathbf{V}_{M}(\mu)$. Fix  an $M$-invariant inner product on $V_{\sigma}$, and consider 
\begin{equation} \label{hilbcomp} \mathbf{H}^{\text{comp}}_{\sigma} = \{ f \in \mathbf{L}^2(K; V_{\sigma}) \ / \ f(km) = \sigma(m)^{- 1} f(k), \forall (k, m) \in K \times M \}. \end{equation}
To endow $\mathbf{H}^{\text{comp}}_{\sigma}$ with a $G$-action, we call in the Iwasawa projections $\kappa$, $\mathbf{m}$, $\mathbf{a}$, $\nu$ sending an element of $G$ to the unique quadruple 
\begin{align} \label{iwa} & (\kappa(g), \mathbf{m}(g),\mathbf{a}(g), \nu(g)) \in K \times  \exp_{G}\left( \mathfrak{m}\cap \pe\right) \times \mathfrak{a} \times N \quad \text{such that} \nonumber \\  & g = \kappa(g) \mathbf{m}(g) \exp_{G}(\mathbf{a}(g)) \nu(g). \end{align}  We recall that if $P$ is minimal, the map $\mathbf{m}$ is trivial. For every $g$ in $G$, an endomorphism of $\mathbf{H}^{\text{comp}}_{\sigma}$ is defined by setting
\begin{equation}  \label{actcomp} \pi^{\text{comp}}_{\chi, \mu}(g) = f \mapsto  \left[ k \mapsto  e^{\langle - i\chi - \rho,\ \mathbf{a}(g^{- 1} k) \rangle} \sigma(\mathbf{m}(g^{- 1}k))^{- 1} f\left(\kappa(g^{- 1} k)\right) \right] \end{equation}
where $\rho$ is the half-sum of those roots of $(\mathfrak{g}, \mathfrak{a})$ that are positive in the ordering used to define $N$.
The Hilbert space \eqref{hilbcomp} does not depend on $\chi$, but the $G$-action \eqref{actcomp} does. 
\separateur
We now assume that $\chi$ is a \emph{regular} element of $\mathfrak{a}^\star$, so that $P_\chi$ is a \emph{minimal} parabolic subgroup of $G$, and describe a Fréchet contraction of $\pi=\mathbf{M}(\chi, \mu)$ onto $\mathbf{M}_0(\chi, \mu)$.

Recall from Remark \ref{remarque_regulier} that in the decomposition $P_\chi = M_\chi A_\chi N_\chi$, the subgroup $M_\chi=M = Z_K(\mathfrak{a})$ is contained in $K$, and that $A_\chi = A$. 

Fix $t>0$, $\lambda$ in $\mathfrak{a}^\star$ and $\sigma$ in $\widehat{M}$. Then we can view $(\lambda, \sigma)$ as a Mackey parameter for $G_t$ and run through the constructions of \S \ref{subsec:correspondance}  for $G_t$: this leads to defining a parabolic subgroup $P_{\lambda,t} = M_t A_t N_t$ of $G_t$, with $M_t=M$, and to considering the principal series representation $\pi^{t, \,\text{comp}}_{\lambda, \sigma}$ of $G_t$ acting on
\begin{equation*} \mathbf{H}^{\text{comp}}_{\sigma} \quad \text{from \eqref{hilbcomp}} \end{equation*}
via
\begin{equation} \label{opGt} \pi^{t, \,\text{comp}}_{\lambda, \sigma}(\gamma): f \mapsto  \left[ k \mapsto  \exp{\langle - i\lambda -  \rho_t, \mathbf{a}_{t}(\gamma^{- 1} k) \rangle} f\left(\kappa_{t}(\gamma^{- 1} k)\right) \right]\end{equation}
where $\rho_t$ is the half-sum of positive roots for $(\g_t, \mathfrak{a})$ associated with the ordering used to define $N_t$, and $\kappa_{t}:G_t \to K$, $\mathbf{a}_{t}: G_t \to \mathfrak{a}$ are the Iwasawa projections. 

 At this point, given our Mackey parameter $(\chi, \mu)$, it may seem natural to fix an element $(k,v)$ in $K \times \pe$, view it as an element of $G_t$ for each $t$ and try to study the convergence of $ \pi^{t, \,\text{comp}}_{\chi, \mu}(k,v)$ from \eqref{opGt} as $t$ goes to zero. We shall pursue this, but we must point out that this idea surreptitiously brings in the renormalization of continuous parameters from property ($\text{Nat}_{\text{gen}}$) in \S \ref{subsec:programme}. To see why, define a representation of $G$ as the composition
\[  \pi^{t, \,\text{comp}}_{\chi, \mu}(\gamma) \circ \varphi_t^{-1} \ : \  G \overset{\varphi_{t}^{- 1}}{\longrightarrow} G_{t}  \overset{\pi^{t, \,\text{comp}}_{\chi, \mu}}{\longrightarrow} \text{End}(\mathbf{H}^{\text{comp}}_{\sigma}). \]

\begin{prop} \label{renorm2} For each $t>0$, $ \pi^{t, \,\text{\emph{comp}}}_{\chi, \mu}(\gamma) \circ \varphi_t^{-1}$  is equal to $\displaystyle \pi^{\text{\emph{comp}}}_{\chi/t, \mu}$. \end{prop}

To prove Proposition \ref{renorm2}, we only have to understand what happens to the half-sum of positive roots when we go from $G$ to $G_{t}$, and to make the relationship between the Iwasawa decompositions in both groups clear. Here are statements that apply to any parabolic subgroup. 

\begin{lem} \label{racines_gt}Suppose $P = MAN$ is an arbitrary parabolic subgroup of $G$ with $A$ contained in $\exp_G(\pe)$. If $\alpha \in \mathfrak{a}^\star$ is a root of $(\g, \mathfrak{a})$, then $t  \alpha$ is a root of $(\g_{t}, \mathfrak{a})$. \end{lem}

\begin{proof} When $\alpha$ is a root of $(\g, \mathfrak{a})$, there is a nonzero $X \in \g$ such that $[X, H] = \alpha(H) X$ for each $H \in \mathfrak{a}$. Split $X$ as $X_{e} + X_{h}$, with $X_{e} \in \mathfrak{k}$ and $X_{h} \in \pe$, according to the Cartan decomposition. Then
\begin{equation}\label{jordan} [X_{e}, H] - \alpha(H) X_{h} = - \left(  [X_{h}, H] - \alpha(H) X_{e}\right).\end{equation}
The left-hand side of~\eqref{jordan} is in $\pe$ and the right-hand side is in $\ka$, so both are zero. 

Now, the group morphism $\varphi_t: G_t \to G$ from \S \ref{subsec:defo} induces a Lie algebra isomorphism $\phi_t: \g_t \to \g$. The isomorphism $\phi_{t}^{- 1}$ sends $X$ to $X^t =  X_{e} + \frac{1}{t}X_{h} \in \g_{t}$, and for each $H \in \mathfrak{a}$, 
\begin{equation}\label{jordan2}  [X^t, H]_{\g_{t}} = [X_{e}, H]_{\g_{t}} + \frac{1}{t} [X_{h}, H]_{\g_{t}} =   [X_{e}, H]_{\g} + t [X_{h}, H]_{\g} \end{equation}
(for the last equality, recall that if $U,V$ are in $\g$, the bracket $[U,V]_{\g_{t}}$ is defined as $\phi^{- 1}_{t}[\phi_{t} U, \phi_{t} V]_{\g}$, so that $[U,V]_{\g_{t}} = [U,V]_{\g_{}}$ when $U$ lies in $\ka$ and $V$ lies in $\pe$, whereas $[U,V]_{\g_{t}} =t^2 [U,V]_{\g_{}} $ when they both lie in $\pe$).

The right-hand side of \eqref{jordan2} is  $\alpha(H) X_{h} + t \cdot \alpha(H) X_{e} = t \cdot \alpha(H) X^t$, so $X^t$ is in the $(\mathfrak{g}_{t}, \mathfrak{a})$ root space for $t \cdot \alpha$, whence Lemma \ref{racines_gt}.
\end{proof}
\noindent The proof shows that the root space for $t \alpha$ is the image of $\g_{\alpha}$ under $\phi_{t}^{- 1}$; we spell out an immediate consequence. 
\begin{lem} \label{iwa_gt} Suppose $P=MAN$ is a parabolic subgroup of $G$. Then $P_t:= \varphi_t^{-1}(P)$ is a parabolic subgroup of $G_t$, with Iwasawa decomposition $M_tA_tN_t=(\varphi_t^{-1}M) A (\varphi_t^{-1}N)$. The Isawawa projections \label{iwa} for $G$ and $G_t$ are related as follows: for every $g$ in $G$,
\begin{align*}  \kappa_{t}(\varphi_{t}^{- 1} g) &= \kappa(g); \\
 \mathbf{m}_{t}(\varphi_{t}^{- 1} g) &= \mathbf{m}(g);\\
\mathbf{a}_{t}(\varphi_{t}^{- 1} g)& = \frac{\mathbf{a}(g)}{t}. \end{align*} \end{lem} 
We now return to the proof of Proposition \ref{renorm2}, assuming again that $\chi$ is regular. Inspecting the definitions of \S \ref{subsec:correspondance}, we know from Lemma \ref{iwa_gt} that the parabolic subgroup $P_{\chi,t}$ of $G_t$ used in \eqref{opGt} is $\varphi_t^{-1}(P_\chi) = MA \varphi_t^{-1}(N_\chi)$.  From Lemma \ref{racines_gt} we know that $\rho_t$ in \eqref{opGt} is $t \rho$. We can now conclude that for every $g$ in $G$,
\begin{align*} \pi^{t, \,\text{comp}}_{\chi, \mu}(\varphi_{t}^{- 1}(g)) & = f \mapsto  \left[ k \mapsto  \exp{\langle - i\frac{\chi}{t} - \rho, t \cdot \mathbf{a}_{t}( \varphi_{t}^{- 1}\left[g^{- 1} k\right]) \rangle} f\left(\kappa_{t}(\varphi_{t}^{- 1}\left[g^{- 1} k\right])\right) \right] \\
& = f \mapsto \left[ k \mapsto  \exp{\langle - i\frac{\chi}{t} - \rho, \mathbf{a}( g^{- 1} k) \rangle} f\left(\kappa(g^{- 1} k)\right) \right] \Fin{\text{(using Lemma \ref{iwa_gt})}}\\
& = \pi^{\text{comp}}_{\frac{\chi}{t}, \mu}(g). \hspace{10cm}  \end{align*} 
This concludes the proof of Proposition \ref{renorm2}.\qed

\begin{rem} \label{rem_identification} There is an implicit choice behind the interpretation of $(\chi, \mu)$ as a Mackey datum for $G_t$ which leads to Proposition \ref{renorm2}. In the present context, we have $A_t = A_1$ for all $t>0$; as one referee pointed out, our way of viewing $e^{i\chi}$ as a character of $A_t$ amounts to using the identification between the Pontryagin duals $\widehat{A_t}$ and $\widehat{A}$ induced by the identity map $\mathfrak{a}_t \to \mathfrak{a}$. There are other ways to identify $\widehat{A_t}$ and $\widehat{A}$; but they do not lead to the convergence theorems to be proven below. The renormalization of continuous parameters in property ($\text{Nat}_{\text{gen}}$) of \S \ref{subsec:programme} reflects~the~identification.
\end{rem}

\separateur 
With Proposition \ref{renorm2} and its proof in hand, discussing the contraction from $G$ to $G_{0}$ becomes particularly simple. Set
\begin{equation} \tag{A} \mathbf{E}= \mathbf{H}^{\text{comp}}_{\sigma} \quad \text{from \eqref{hilbcomp}} \end{equation}
and for every $t>0$, 
\begin{align} \tag{B} \mathbf{V}_t &= \mathbf{E} \\
\tag{B'} \pi_t &=  \pi^{t, comp}_{\chi, \mu} \quad \text{from \eqref{opGt}}\\
\tag{C} \mathbf{C}_t &= \text{id}_{\mathbf{E}}.
 \end{align}
Fix $(k,v)$ in $K \times \pe$. For every $t>0$, the endomorphism ${\pi}_{t}(k,v)$ of $\mathbf{E}$ is an operator for a principal series representation of $G$ with continuous parameter  $\frac{\chi}{t}$; but as $t$ goes to zero it gets closer and closer to an operator for the representation of $G_{0}$ with Mackey parameter $(\chi, \mu)$: 

\begin{prop} \label{dooleyrice} Fix $(k,v)$ in $K \times \pe$ and $f$ in $\mathbf{E}$. There is is a limit to ${\pi}_{t}(k,v)f$ as $t$ goes to zero, given by the map $\pi_{0}(k,v)f$ in \eqref{actionG0}. The convergence holds both in $\mathbf{L}^2(K, V_\mu)$ and in the sense of uniform convergence in $\mathbf{C}(K, V_\mu)$. \end{prop}
Recall that when $k$ is in $K$ and $v$ in $\pe$,
\begin{equation}\label{interpGt} \text{ the element $\exp_{G_t}(v)k$ of $G_t$ is just $(k,v)$ in the underlying set $K \times \pe$.}\end{equation}
As a consequence, we can rewrite \eqref{opGt} as
\[ \pi^{t, \,\text{comp}}_{\chi, \mu}(k,v) = f \mapsto  \left[ u \mapsto  \exp{\langle - i\chi - t \rho, \mathbf{a}_{t}\left(  (k^{- 1} \cdot_{t} \exp_{G_{t}}(- v) \cdot_{t} u) \right) \rangle} f\left(\kappa_{t}( k^{- 1}  \exp_{G_{t}}(- v) u)\right) \right]. \]
where the symbols $\cdot_t$ indicate products in $G_{t}$. On the other hand, the operator $\pi_{0}(k,v)$ from \S \ref{subsec:correspondance}  reads
\[ \pi_{0}(k,v) = f \mapsto  \left[ u \mapsto  \exp{\langle i\chi,  \ad(u^{- 1}) v  \rangle} f\left( k^{- 1} u \right) \right]. \]
To make the two look more similar, notice that 
\begin{align*} \left[\pi^{t, \,\text{comp}}_{\chi, \mu}(k,v) f \right](u) & =  e^{\langle - i\chi - t \rho, \ \mathbf{a}_{t}\left(  (k^{- 1} u)\cdot_{t} \exp_{G_{t}}(- \ad(u^{- 1}) v)  \right) \rangle} f\left(\kappa_{t}( k^{- 1} u \exp_{G_{t}}(- \ad(u^{\ 1}) v)\right) \\ & =   e^{\langle - i\chi - t \rho, \ \mathbf{a}_{t}\left( \exp_{G_{t}}(- \ad(u^{- 1}) v)  \right) \rangle} f\left( (k^{- 1} u) \cdot \kappa_{t}(\exp_{G_{t}}(- \ad(u^{- 1}) v)) \right).  \end{align*}
 We just need to see how the Iwasawa projection parts behave as $t$ goes to zero. Let us introduce, for every $t>0$, the Iwasawa maps 
\begin{align} \label{iwa_It}\mathfrak{K}_t: \pe &\to K \quad & \text{and} \quad\quad\quad\quad\quad & \mathfrak{I}_t: \pe  \to\mathfrak{a} \\ v & \mapsto \kappa_t\left( \exp_{G_t}(v)\right) & & v  \mapsto \mathbf{a}_t\left( \exp_{G_t}(v)\right). \nonumber \end{align} 
The Iwasawa projection $\mathfrak{I}_{t}: \pe \to \mathfrak{a}$ is nonlinear, but as $t$ goes to zero it gets closer and closer to a linear projection:
\begin{lem} \label{conv_iwa} As $t$ goes to zero, 
\begin{itemize}
\item[$\bullet$] $\mathfrak{I}_{t}$ admits as a limit (in the sense of uniform convergence on compact subsets of $\pe$ of every derivative) the orthogonal projection from $\pe$ to $\mathfrak{a}$, 
\item[$\bullet$] $\mathfrak{K}_{t}$ tends to the constant function on $\pe$ with value $\mathbf{1}_{K}$. 
\end{itemize}\end{lem}

\begin{proof} $\bullet$ We check that $\mathfrak{I}_{t}$ is none other than $v \mapsto \frac{1}{t} \mathfrak{I}(tv)$, where $\mathfrak{I}: \pe \to \mathfrak{a}$ is the projection attached to $G$. Since $\varphi_{t}$ is a group morphism from $G_{t}$ to $G$, the definition of group exponentials does imply that $\exp_{G}(t v) = \exp_{G}(d\varphi_{t}(1) v) = \varphi_{t} \left( \exp_{G_{t}} v \right)$. Let us write $\exp_{G_{t}} v = k e^{\mathfrak{I}_{t}(v)} n_{t}$ with $k \in K$ and $n_{t} \in N_{t}$, then  $ \varphi_{t} \left( \exp_{G_{t}} v \right) = k e^{t\mathfrak{I}_{t}(v)} n$, where $n = \varphi_{t}(n_{t})$ lies in $N$. So we know that 
\begin{equation}~\label{Iwa} \exp_{G}(t v) = k e^{t\mathfrak{I}_{t}(v)} n\end{equation}
and thus that $\mathfrak{I}(tv) = t \mathfrak{I}_{t}(v)$, as announced. 

But then as $t$ goes to zero, the limit of $\mathfrak{I}_{t}(v)$ is the value at $v$ of the derivative $d\mathfrak{I}(0)$. This does yield the orthogonal projection of $v$ on $\mathfrak{a}$: although the Iwasawa decomposition of $\g$ is not an orthogonal direct sum because $\ka$ and $\mathfrak{n}$ are not orthogonal to each other, they are both orthogonal to $\mathfrak{a}$ with respect to the Killing form of $\g$, so the direct sum $\ka \oplus \mathfrak{n}$ is the orthogonal of $\mathfrak{a}$.  

$\bullet$ As for $\mathfrak{K}_{t}$, from \eqref{Iwa} we see that $\mathfrak{K}_{t}(v) = \kappa_{t}\left(exp_{G_{t}} v \right) = \kappa \left( \exp_{G}(tv) \right)$, and this does go to the identity uniformly on compact subsets as $t$ goes to zero.
\end{proof}
\noindent Both parts of  Proposition \ref{dooleyrice} follow immediately.
 
 \begin{rem} Proposition \ref{dooleyrice} can be viewed as a reformulation of Theorem 1 in Dooley and Rice \cite{DooleyRice}. Our reason for including it here is that in our treatment of the general case in \S \ref{sec:contraction_generale}, we will take up the strategy of the proof and use the lemmas here proven. The interplay with Helgason's waves (in the next subsection) may also throw some light on the phenomenon, for instance on the necessity of renormalizing continuous parameters.\end{rem}

\subsubsection{A contraction of generic spherical representations using Helgason's waves}
\label{subsec:contraction_helgason}

\paragraph[Helgason's waves and the spherical principal series]{Helgason's waves and picture for the spherical principal series.}

We first recall how the compact picture of the previous subsection is related to the usual "induced picture", for which the Hilbert space is 
\begin{equation}\small  \label{hilbind} \mathbf{H}^{\text{ind}}_{\delta} = \left\{ f: G \rightarrow V_{\sigma} \ \ \big| \quad \ \begin{aligned} f(gme^{H}n)& = e^{\langle - i\chi - \rho, H\rangle} \sigma(m)^{- 1}f(g) \ \text{for all} \ (g, me^{H}n) \in G \times P \\ \text{and} \ & f\big|_{K} \in \mathbf{L}^2(K; V_{\sigma}) \end{aligned} \right\},\normalsize  \end{equation}
the inner product is the $\mathbf{L}^2$ scalar product between restrictions to $K$, and the $G$-action is given by $\pi^{\text{ind}}_{\delta}(g) f = \left[ x \mapsto f(g^{- 1} x) \right]$ for $(g,f)$ in $G \times \mathbf{H}^{\text{ind}}_{\delta}$. Because of the $P$-equivariance condition in \eqref{hilbind}, every element of $\mathbf{H}^{\text{ind}}_{\delta}$ is completely determined by its restriction to $K$, so  restriction to $K$ induces an isometry (say $\mathcal{I}$) between $\mathbf{H}^{\text{ind}}_{\delta}$ and $\mathbf{H}^{\text{comp}}_{\sigma}$. The definition of  $\pi^{\text{comp}}_{\chi, \mu}$ in \eqref{actcomp} is just what is needed to make $\mathcal{I}$ an intertwining operator. \\

\noindent We now consider a Mackey parameter $\delta = (\chi, \mu)$ and assume that 
\begin{itemize}
\item[$\bullet$] $\chi$ is regular and lies in $\mathfrak{a}^\star$, so that $K_\chi$ is equal to $M=Z_K(\mathfrak{a})$ ;

\item[$\bullet$] $\mu$ is the trivial representation of $M$. \\
\end{itemize}
There is then a distinguished element in $\mathbf{H}^{\text{comp}}_{\sigma}$: the constant function on $K$ with value one. Under the isometry $\mathcal{I}$, it corresponds to the function 
\[ \bar{e}_{\chi, 1} = ke^{H}n \mapsto e^{\langle - i\chi - \rho, H\rangle} \]
in $\mathbf{H}^{\text{ind}}_{\delta}$; this in turn defines a function on $G/K$ if we set $\tilde{e}_{\chi, 1}(gK) = \bar{e}_{\chi, 1}(g^{- 1})$, and a function on $\pe$ if we set $e_{\lambda, 1}(v) = \tilde{e}_{\chi, 1}(\exp_{G}(v)K)$. Figure \ref{onde} shows a plot of $e_{\lambda, 1}$ when $G$ is $SL(2,\R)$.

\begin{figure}[h]
\begin{center}
\includegraphics[width=0.29\textwidth]{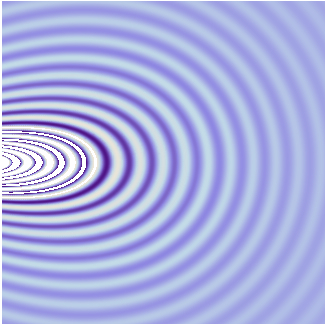} 
\caption{ Plot of the real part of the Helgason wave $e_{30, 1}$. We used the mapping from $\R^2$ to the unit disk provided by the Cartan decomposition, and the explicit formulae availiable on the unit disk: see \cite{HelgasonGGA}, chapter 0. The $x$- and $y$- range is [-1.5, 1.5] (this region is chosen so that the modulus varies clearly but within a displayable range, and the choice of $\lambda$ is to have enough waviness in the region).}
\label{onde}
\end{center}
\end{figure}

 Now set $e_{\chi, b}(v) = e_{\chi, 1}(b^{- 1}v)$ for $b$ in $K/M$ and $v$ in $\pe$ (the notation comes from \cite{HelgasonGASS}, \S III.1: for $\tilde{b}$ in $K$, the map $v \mapsto e_{\chi,1}(\tilde{b}^{-1}v)$ depends only on the class of $\tilde{b}$ in $K/M$). Equip $K/M$ with the measure inherited with the Haar measure of $K$. Then the ``Poisson transform''
\begin{align*} \mathbf{L}^2(K/M) = \mathbf{H}^{\text{comp}}_{\sigma} & \rightarrow \mathbf{C}^{\infty}(\pe) \\
 F & \mapsto \int_{K/M} e_{\chi, b} F(b)db 
\end{align*}
intertwines $\pi^{\text{comp}}_{\chi,1}$ with the quasi-regular action of $G$ on the Fréchet space
\begin{equation} \tag{A} \mathbf{E}= \mathbf{C}^{\infty}(\pe)\end{equation} associated with the action $(g,x) \mapsto g\cdot x$ defined below \eqref{cartanGt}, and happens to be an injection (see \cite{HelgasonGASS}, Chapter 3). We write 
\begin{equation}\label{real_helg} \mathbf{V}_{\text{Helgason}}^{\chi} = \left\{  \int_{K/M} e_{\chi, b} F(b)db, \quad F \in \mathbf{L}^2(K/M) \right\} \end{equation}  
for the image of this map, a geometric realization for the spherical principal series representation of $G$ with continuous parameter $\chi$.

 For each $t>0$, we can repeat this construction to attach to every regular $\lambda$ in $\mathfrak{a}^\star$ a geometric realization for the spherical principal series representation of $G_t$ with continuous parameter $\lambda$. Using the quasi-regular action 
\begin{equation} \label{operat_princ} \pi_t: G_t \to \text{End}(\mathbf{E})\end{equation} attached to the action $\cdot_t$ of \S \ref{subsec:cadregeom}, and comparing  the above constructions of Helgason with \S \ref{subsec:contraction_compact} (especially Lemma \ref{racines_gt}), we now define, for every $\lambda$ in  $\mathfrak{a}^\star$ and $b$ in $K/M$, a function
\begin{align}\varepsilon^{t}_{\lambda, b}: \pe & \mapsto \C \nonumber \\ v & \mapsto e^{ \ \langle \ i \lambda + t \rho \ , \ \mathfrak{I}_{t}(Ad(b) \cdot v) \ \rangle} \nonumber \end{align}
where $\mathfrak{I}_t$ is the Iwasawa projection from \eqref{iwa_It}. We set 
\begin{equation} \tag{B} \tilde{\mathbf{V}}^\lambda_t = \left\{ \int_{K/M} \varepsilon^t_{\lambda,b} F(b) db, \quad F \in \mathbb{L}^2(K/M) \right\};\end{equation}
this is a $\pi_t(G_t)$-stable subspace of $\mathbf{E}$. We note that $\tilde{\mathbf{V}}^\lambda_1$ coincides with ${\mathbf{V}}_{\text{Helgason}}^{\lambda}$ from \eqref{real_helg}.
\paragraph{Convergence of Helgason waves and contraction of spherical representations.} Still assuming that $\chi$ is a regular element of $\mathfrak{a}^\star$, suppose $\pi$ is the spherical representation of $G$ with continuous parameter $\chi$; then the geometric realization for $\pi$ described above is strongly reminiscent of the description we gave in Remark \ref{remarque_fourier} for the representation $\mathbf{M}_0(\chi, 1)$ of $G_0$ that corresponds to $\pi$ in the Mackey-Higson  bijection. There $G_0$ acted on the subspace of $\mathbf{E}$ whose elements are combinations of plane waves whose wavevectors lie on the $\ad^\star(K)$-orbit of $\chi$ in $\pe^\star$. But in view of Lemma \ref{conv_iwa}, the two kinds of waves can be deformed onto one another (see Figure \ref{convergence_ondes}):

\begin{figure}
\begin{center}
\includegraphics[width=0.24\textwidth]{Helgason1.jpg} \includegraphics[width=0.24\textwidth]{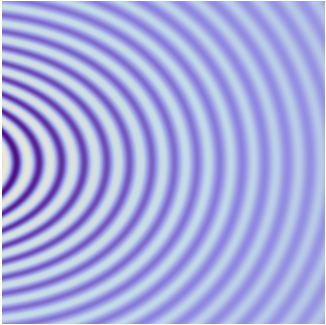} \includegraphics[width=0.24\textwidth]{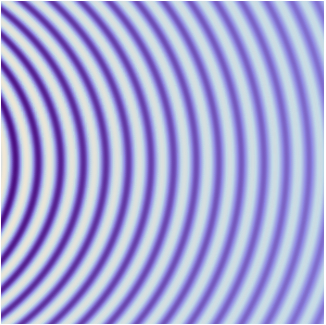} \includegraphics[width=0.24\textwidth]{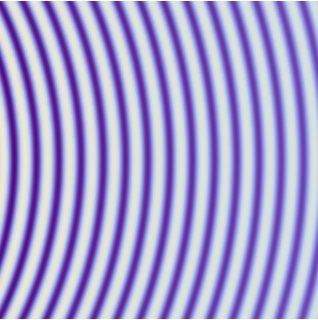} 
\caption{ Illustration of Lemma \ref{conv_ondes} : these are plots of $\varepsilon^{1/2^{k}}_{\chi,1}$, $k = 0, 1, 2, 3$, in the same domain as in Figure \ref{onde}. Each of these waves is a building block for a principal series representation of $G$ whose continuous parameter is $2^k \chi$, with $\chi = 30$ here.}
\label{convergence_ondes}
\end{center}
\end{figure}

\begin{lem} \label{conv_ondes} Fix $\lambda$ in $\mathfrak{a}^\star$ and  $b$ in $K/M$. Then as $t$ goes to zero, the Helgason waves $\varepsilon^{t}_{\lambda, b}$ converge (in the Fréchet space $\mathbf{E}$) to the Euclidean plane wave $v \mapsto e^{\langle i\lambda, Ad(b) \cdot v \rangle }$.
\end{lem}

It is then natural to use the spaces in (B) to write down a Fréchet contraction of $\pi$ to the representation of $G_0$ on 
\begin{equation} \mathbf{V}_0 = \left\{ \left[v \mapsto \int_{K/M} e^{\langle i\chi, Ad(b) \cdot v \rangle } F(b) db\right], \quad F \in \mathbf{L}^2(K/M) \right\} \end{equation}
where an element $(k,v)$ of $G_0$ acts on $\mathbf{E}$ through the usual  $\pi_0(k,v): f \mapsto \left[x \mapsto f(k^{-1}x-k^{-1}v)\right]$.

The technical necessities have already been dealt with in \S \ref{subsec:cadregeom} and \S \ref{subsec:serie_discrete} and involve the zooming-in operator 
\begin{align} \tag{C} \mathbf{C}_t: \mathbf{E}& \to \mathbf{E}\\ f & \mapsto \left[ x \mapsto f(tx) \right].\nonumber \end{align}

\begin{lem} \label{zoom_ondes} Suppose $\lambda$ is an element of $\mathfrak{a}^\star$ and $t>0$.
\begin{enumerate}[(i)]
\item For every $b$ in $K/M$, the operator $\mathbf{C}_t$ sends $e_{\lambda,b}$ to $\varepsilon^{t}_{t\lambda, b}$.
\item Fix $f$ in $\mathbf{E}$. Then $\mathbf{C}_t f$ lies in $\tilde{\mathbf{V}}_t^{\lambda}$ if and only $f$ lies in $\tilde{\mathbf{V}}_{1}^{\frac{\lambda}{t}}$.
\item The map $\mathbf{C}_t$ inertwines the actions $\pi$ and $\pi_t \circ \varphi_t^{-1}$ of $G$ on $\tilde{\mathbf{V}}_{1}^{\frac{\lambda}{t}}$ and $\tilde{\mathbf{V}}_t^{\lambda}$, respectively.
\item It is the only map satisfying (iii) which preserves the values of functions at zero.
\end{enumerate} 
\end{lem}

\begin{proof} For (i), we use Lemma \ref{iwa_It} and find that for all $v$ in $\pe$,
\[ \varepsilon^{t}_{t\lambda, b}(v) = e^{ \ \langle \ i {\lambda} + \rho \ , \ \mathfrak{I}(b \cdot (tv) ) \ \rangle} = \varepsilon^{1}_{{\lambda}, b}(tv) = \varepsilon^{1}_{{\lambda}, b}(t v)=e_{{\lambda}, b}(tv).  \]
Deducing (ii) is immediate. Part (ii) is from Lemma \ref{conjactions}(i), and part (iv) is Schur's Lemma.
 \end{proof}
 
\begin{rem} \label{spreading} It may be instructive to compare (i) and (ii) of the above result with Proposition \ref{renorm2} and property ($\text{Nat}_{\text{gen}}$) of \S \ref{subsec:programme}: we already saw that the zooming-in operator $\mathbf{C}_t$ intertwines the $G$- and $G_t$-actions on $\pe$, so it is tempting to zoom-in on a Helgason wave for $G$ to get a wave for $G_t$; without any renormalization of the spacing between phase lines, however, they would get spread out and the zooming-in process would yield a trivial outcome. 
\end{rem}
	
To obtain a contraction from $\pi \simeq \mathbf{M}(\chi, \mu_{\text{triv}})$ to $\pi_0 \simeq \mathbf{M}_0(\chi, \mu_{\text{triv}})$, the above makes it natural to use, for the carrier space at time $t$,
\begin{equation} \tag{B, corrected} \mathbf{V}_{t} = \tilde{\mathbf{V}}_{t}^{\frac{\chi}{t}}.\end{equation}

\begin{prop} The quadruple $\left( \mathbf{E}, (\mathbf{V}_t)_{t>0}, (\pi_t)_{t>0}, (\mathbf{C}_t)_{t>0}\right)$ describes a Fréchet  contraction from $\pi$ to $\pi_0$.\end{prop}
 
\noindent The only property that still needs proof is (Op1): we take up the reasoning in Lemma \ref{convop}, fix  $(k,v)$ in $K \times \pe$ and $f =   \int_{K/M} \varepsilon^1_{\lambda,b} F(b) db $ in $\mathbf{V}_1$, set $f_t = \mathbf{C}_t f =  \int_{K/M} \varepsilon^t_{\lambda,b} F(b) db$ for every $t>0$, write $f_0: v\mapsto \int_{K/M} e^{\langle i\lambda, Ad(b) \cdot v \rangle } F(b) db$ for the limit of $f_t$ as $t$ goes to zero, and remark that 
\begin{equation} \label{convop2} \pi_{t}(k,v) f_{t} = \pi_{t}(k,v)(f_{t} - f_{0}) + \pi_{t}(k,v) f_{0}. \end{equation}%
We now insert Lemma \ref{lipschitz_seriediscrete} (in the case of trivial fibers) to find that the first term goes to zero as $t$ goes to zero, and Lemma \ref{convactions} to find that the second term goes to $\pi_0(k,v) f_0$.\qed

\section{Real-infinitesimal-character representations}
\label{sec:real_inf_char}

\noindent We begin our search for a Fréchet contraction in the general case. This section describes a contraction of every irreducible tempered representation of $G$ with real infinitesimal character onto its lowest $K$-type, in case $G$ is linear connected semisimple. In \S \ref{sec:contraction_generale}, we will reduce the general case to that one.

 Fix an irreducible tempered representation $\pi$ with real infinitesimal character; write $\mu$ for its lowest $K$-type. A body of work by Vogan, Zuckerman, Schmid, Wong and others makes it possible to build out of $\mu$ a pair~$(L, \sigma)$,~where

\begin{itemize}
\item[$\bullet$] $L$ is a \emph{quasi-split} Levi subgroup of $G$, the centralizer of an elliptic element in $\g^\star$,
\item[$\bullet$] $\sigma$ is a tempered irreducible representation of $L$ with real infinitesimal character.
\end{itemize}
\noindent One can then obtain from $(L, \sigma)$ a geometric realization for $\pi$, based on the existence of a $G$-invariant complex structure on the elliptic coadjoint orbit $G/L$. Indeed, from  $(L, \sigma)$ we may form a holomorphic equivariant bundle $\mathcal{V}^\sharp$ over $G/L$  with fiber an irreducible $L$-module built from $\sigma$. The bundle usually has infinite rank; one of the Dolbeault cohomology spaces $H^{0,q}(G/L, \mathcal{V}^\sharp)$ yields a geometric realization for $\pi$ (see \S \ref{subsec:dolbeault}).

In \S \ref{sec:reduction}, we will prove that if there exists a Fréchet contraction of $\sigma$ onto its lowest $(K \cap L)$-type (and provided the contraction satisfies ``nice'' auxiliary conditions), then it is possible, using complex-analytic methods together with some of the ideas in \S \ref{subsec:serie_discrete}, to find a contraction of $\pi$ onto its lowest $K$-type. This reduces the problem of contracting $\pi$ to that of finding a (``nice'') contraction of $\sigma$. 

In \S \ref{sec:contraction_QS}, we will solve the latter problem. There $\sigma$ may be realized as an irreducible factor in a principal series representation of $L$; we will use methods of real harmonic analysis in the spirit of Helgason's wave construction and \S \ref{subsec:contraction_helgason} to obtain a (``nice'') contraction of $\sigma$.

\subsection{Reduction to the quasi-split case} \label{sec:reduction}

\noindent Throughout \S \ref{sec:reduction}, we assume that $G$ is a linear connected semisimple Lie group with finite center, we fix a  tempered irreducible representation  $\pi$ of $G$ with real infinitesimal character, and write $\mu$ for its lowest $K$-type. Our discussion is organized as follows: 
\begin{itemize}
\item[$\bullet$] In \S \ref{subsec:dolbeault}, we give details about the realization of $\pi$ by cohomological induction from $\sigma$. 
\item[$\bullet$] In \S \ref{subsec:zoom_RIC}, we construct, given a (``nice'') contraction of $\sigma$, a candidate for a Fréchet contraction of $\pi$ onto~$\mu$, and state our main finding, Theorem \ref{contraction_RIC}.
\item[$\bullet$] In \S \ref{subsec:proof}, we prove the convergence results expressed by Theorem \ref{contraction_RIC}.
\item[$\bullet$] We have gathered in \S \ref{subsec:lemmes} the proofs of several technical statements from \S \ref{subsec:zoom_RIC}.
\end{itemize}

\subsubsection{Cohomological induction and real-infinitesimal-character representations}\label{subsec:dolbeault}

We here recall how a pair $(L, \sigma)$ may be built from the lowest $K$-type $\mu$ of $\pi$, how to form the ``cohomologically induced'' module $H^{0,q}(G/L, \mathcal{V}^\sharp)$ to obtain a geometric realization for $\pi$, and how the lowest $K$-type $\mu$ can in turn be recovered from that geometrical construction. 

Fix a maximal torus $T$ in $K$. Vogan defines from $\mu$ an element $\lambda(\mu)$ in $\mathfrak{t}^\star$ (see \cite{SalamancaVogan}, Proposition 2.3 and Corollary 2.4), as follows. Write $\Delta_c^+$ for a system of positive roots for $\Delta(\ka_\C, \mathfrak{t}_\C)$, $\rho_c$ for the half-sum of positive roots, $\vec{\mu}$ for the corresponding highest weight of $\mu$. In the larger $\Delta=\Delta(\g_\C, \mathfrak{t}_C)$, fix a positive root system $\Delta^+$ making $\vec{\mu}+\rho_c$ dominant; form the corresponding half-sum  $\rho$ of positive roots; finally, define $\lambda(\mu) \in \mathfrak{t}^\star$ to be the projection of $\vec{\mu}+2\rho_c-\rho$ on the convex cone of $\Delta^+$-dominant elements.

Consider the centralizer $L$ of $\lambda(\mu)$ in $G$ for the coadjoint action. This is a \emph{quasi-split} reductive group of $G$, which is usually not compact; it admits $(K \cap L)$ as a maximal compact subgroup. 

The representation $\sigma$ of $L$ to be used in the construction is specified, for instance, in Theorem 2.9 of \cite{SalamancaVogan}: it is the inverse image of $\pi$ under the bijection (2.2) from there (here are some precisions : $\pi$ belongs to the class denoted by $\Pi^{\lambda_a}_a(G)$ in the right-hand side of Eq. (2.2) in \cite{SalamancaVogan}; the inverse image $\sigma$ of $\pi$ is an irreducible admissible representation of $L$ ; from Corollary 4.4 in \cite{SalamancaVogan} we know that $\sigma$ is tempered; from the way the bijection in \cite{SalamancaVogan} affects infinitesimal character, we know that $\sigma$ has real infinitesimal character). The unique lowest $(K \cap L)$-type of $\sigma$ is \emph{fine}, in the sense of \cite{Vogan81}, Definition 4.3.9; as a consequence, $\sigma$ appears as an irreducible constituent of a principal series representation of $L$ $-$ we shall say more on this in \S \ref{sec:contraction_QS}.

In the complexification $G_\C$, there exists a parabolic subgroup $Q$ with Levi factor $L_\C$ (for an especially convenient choice, see \cite{VoganBranching}, \S 13). Fix such a parabolic subgroup and write $\mathfrak{u}$ for the unipotent radical of the Lie algebra $\mathfrak{q}$, so that $\mathfrak{q}=\mathfrak{l}_\C \oplus \mathfrak{u}$; the complex dimension of $\mathfrak{u}$ is  $n=\frac{1}{2}(\dim(G)-\dim(L))$. The action of $L$ on $\mathfrak{q}/\mathfrak{l}_\C$ induces an action on $\mathfrak{u}$; thus, on the complex line $\wedge^{n}(\mathfrak{u})$ (top exterior power), there acts an abelian character of $L$: we will 
write $e^{2\rho(\mathfrak{u})}$ for it. 

From an irreducible $(\mathfrak{l}_\C, (K \cap L))$-module $V$ with class $\sigma$, form the $(\mathfrak{l}_\C, (K \cap L))$-module $V^\sharp = V \otimes \wedge^{n}(\mathfrak{u})$; the given choice of $Q$ equips $G/L$ with a $G$-invariant complex structure inherited from $G_\C / Q$; we can now form the holomorphic vector bundle $\mathcal{V}^\sharp=G \times_L (V^\sharp)$ over $G/L$, then the Dolbeault cohomology groups $H^{(p,q)}(G/L, \mathcal{V}^\sharp)$ (for these, see \cite{VoganVenice}, \S 7, and of course \cite{Wong99}, \S 2). Wong's work on the closed-range property for the Dolbeault operator (\cite{Wong99}, Theorem 2.4.(1)) equips $H^{(p,q)}(G/L, \mathcal{V}^\sharp)$ with a Fréchet topology and a $(\g,K)$-module structure.

The inclusion $K/(K \cap L)  \hookrightarrow G/L$ is holomorphic; we write $X$ for the (maximal) compact complex submanifold $K/(K \cap L)$ of $G/L$, and 
\[ s = \dim_\C \left( K \big/ (K \cap L )\right) \] for the complex dimension of $X$. 

\begin{thm}[\textbf{from work by  Vogan, Zuckerman, Knapp, and Wong}] The $(\g, K)$-module $H^{(0,s)}(G/L, \mathcal{V}^\sharp)$ is irreducible and its equivalence class is that of $\pi$.\end{thm}

In the rest of \S \ref{sec:reduction}, we will use this realization of $\pi$ $-$  we shall thus view $\pi$ as a $(\g, K)$-module and stay away from the delicate questions related to unitary globalizations. \\

\noindent Keeping in mind that our objective is the description of a Fréchet contraction of $\pi$ onto its lowest $K$-type $\mu$, we mention an analogous geometrical realization for $\mu$.

Recall that the representation $(V,\sigma)$ of $L$ has a unique lowest $(K\cap L)$-type: let us write $\mu^\flat$ for it, and $W$ for the $\mu^\flat$-isotypical subspace of $V$. Because $\mu^\flat$ occurs with multiplicity one in $\sigma$, the $(K \cap L)$-module $W$ is irreducible. The actions of $L$ on $V$ and $\wedge^n(\mathfrak{u})$ induce an action of $(K\cap L)$ on $W$ and another action of $(K\cap L)$ on $\wedge^n(\mathfrak{u})$; we write $W^\sharp$ for the $(K \cap L)$-module $W \otimes \wedge^n(\mathfrak{u})$. We should emphasize the importance of using $\wedge^n(\mathfrak{u})$ once more: we do not switch to $\wedge^{\text{top}} (\mathfrak{u} \cap \ka)$, although $\wedge^n(\mathfrak{u})$ does not have any meaning ``internal to $K$''.

We can as before form the holomorphic vector bundle $\mathcal{W}^\sharp=K \times_{K \cap L} (W^\sharp)$ over $X=K/(K \cap L)$ (this one has finite rank), and obtain the Dolbeault cohomology groups $H^{(p,q)}(X, \mathcal{W}^\sharp)$.

\begin{prop} The $K$-module $H^{(0,s)}\left(X,  \mathcal{W}^\sharp \right)$ is irreducible and its equivalence class is $\mu$.\end{prop}

\begin{proof} From Wong's work (\cite{Wong99}, Theorem 2.4) and the isomorphism theorems in \cite{KnappVoganLivre}, Chapter VIII, we know that the $(\g, K)$-module $H^{(0,s)}\left( X, \mathcal{W}^\sharp\right)$ defined above, the $(\g, K)$-module $\mathcal{R}^s(V)$ defined in \cite{KnappVoganLivre}, Eq. (5.3b), and the $(\g, K)$-module $\mathcal{L}_s(V)$  defined in \cite{KnappVoganLivre}, Eq. (5.3a), are all equivalent. To describe their common lowest $K$-type, we can proceed as follows:  Wong's work shows that the $K$-module $H^{(0,s)}\left(X,  \mathcal{W}^\sharp \right)$ is isomorphic with that written $\mathcal{L}^K_s(W)$ in \cite{KnappVoganLivre}, (5.70);  because of the remark below Corollary 5.72 in \cite{KnappVoganLivre}, the latter $K$-module is irreducible. Finally, the bottom-layer map (\cite{KnappVoganLivre}, \S V.6) induces, as recalled for instance in Theorem 2.9 in \cite{SalamancaVogan}, a $K$-equivariant isomorphism between $\mathcal{L}^K_s(W)$ (viewed as a $K$-submodule of $\mathcal{L}^K_s(V)$) and the $\mu$-isotypical subspace~in~$\mathcal{L}_s(V)~\simeq~\pi$. \end{proof}

%
%
%
%
%
\subsubsection{Zooming-in operators and contraction in Dolbeault cohomology}\label{subsec:zoom_RIC}
\noindent In this section, we show that given a sufficiently pleasant Fréchet contraction of the representation $\sigma$ of $L$ onto its lowest $(K \cap L)$-type  $\mu^\flat$, it is possible to build a Fréchet contraction of $\pi$ onto its lowest $K$-type.

The construction is quite technical: bringing in the deformation $(G_t)_{t>0}$, and running through the constructions of \S \ref{subsec:dolbeault} for $G_t$, we shall need to find a Fréchet space independent of $t$ in which one can embed all Dolbeault cohomology spaces. Since the isotropy groups $L_t$, infinite-dimensional modules ${V}_t$, and Dolbeault operators that appear in the definition of  $H^{(0,s)}(G_t/L_t, \mathcal{V}_t^\sharp)$ all will depend on $t$ in a more delicate way than the ingredients of \S \ref{sec:discreteprincipale}, and because of the quotient in the definition of cohomology, we will need some preparation before we can state our main result. We have relegated the proof of four technical statements to \S \ref{subsec:lemmes}, after the proof of Theorem~\ref{contraction_RIC}.

\paragraph{A. On a Mostow decomposition and the structure of elliptic coadjoint orbits.} We first give a meaning to the notion of ``zooming-in on a neighborhood of $K/(K\cap L)$ in $G/L$''. A few remarks on the structure of the coadjoint orbit $G/L$ will help.

Write $\mathfrak{s}$ for the orthogonal complement of $(\mathfrak{l} \cap \pe)$ in $\pe$ with respect to the Killing form; the dimension of $\mathfrak{s}$ is an even integer. Of course the map
\begin{align*} \Psi: K \times \mathfrak{s} & \rightarrow G/L\\ \left(k,v \right) & \mapsto k \cdot \exp_G(v) \cdot L \end{align*}
cannot be a global diffeomorphism, since the dimensions of the source and target spaces are different (the difference is $\dim(K\cap L)$). Still, this map will be useful to us:

\begin{lem} \label{mostow} \begin{enumerate}[(a)]
\item The map $\Psi$ is surjective.
\item Given $(k,v)$ and $(k', v')$ in $K \times \mathfrak{s}$, the condition $\Psi(k,v)=\Psi(k',v')$ is equivalent with the existence of some $u$ in $K \cap L$ such that: $v=Ad(u) \cdot v'$ and $k=k' \cdot u$.
\end{enumerate}
\end{lem}

\begin{proof} Part (a) is essentially due to Mostow, who proved in 1955 (\cite{Mostow}, Theorem 5) that the map
\begin{align}\label{dec_mostow} K \times \mathfrak{s} \times (\mathfrak{l} \cap \pe) & \rightarrow G \nonumber \\ \left(k, v, \beta \right) & \mapsto k \cdot \exp_G(v) \cdot \exp_G(\beta) \end{align}
is a diffeomorphism. For (b), fix two pairs  $(k,v)$ and $(k', v')$ in $K \times \mathfrak{s}$ and assume that $\Psi(k,v)=\Psi(k',v')$, in other words that $k'e^{v'} \in k e^{v} L$; then there is an element $\beta$ in $(\mathfrak{l} \cap \pe)$ and an element $u$ in $K \cap L$ such that:  $k'e^{v'} = k e^{v} u e^{\beta} = (k'u)e^{\text{Ad}(u) v'} e^{\beta}$. Fom this and Mostow's decomposition theorem we get (b). We mention that the above is very close to the proof of Lemma 4.4 in \cite{BarchiniZierau}.\end{proof} 

Now, equip  $K \times \mathfrak{s}$ with the $K \cap L$-action in which an element $u$ of $K \cap L$ acts through $(k,v) \mapsto (uk, uv)$; write $Y$ for the quotient manifold $\left( K \times \mathfrak{s}\right)/(K \cap L)$ and $\text{proj}_Y: K \times \mathfrak{s} \to Y$ for the quotient map. Given Lemma \ref{mostow}, the unique map $\psi: Y \to G/L$ such that $\psi \circ \text{proj}_Y = \Psi$ is a global diffeomorphism. The inverse image of $X=K/(K \cap L)$ under the diffeomorphism $\psi$ is a compact submanifold of $Y$ with codimension $\dim(\mathfrak{s})$, which we will still denote by $X$.\label{varieteY} For every $t>0$, the map $(k,v) \to (k,\frac{v}{t})$ (from $K \times \mathfrak{s}$ to itself) induces  a map from $Y$ to itself; this ``zooming-in map'' preserves the compact submanifold $X$. 

\paragraph{B. Spaces of ``differential forms'' and models for the Dolbeault cohomologies.}  We now bring in the deformation $(G_t)_{t \in \R}$. Taking up notations from \S \ref{subsec:programme} and  \S \ref{subsec:dolbeault}, we assume given a Fréchet contraction 
\begin{equation} \label{frechet_rec} \left(E, (V_t)_{t>0}, (\sigma_t)_{t >0}, (\mathbf{c}_t)_{t>0} \right)\end{equation} of $\sigma$ onto its lowest $(K \cap L)$-type. For every $t>0$, we can consider the subgroup $L_t = \varphi_t^{-1}(L)$ of $G_t$ (where $\varphi_t: G_t \to G$ is the isomorphism from \S \ref{subsec:defo}); the maximal compact subgroup $K \cap L_t$ does not depend on $t$. Extending $\varphi_t$ (resp. the derivative $\phi_t$) to an isomorphism between the complexifications of $G_t$ and $G$ (resp. of $\mathfrak{g}_t$ and $\mathfrak{g}$), we set $Q_t = \varphi_t^{-1}(Q)$  to obtain a parabolic subgroup in the complexification of $G_t$; let us write $\mathfrak{u}_t= \phi_t^\star \mathfrak{u}$ for the unipotent radical of its Lie algebra. From the $(\mathfrak{l}_{t,\C}, K \cap L)$-module $V_t$ and the abelian character $e^{2\rho(\mathfrak{u}_t)}$ of $L_t$ acting on $\wedge^n(\mathfrak{u}_t)$, we can form the $(\mathfrak{l}_{t,\C}, K \cap L)$-module $V^\sharp_t = V_t \otimes \wedge^n(\mathfrak{u}_t)$ and the Dolbeault cohomologies $H^{p,q}(G_t/L_t, \mathcal{V}^\sharp_t)$ from~\S~\ref{subsec:dolbeault}.

The results recalled in \S \ref{subsec:dolbeault} show that $H^{0,s}(G_t/L_t, \mathcal{V}^\sharp_t)$ carries an irreducible tempered representation of $G_t$ with lowest $K$-type $\mu$. \\

\noindent The program described in \S \ref{subsec:programme} can only be realized if we embed all of the Fréchet spaces $H^{0,s}(G_t/L_t, \mathcal{V}^\sharp_t)$ into a common Fréchet space. To do so, we need a convenient definition for the notion of ``differential form of type $(0,s)$ on $G_t/L_t$ with values in $\mathcal{V}_t^\sharp$'', and for the notions of ``closed'' and ``exact'' forms. 

 It will be helpful to assume the given contraction \eqref{frechet_rec} to have additional~``nice''~properties: 
\begin{hyp} \label{bonnedefo}
\begin{enumerate}[i.]\leavevmode
\item The Fréchet space $E$ is nuclear.
\item  For every $t>0$, the irreducible representation $\sigma_t: L_t \to \text{End}(V_t)$ is the restriction to $V_t$ of a (very reducible) smooth representation $\sigma_t: L_t \to \text{End}(E)$ ; furthermore, for every $u$ in $K \cap L_t = K \cap L$, the operator $\sigma_t(u)$ does not depend on $t$.
\item Suppose $\mathfrak{B}$ is a compact subset of $ \pe$. Then the Fréchet topology of $E$ may be defined by a distance with the property that there exists some $\kappa>0$ such that
\begin{enumerate}[(a)]
\item each the of $\sigma_t(k,v)$, $k \in K\cap L$, $v \in \mathfrak{B}\cap \mathfrak{l}_t \cap$, $t \in ]0,1]$,
\item each of the $\sigma_t(X)$, $X \in U(\mathfrak{l}_{t, \C})$ (enveloping algebra),
\item and each of the $\mathbf{c}_t, t \in ]0,1]$, 
\end{enumerate}
is $\kappa$-Lipschitz as an endomorphism of ${V}_t$.
\end{enumerate}
\end{hyp}

Property $iii.$ will be used at the end of this subsection, when we will need an analogue of Lemma \ref{lipschitz_seriediscrete} (see page \pageref{oplip}). Property $ii$, of the kind already encountered in \eqref{operat_disc} and \eqref{operat_princ}, makes it possible for us to define for every $t>0$ a holomorphic vector bundle $\mathcal{E}_t^\sharp$ over $G_t/L_t$.  In order to circumvent, in our discussion of the space of differential forms with values in $\mathcal{E}_t^\sharp$, some of the analytical difficulties related with the fact that $\mathcal{E}_t^\sharp$ is an infinite-rank bundle, we will follow Wong \cite{Wong99} and define the Dolbeault complex directly: for each nonnegative~integer~$q$,~set
\begin{equation} \label{omega_t} \Omega_t^q =\left\{ \varphi \in \mathcal{C}^\infty\left(G_t, \ E^\sharp \otimes \wedge^q(\mathfrak{u}_t^\star)\right) \quad | \quad  \forall \ell \in L_t, \forall \gamma \in G_t, \quad \varphi(\gamma\ell)= \varsigma_t(\ell)^{-1} f(\gamma) \right\} \end{equation}
where $\varsigma_t$ is the $L_t$-representation obtained from the action $\sigma_t$ on $E$, from the action of $L_t$ on $\wedge^n(\mathfrak{u}_t)$ through an abelian character denoted $e^{2\rho(\mathfrak{u}_t)}$, and from the adjoint-action-induced representation $\xi_t$ of $L_t$ on $\wedge^q(\mathfrak{u}_t^\star)$, by setting $\varsigma_t = \sigma_t \otimes e^{2\rho_{\mathfrak{u}_t}} \otimes \xi_t$. The notion of smooth function on $G_t$ with values in $E_t^\sharp \otimes \wedge^q(\mathfrak{u}_t^\star)$ is that from \cite{Treves}, Chapter 40 (see also \cite{Wong99}, \S 2). In the present context, the Dolbeault operator $\bar{\partial}_{\Omega^q_t}: \Omega^q_t \to \Omega^{q+1}_t$ is defined in Wong \cite{Wong99}, \S 2; we omit the details. 

Using the map 
\begin{align*} \Psi_t: K \times \mathfrak{s} & \rightarrow G_t/L_t\\ \left(k, v \right) & \mapsto k \cdot \exp_{G_t}(v) \cdot L_t \end{align*}
from Lemma \ref{mostow}, we can now exhibit a Fréchet space isomorphic with $\Omega^t_q$, but contained in a Fréchet space independent of $t$. Note first that the derivative of $\Psi_t$ at $(1_K, 0)$ induces a $K\cap L$-equivariant isomorphism between
 \[ \mathfrak{r}=\left(\ka/(\ka \cap \mathfrak{l})\right) \oplus \mathfrak{s}\]
 and $\g_t/\mathfrak{l}_t$, and between the complexifications $\mathfrak{r}_\C$ and $\g_{t,\C} / \mathfrak{l}_{t,\C}$. Write $\iota_t: \mathfrak{r}_\C  \to \g_{t,\C} / \mathfrak{l}_{t,\C}$ for that map, and
\begin{equation} \label{eta_t} \eta_t \subset \mathfrak{r}_\C\end{equation} for the inverse image of $\mathfrak{q}_t/\mathfrak{l}_{t, \C}$ under $\iota_t$. Identifying  $\mathfrak{q}_t/\mathfrak{l}_{t, \C}$ with $\mathfrak{u}_t$, we obtain a $(K \cap L)$-equivariant isomorphism
\begin{equation}\label{iota} \iota_t: \eta_t \to \mathfrak{u}_t; \end{equation}
that one will make it possible to view $\eta_t$ as the ``antiholomorphic tangent space to the manifold $Y$ from page \pageref{varieteY} at $\text{proj}_Y[(1_K, 0)]$, when $Y$ is equipped with the complex structure from $G_t/L_t$''.

We now define our spaces of ``differential forms of degree $q$ on $Y$ with values in $\mathcal{E}^\sharp$'' and ``differential forms of type $(0,q)$ for the $G_t$-invariant complex structure on $Y$'', setting
\begin{align} \small  A^q & :=\left\{ f \in \mathcal{C}^\infty\left(K \times \mathfrak{s}, E^\sharp \otimes \wedge^q(\mathfrak{r_\C}^\star)\right) \quad \big| \quad  \begin{aligned} & \forall u \in K \cap L, \forall (k,v) \in K \times \mathfrak{s}, \\ & f(ku, \text{Ad}(u)v)= \lambda(u)^{-1} f(k,v) \end{aligned}\quad \right\} \\
\text{and}& \nonumber \\
 A^{(0,q)_t} & :=\left\{ f \in \mathcal{C}^\infty\left(K \times \mathfrak{s}, E^\sharp \otimes \wedge^q(\mathfrak{\eta}_t^\star)\right) \quad \big| \quad  \begin{aligned} & \forall u \in K \cap L, \forall (k,v) \in K \times \mathfrak{s}, \\ & f(ku, \text{Ad}(u)v)= \lambda(u)^{-1} f(k,v) \end{aligned}\quad  \right\}\\ & \nonumber \\ &= \left\{ f \in A^q \quad / \quad \text{$f$ is $E^\sharp \otimes \wedge^q(\mathfrak{\eta}_t^\star)$-valued}\quad \right\},\nonumber\end{align}
 \normalsize
where  $\lambda: K \cap L \to \text{End}\left(E^\sharp \otimes \wedge^q(\mathfrak{r}_\C^\star)\right)$ denotes the action of $K \cap L$ on $E^\sharp \otimes \wedge^q(\mathfrak{r}_\C^\star)$ induced by $\sigma_t$, by $e^{2\rho_{\mathfrak{u}_t}}$ and by the natural representation of $K\cap L$ on $\mathfrak{r}$. It should be noted that $\lambda$ does not depend on $t$ and preserves $E^\sharp \otimes \wedge^q({\eta}_t^\star)$: this is due to Assumption \ref{bonnedefo}(i) and to the easily verified fact that $e^{2\rho(\mathfrak{u}_t)}|_{(K \cap L_t)}$ does not depend on $t$.

We equip $A^q$ with the Fréchet topology from \cite{Treves}, Chapter 40 (see \cite{Wong99}, page 5); then $A^{(0,q)_t} $ is closed~in~$A^q$.

Using Mostow's decomposition \eqref{dec_mostow}, we can extend every element in $A^{(0,q)_t} $ to an element of $\Omega_t^q$, obtaining an isomorphism of Fréchet spaces
\begin{align}\label{identif} \mathcal{I}_t: A^{(0,q)_t}  & \to \Omega^q_t \nonumber \\ f&  \mapsto \Phi^t_f :\end{align}
to be precise, when $f$ is an element of $A^{(0,q)_t} $ and $\gamma$ is in $G_t$, define $\Phi^t_f(\gamma)$ by writing $\gamma=k \exp_{G_t}^{v} \exp_{G_t}^{\beta}$   for the Mostow decomposition of $\gamma$ (here $k\in K$, $v \in \mathfrak{s}$, $\beta \in \mathfrak{l}_t \cap \pe$), then setting $\Phi^t_f(\gamma) =\varsigma_t( \exp_{G_t}({- \beta})) f(k,v)$ (where we identified  $\wedge^q(\mathfrak{\eta_t}^\star)$ and $\wedge^q(\mathfrak{u}_t^\star)$ through the isomorphism $\iota_t$ from \eqref{iota}). The map $\Phi^t_f$ is an element of $\mathcal{C}^\infty\left(G_t, E^\sharp \otimes \wedge^q(\mathfrak{u}_t^\star)\right)$, and it is easily verified that it satisfies the $L_t$-equivariance condition from \eqref{omega_t}, so that $\Phi^t_f \in \Omega_t^q$. To check that $\mathcal{I}_t$ is an isomorphism, we only have to add that every $\Phi$ in $\Omega^q_t$ is the image under $\mathcal{I}_t$ of the map $(k,v) \mapsto \Phi(k\exp_{G_t}(v))$ (here we keep identifying $\wedge^q(\mathfrak{\eta_t}^\star)$ and $\wedge^q(\mathfrak{u}_t^\star)$ through $\iota_t$).

By using $\mathcal{I}_t$ to transfer the group action and Dolbeault operator on $\Omega^q_t$ to $A^{(0,q)_t}$, we obtain, for every $t>0$:
\begin{itemize}
\item[$\bullet$] An action of $G_t$ on $A^{(0,q)_t}$; when $f$ is an element $A^{(0,q)_t}$, we will write 
\begin{equation} \label{ac_t} \textbf{ac}_t(k,v)[f]\end{equation} for the image of $f$ under the action of an element $(k,v)$ in $G_t$. 
\item[$\bullet$] A linear $\textbf{ac}_t(G_t)$-invariant differential operator
\[ {\bar{\partial}}_t: A^{(0,q)_t} \to A^{(0,q+1)_t},\] whose range is closed thanks to Wong's work \cite{Wong99}.
\end{itemize}

\noindent We now define the ``spaces of $\bar{\partial}_t$-closed and $\bar{\partial}_t$-exact forms''
\begin{equation*} F^q_t = \left\{ f \in A^q \quad / \quad f \text{ is } V_t^\sharp \otimes \wedge^q(\mathfrak{\eta_t}^\star)\text{-valued}  \quad \text{ and } \quad \bar{\partial}_t f = 0\right\}, \quad \text{and} \end{equation*}
\begin{equation*} X^q_t = \left\{ f \in A^q \quad / \quad f \text{ is } V_t^\sharp \otimes \wedge^q(\mathfrak{\eta_t}^\star)    \text{-valued, and there exists  $\omega \in A^{q-1}$ so that }  \bar{\partial}_t \omega = f\right\}. \end{equation*}
For every $t>0$,  $F^q_t$ and $X^q_t$ are  $\textbf{ac}_t(G_t)$-invariant closed subspaces of $A^{(0,q)_t} $, hence of $A^q$. The induced action and topology on $F^s_t/X^s_t$ equip that quotient with the structure of an admissible representation of $G$; given the results recalled in \S \ref{subsec:dolbeault}, this representation is irreducible, tempered, and has real infinitesimal character and lowest $K$-type $\mu$.\\

There is of course an analogous description of the (less subtle) notion of differential form on $K/(K \cap L)$ with values in the (finite-rank) vector bundle $\mathcal{W}^\sharp$ whose fiber is
\begin{itemize}
\item[$\bullet$] the subspace $W=\left\{ \lim_{t \to 0} \mathbf{c}_t \phi, \ \phi \in V_1 \right\}$ of $E$ obtained by going through the Fréchet contraction of $\sigma$ onto its lowest $(K \cap L)$-type $\mu^\flat$,
\item[$\bullet$] equipped with the action  $\mu^\flat \otimes e^{2(\rho_{\mathfrak{u}})_{|(K \cap L)}}$ of $K \cap L$.
\end{itemize} 
The intersection $\mathfrak{u}_t \cap \ka_\C$ does not depend on $t$ and can be identified with the antiholomorphic tangent space to $K/(K \cap L)$ at $1_K(K \cap L)$. The inverse image $\iota_t^{-1}(\mathfrak{u}_t \cap \ka_\C)$ is a vector subspace of $\left(\ka/(\ka\cap \mathfrak{l})\right)$ that does not depend on $t$ and is contained in each of the $\eta_t$, $t>0$: we will write $\eta_0$ for it. The space of differential forms of type $(0,q)$ on $K/(K \cap L)$ with values in $\mathcal{W}^\sharp$ can thus be identified with 
\begin{equation*} B^q = \left\{ f \in A^q \quad / \quad f \text{ is } W^\sharp \otimes \wedge^q( \eta_0 )^\star\text{-valued}  \quad \text{ and } \quad \forall (k,v) \in K \times \mathfrak{s}, \ f(k,v)=f(k,0) \right\}. \end{equation*}
For every $t>0$, we remark that $B^q$ is contained in $A^{(0,q)_t}$ and that the action $\textbf{ac}_t$ induces an action of $K$ on $B^q$, which does not depend on $t$; we will write $\textbf{ac}: K \to \text{End}(B^q)$ for it. For every $q$ is defined, as before, an $\textbf{ac}(K)$-invariant Dolbeault operator $\bar{\partial}_0: B^q \to B^{q+1}$ with closed range (the closed-range property is from \cite{Wong95}); we set 
\[ F^q_0 = \textbf{Ker}\left(\bar{\partial}_0: B^q \to B^{q+1}\right) \quad \text{and} \quad X^q_0 = \textbf{Im}\left(\bar{\partial}_0: B^{q-1} \to B^{q}\right), \]
so that $F^q_0$ and $X^q_0$ are $K$-invariant subspaces of $A^q$; the quotient $F^q_0/X^q_0$ carries an irreducible $K$-module of class $\mu$.

\paragraph{C. Contraction operators.}  We started this subsection by giving a meaning to the notion of ``zooming-in on a neighborhood of $K \cap (K \cap L)$ in $G/L$''. Furthermore, we assumed given a family of contraction operators $\mathbf{c}_t: E \to E$ that act on the fibers of the holomorphic vector bundles considered in this section. For every $t>0$, the operator $\mathbf{c}_t$  induces an endomorphism $\tilde{\mathbf{c}}_t$ of $E^\sharp \otimes \wedge^q \mathfrak{r}_\C^\star$: using the decomposition $\mathfrak{r} = \ka/(\ka\cap \mathfrak{l})\oplus \mathfrak{s}$ and writing $\tilde{\phi_t}:  \mathfrak{r} \to \mathfrak{r}$ for the map $k+v \to k+tv$ induced by the isomorphism $\varphi_t: G_t \to G$, and $\wedge^p\tilde{\phi}_t^\star$ for the endomorphism of $\wedge^q \mathfrak{r*}_\C^\star$ induced by $\tilde{\phi_t}$, we set $\tilde{\mathbf{c}}_t=\mathbf{c}_t \otimes \wedge^p\tilde{\phi}_t^\star$.

Combining these two operations on the base space and on the fibers, we obtain the linear automorphisms of $A^q$ to be used for our contraction purposes,
\begin{align} \label{contract_dolb} \mathbf{C}_t: A^q & \to A^q \\ f & \mapsto \mathbf{C}_t f = \text{ the map } (k,v) \mapsto \tilde{\mathbf{c}}_t  \cdot f(k, tv) \quad (k \in K, v \in \mathfrak{s}) .\nonumber\end{align}

We now point out that $\mathbf{C}_t$ intertwines the action of $G_1$ on $A^{(0,q)_1}$ and that of $G_t$ on $A^{(0,q)_t}$, as well as the Dolbeault operators $\bar{\partial}:= \bar{\partial}_1$ and $\bar{\partial}_t$:

\begin{lem} \label{intertw_dolb} Fix $t>0$,  $f$ in $A^{(0,q)_1}$ and $g$ in $G_1$. Then $\mathbf{C}_t\left[ \textbf{ac}(g) \cdot f\right] = \textbf{ac}_t\left(\varphi_t^{-1}(g)\right) \cdot \left[ \mathbf{C}_t f\right]$.\end{lem}

\begin{lem} \label{int_dolbeault} For every $t>0$, the operator $\mathbf{C}_t^{-1} \circ \bar{\partial}_t \circ \mathbf{C}_t$ is none other than $\bar{\partial}$. \end{lem}

For the proofs of Lemmas \ref{int_dolbeault} and \ref{intertw_dolb}, see \S \ref{subsec:lemmes}. They have an immediate consequence:

\begin{cor} \label{vect1_dolb} If $f$ lies in $F^q$, then $\mathbf{C}_t f$ lies in $F^q_t$; if $f$ lies in $X^q$, then $\mathbf{C}_t f$ lies in $X^q_t$.\end{cor}

\paragraph{D. Representatives whose contraction is harmonic.} Our aim of contracting $\pi$ onto its lowest $K$-type $\mu$ now seems within reach: at this stage, we can associate, to every vector in the carrier space  $H^{(0,s)}(G/L, \mathcal{V}^\sharp)$ for $\pi$, an element in the carrier space $H^{(0,s)}(K/(K \cap L), \mathcal{W}^\sharp)$ for $\mu$. Indeed, if $f \in F^s:=F^s_1$ is a representative of a cohomology class in $H^{(0,s)}(G/L, \mathcal{V}^\sharp)$, then as $t$ goes to zero, $f_t:=\mathbf{C}_t f$ goes (for  the given topology on $A^q$) to an element $f_0$ in $A^q$ that is constant in the $\mathfrak{s}$-directions and is $W^\sharp \otimes \left( \ka / (\ka \cap \mathfrak{l})\right)$-valued. Given the definitions of the Dolbeault operators, when $f$ lies in $F^s$, the contraction $f_0$ lies in $F^s_0$ ; furthermore, if $f$ and $g$ are two elements of $F^s$ with the same cohomology class (meaning that $f-g$ lies in $X^s:=X^s_1$), then $f_0$ and $g_0$ are two elements of $F^s_0$ with the same cohomology class (meaning that $f_0-g_0 \in X^s_0$): so the class of $f_0$ in $F^s_0/X^s_0$ depends only of that of $f$ in $F^s/X^s$, and defines an element in $H^{(0,s)}(K/(K \cap L), \mathcal{W}^\sharp)$. 

However, in order to achieve a full implementation of the program in \S \ref{subsec:programme}, we need to embed all of the $F^s_t/X^s_t$, $t>0$, inside a common Fréchet space. The need to mod out a closed subspace $X^s_t$ that depends on $t$ will make things slightly acrobatic, and we will have to choose a representative in $F^s_t$ for each cohomology class in $F^s_t/X^s_t$; in other words, we will choose for every $t>0$ a ``sufficiently pleasant'' algebraic complement to $X^s_t$ within $F^s_t$. 

Now, the complex manifold $K/(K \cap L)$ is compact:  in the space $B^q$ of differential forms with type $(0,q)$ with values in the (finite-rank) bundle $\mathcal{W}^\sharp$, we can call in the usual notion of \emph{harmonic form} (see for instance \cite{VoisinHodge}, \S II.6): a $\bar{\partial}$-closed form of type $(0,q)$ on $K/(K \cap L)$ with values in $\mathcal{W}^\sharp$ is harmonic when it is orthogonal, for the inner product on forms associated with the $K$-invariant Kähler metric on $K/(K \cap L)$ and a $K$-invariant inner product on $W$, to every $\bar{\partial}$-exact form. We thus have a notion of harmonic form in $B^s$, and a harmonic form can be $\bar{\partial}_0$-exact only if it is zero; in addition, if $f \in B^q$ is harmonic, then $\textbf{ac}(k) \cdot f$ is harmonic too. For background on harmonic forms, see \cite{VoisinHodge}, Chapter III, especially around Theorems 5.22-5.24.

Write $H^s_0$ for the subspace of $B^s$ consisting of harmonic forms: this is an $\mathbf{ac}(K)$-invariant subspace of $F^s_0$, and we have $F^s_0 = H^s_0 \oplus X^s_0$. Now, define
\begin{equation*} H^s:= \left\{  f \in F^s \enskip | \enskip f_0 \in H^s_0 \right\},\end{equation*}
a subset of $A^{(0,s)}$.

\begin{lem} \label{dec_fermees} The space $F^s$ of closed forms decomposes as $F^s=H^s+X^s$.\end{lem}
 For the proof, see \S \ref{subsec:lemmes}. By calling in the ``nuclear space'' hypothesis in Assumption \ref{bonnedefo}, we can go a step futher: not only does every class in $F^s/X^s$ admit exactly one representative in $H^s$, but we can obtain a \emph{closed} subspace of $H^s$ comprising exactly one representative for each cohomology class. 
\begin{lem} \label{nucleaire}  There exists a closed subspace $\Sigma$ of $H^s$ such that  that $F^s = \Sigma \oplus X^s$.\end{lem}
\noindent For the proof, see \S \ref{subsec:lemmes}.

Now fix such a $\Sigma$. For every $t>0$, the space 
\begin{equation} \label{sigma_t} \Sigma_t = \left\{ \mathbf{C}_t[f] \ | \ f \in \Sigma \right\}\end{equation}
is a closed complement to $X^s_t$ in $F^s_t$; since $X^s_t$ and $\Sigma_t$ are closed in $F^s_t$, the projection
\begin{equation} \label{repclasses} q_t: \Sigma_t \to F^s_t/X^s_t \end{equation}
is a Fréchet isomorphism. Thus, we have defined a subspace of $A^s$ such that every class in $H^{(0,s)}(G_t/L_t, \mathcal{V}_t^\sharp)$ has a unique representative in $\Sigma_t$.

In Lemma \ref{nucleaire}, we cannot expect $\Sigma_t$ to be stable under the action $\mathbf{ac}_t(G_t)$ of \eqref{ac_t}. We can however use the isomorphism  \eqref{repclasses} to transfer to $\Sigma_t$ the action $\textbf{ac}^{\text{quo}}_t$ of $G_t$ on $F^s_t/X^s_t$ inherited from $\mathbf{ac}_t$: recall that the projection 
\begin{equation} \label{rep_grosses_classes} p_t: F^s_t \to F^s_t/X^s_t \end{equation}
makes it possible to define for every $\gamma$ in $G_t$ a map $\textbf{ac}^{\text{quo}}_t(\gamma)$ from $F^s_t/X^s_t$ to itself, where
\begin{equation} \label{rep_quotient} \text{if $\omega = p_t(F)$ for some $F$ in $F^s_t$, then } \quad \textbf{ac}^{\text{quo}}_t(\gamma) \cdot \omega= p_t(\textbf{ac}_t(\gamma) F). \end{equation}
We then set
\begin{align} \label{op_Gt_dolb} \pi_t(k,v) \cdot F :&= q_t^{-1} \left[ \textbf{ac}^{\text{quo}}_t(k,v) \cdot q_t(F) \right]\\ &=  q_t^{-1} \left[p_t\left( \textbf{ac}_t(k,v) \cdot F\right) \right]\end{align}
for every $(k,v)$ in $K\times \pe$ and every $F$ in $\Sigma_t$: the representation $\pi_t: G_t \to \textbf{End}(\Sigma_t)$ is irreducible tempered and its lowest $K$-type is $\mu$.

\paragraph{E. Conclusion.} We now have all the pieces for a Fréchet contraction of $\pi$ onto $\mu$. Consider 
\begin{align} \tag{A} \mathbf{E}& =A^s ;\\
\tag{B} \mathbf{V}_t &= \Sigma_t  \text{ from \eqref{sigma_t}}; \\
\tag{B'} \pi_t  &\text{ from \eqref{op_Gt_dolb}}; \\
\tag{C} \mathbf{C}_t  &\text{ from \eqref{contract_dolb}}.
\end{align}
The following result is the announced reduction theorem.

\begin{thm} \label{contraction_RIC} The quadruple $\left(\mathbf{E}, (\mathbf{V}_t )_{t>0}, (\pi_t)_{t>0}, (\mathbf{C}_t)_{t>0}\right)$ describes a Fréchet contraction of the real-infinitesimal character representation $\pi$ onto its lowest $K$-type $\mu$. \end{thm}

\subsubsection{Proof of Theorem \ref{contraction_RIC}, and a remark} \label{subsec:proof}
\noindent We have already checked properties {(Vect1)} and {(Vect2)} of \S \ref{subsec:programme}: see Corollary \ref{vect1_dolb} and \S \ref{subsec:zoom_RIC}.\,D. Property ($\text{Nat}_{\text{RIC}})$ follows from Lemma \ref{intertw_dolb} togeher with the fact that $X^s_t$ is $\mathbf{ac}_t(G_t)$-stable. We now need to prove properties (Op1) and (Op2). 

Let us consider the outcome
\[ \mathbf{V}_0 = \left\{ f_0 \ / \ f \in \Sigma\right\}\]
of the contraction; this is a subspace of $F^s_0$. The contraction of every element in $\Sigma$ is by definition a harmonic form, so that $\mathbf{V}_0 \subset H^s_0$. The reverse inclusion is true, and guarantees that the $K$-module $\mathbf{V}_0$ is irreducible of class $\mu$: 

\begin{lem} \label{pas_de_blague} Every harmonic form in $B^s$ is the contraction of a form in $\Sigma$: we have $\mathbf{V}_0 = H^s_0$.\end{lem}

\begin{proof} Since the subspace $\Sigma_0$ of $F^s_0$ is $\mathbf{ac}(K)$-invariant, the projection $\Sigma_0 \to F^s_0/ X^s_0$ is a morphism of $K$-modules; since every element in $\mathbf{V}_0$ is harmonic and no nonzero harmonic form can be exact, this morphism is injective.  Now, the $K$-module $F^s_0/X^s_0$ is irreducible, so that morphism must also be surjective, unless $\mathbf{V}_0$ be zero. But if that were the case, for every $f$ in $F^s = \Sigma \oplus X^s$, the contraction $f_0$ would lie in $X^s_0$ and be exact. 

To prove that this is impossible, we call in the bottom-layer map of \cite{KnappVoganLivre}, \S V.6 (see also \cite{SalamancaVogan}, \S 2, and especially \cite{HermitianFormsSMF}, \S 9). When applied to the $(K \cap L)$-module $W^\sharp$, that map induces an injective morphism of $K$-modules from $F^s_0/X^s_0$ to $F^s/X^s$, whose image is the subspace $\widetilde{W}$ of $F^s/X^s$ carrying the lowest $K$-type of that $G_1$-representation. Inspecting definitions reveals that the map sending an element of $\widetilde{W}$ to its inverse image by the bottom-layer map coincides with the map induced by the contraction  $f \mapsto f_0$. If $f_0$ were exact for every $f$ in $\mathbf{V}_1$, we could deduce $F^s_0=X^s_0$, and that is not true.\end{proof}

Finally, fix an element $F$ in $\mathbf{V}_0$ and an element $f$ in $\Sigma_1$ satisfying $F=f_0$, and set $f_t = \mathbf{C}_t f$ for every $t>0$. To complete the proof of Theorem \ref{contraction_RIC}, we need only check that $\pi_t(k,v) f_t$ goes to $\textbf{ac}(k) \cdot f_0$ as $t$ goes to zero. 

Fix $(k,v)$ in $K \times \pe$. Given Lemma \ref{intertw_dolb} and the definition of our subspaces $\Sigma_t$ in \eqref{sigma_t}, we may write
\begin{align} 
\pi_t(k,v) f_t &= q_t^{-1}\left[ p_t\left( \mathbf{ac}_t(k,v) \cdot \mathbf{C}_t f \right) \right] \nonumber \\
&=q_t^{-1} \circ p_t  \mathbf{C}_t\left[  \left( \textbf{ac}(k\exp_G(tv)) \cdot f \right) \right] \nonumber \\ 
&= \mathbf{C}_t\left[ q_1^{-1} \circ p_1 \left( \textbf{ac}(k\exp_G(tv)) \cdot f \right) \right]  \nonumber \\ 
&= \mathbf{C}_t\left[ q_1^{-1} \circ p_1 \left( \textbf{ac}(k) \cdot f \right) \right]  +  (\mathbf{C}_t\circ q_1^{-1} \circ p_1) \left[ \textbf{ac}(k\exp_G(tv)) \cdot f  - \mathbf{ac}(k) \cdot f \right].\nonumber 
\end{align}
Recall that $f$ lies in $\Sigma_1$, and note that $\Sigma_1$ is $\mathbf{ac}(K)$-invariant, so that $p_1 \left( \textbf{ac}(k) \cdot f \right) =  q_1 \left( \textbf{ac}(k) \cdot f \right)$; applying Lemma \ref{intertw_dolb} again and inserting the fact that $\mathbf{ac}_t(k)= \mathbf{ac}(k)$ for all $t$, we see that
\begin{equation}  \label{decomp_op}\pi_t(k,v) f_t = \quad  \textbf{ac}(k) \cdot f_t   \quad +  (\mathbf{C}_t\circ q_1^{-1} \circ p_1) \left[ \textbf{ac}(k\exp_G(tv)) \cdot f  - \mathbf{ac}(k) \cdot f \right]. \end{equation}
%

When $t$ goes to zero, the first term in \eqref{decomp_op} goes to $\textbf{ac}(k) \cdot f_0$; what we need to check is that the second term goes to zero. We now insert the following two remarks (compare Lemma \ref{lipschitz_seriediscrete}):

\begin{lem}  \label{oplip} \begin{enumerate}[(i)]
\item  The topology of $\mathbf{E}$ may be defined by a distance with respect to which each of the $\mathbf{C}_t$, $t\in ]0,1]$, is $1$-Lipschitz.
\item When $\mathbf{E}$ is equipped with that distance, the Fréchet topology on $F^s/X^s$ may be defined with a distance such that the projection $p_1: F^s \to F^x/X^s$ is $1$-Lipschitz and the inverse projection $q_1^{-1}: F^s/X^s \to \Sigma$ is $2$-Lipschitz.
\end{enumerate}
\end{lem}  
\begin{proof} 
 For (i), build a countable family of semi-norms and a metric whose topology is that of $\mathbf{E}$, following the recipe in \cite{Treves}, page 412: using the notations of the first sentence after Definition 40.2 in that textbook, we use the subsets $\Omega_j = K \times B(0,j)$, $j \in \Z^+$, of  $K \times \mathfrak{s}$ (where $B(0,j)$ is the open ball $\mathfrak{s}$   with center $0_\mathfrak{s}$ and radius $j$), and use Assumption \ref{bonnedefo}iii.(c) to obtain a countable family of semi-norms on the Fréchet space $E^\sharp \otimes \wedge^s\mathfrak{r}^\star$ with respect to which the operators $\mathbf{c}_t$, $t \in ]0,1]$, are all $1$-Lipschitz.\normalsize 

 For (ii), we note that given a Fréchet distance on $F^s$, the usual way to define a Fréchet topology on $F^s/X^s$ is through a distance for which $p_1: F^s \to F^s/X^s$ is $1$-Lipschitz (see \S 12.16.9 in \cite{Dieudonne}); what needs proof is the statement on the inverse $q_1^{-1}$. Now, the inclusion $\Sigma_t \hookrightarrow F^s_t$ and the projection $F^s_t \to F^s_t/X^s_t$ are $1$-Lipschitz, so  the composition $q_1$ is $1$-Lipschitz; the open mapping theorem for Fréchet spaces then shows that $q_1^{-1}$ is actually $2$-Lipschitz with respect to the mentioned distances on $\Sigma$ and on $F^s/X^s$ (see \cite{Dieudonne}, \S 12.16.9 for the open mapping theorem and \S 12.16.8.2 for the Lipschitz statement). \end{proof}

Returning to the proof of Property (Op1) in Theorem \ref{contraction_RIC}, we take up  \eqref{decomp_op} and let $t$ go to zero. Then $\textbf{ac}(k\exp_G(tv)) \cdot f - \textbf{ac}(k) f$ goes to zero; from Lemma \ref{oplip} we deduce that $\pi_t(k,v) f_t$ goes to $\textbf{ac}(k) \cdot f_0$. Lemma \ref{pas_de_blague} yields (Op2);  the proof of Theorem \ref{contraction_RIC} is now complete. \qed\\

\noindent We close this section by mentioning, for future use in \S \ref{sec:contraction_generale}, that our Fréchet contraction satisfies a weakend analogue of Assumption \ref{bonnedefo}:

\begin{lem} \label{bonnedefo_dolb}
Suppose $U$ is a compact subset of $\pe$. Then there is a positive number $\tilde{C}$ with the property that
\begin{enumerate}[(i)]
\item the Fréchet topology of $\mathbf{E}$ may be defined by a distance with respect to which each of the $\pi_t(k,v)$, $k \in K$, $v \in U$, $t \in ]0,1]$, and each of the $\mathbf{C}_t, t \in ]0,1]$, is $\tilde{C}$-Lipschitz as an endomorphism of $\mathbf{V}_t$.
\item  for every $t>0$ and every $(k,v)$ in $K \times U$, the operator $\pi_t(k,v)$ on $\mathbf{V}_t$ may be extended to a $\tilde{C}$-Lipschitz operator acting on all of $\mathbf{E}$. This may be done so that the resulting map $\pi_t: G_t \to \text{\emph{End}}(\mathbf{E})$ is continuous.
\end{enumerate}
\end{lem}

Lemma \ref{bonnedefo_dolb} is a rather poor substitute for Assumption \ref{bonnedefo}, because the extension mentioned in part (ii) is simple-minded:  we do \emph{not} try to define it so that $\pi_t: G_t \to \text{GL}(\mathbf{E})$ is a group morphism, and our Lipschitz estimates are crude. What will matter in \S \ref{sec:contraction_generale} is the existence of an extension to all of $\mathbf{E}$ and of a Lipschitz estimate uniform in $k$, $v$ and $t$ in the chosen domain.

\begin{proof} 
For (i), we first unfold the definition of $\mathbf{ac}_t(\gamma)$ in \eqref{ac_t} and find that for every $F$ in $A^{(0,s)_t}$ and $\gamma$ in $G_t$, \begin{align*} \mathbf{ac}_t(\gamma) \cdot F & = \mathcal{I}_t^{-1} \left[ x \mapsto (\mathcal{I}_t F)(\gamma^{-1} x)\right] \\ &= (k,v) \mapsto \varsigma_t\left( \exp_{G_t}^{-\beta_t(\gamma^{-1} k \exp_{G_t}^{v})}\right) F\left[ \kappa_t(\gamma^{-1} k \exp_{G_t}^v), \ X_t(\gamma^{-1} k \exp_{G_t}^v)\right]
\end{align*}
where $\kappa_t: G_t \to K$, $X_t: G_t \to \mathfrak{s}$ and $\beta_t: G_t \to \mathfrak{l}_t \cap \pe$ are the Mostow maps of Lemma \ref{intertw_dolb}.

We now observe, as in the proof of Lemma \ref{lipschitz_seriediscrete}, that there exists an increasing family $(\Omega_n)_{n \in \N}$ of relatively compact open subsets of $K \times \mathfrak{s}$ such that for every $\gamma$ in $K\exp_{G_t}(U)$ and every $t$ in $]0,1]$, the image of $\Omega_n$ under the map $(k,v) \mapsto \left(\kappa_t(\gamma^{-1} k \exp_{G_t}^v), \ X_t(\gamma^{-1} k \exp_{G_t}^v)\right)$ is contained in $\Omega_{n+1}$. We can in fact choose $\Omega_n$ of the form  $K \times B(0,R_n)$, where $(R_n)_{n \in \N}$ is an increasing sequence of positive radii and $B(0,R_n)$ is the ball with radius $R_n$ in $\mathfrak{s}$. We then use those and Assumption \ref{bonnedefo}.iii.(a)-(b) to follow the proof of Lemma \ref{oplip}.(i) and build the desired metric on $\mathbf{E}$.

For (ii), fix a closed complement $Z$ to $F^s$ in $\mathbf{E}=A^s$, and for all $t>0$, set $Z_t = \mathbf{C}_t Z$, so that 
\[ \mathbf{E} = \Sigma_t \oplus X^s_t \oplus Z_t.\]
Write $\text{proj}_{ \Sigma_t}$, $\text{proj}_{ X^s_t}$ and $\text{proj}_{ Z_t}$ for the corresponding linear projections, and for every $(k,v)$ in $G_t$, extend the operator $\pi_t$ on $\Sigma_t$ by having the extension act as the identity on $X^s_t$ and $Z_t$, setting 
\begin{equation} \label{extension} \tilde{\pi}_t(k,v) F = \pi_t(k,v)\left[\text{proj}_{ \Sigma_t} F\right] + \text{proj}_{ X^s_t}(F)+\text{proj}_{ Z_t}(F)\end{equation}
for all $F$ in $\mathbf{E}$. Of course this extension is so simple-minded that it cannot define a representation of $G_t$ on $\mathbf{E}$, but we note, taking up some arguments from the proof of Lemma \ref{oplip}(ii), that the projections $\text{proj}_{ \Sigma_t}$, $\text{proj}_{ X^s_t}$ and $\text{proj}_{ Z_t}$ are all 2-Lipschitz. Indeed, if $A$ is a Fréchet space equipped with a compatible distance and $B,C$ are closed subspaces such that $A = B \oplus C$, then the projection $A \to C$ is $2$-Lipschitz: it can be obtained by first applying the projection $A \to A/B$, which  is $1$-Lipschitz if the usual distance on $A/B$ is chosen, then applying the isomorphism $A/B \to C$, which is the inverse of the $1$-Lipschitz isomorphism $C \to A/B$ and is as a consequence $2$-Lipschitz. 

Returning to \eqref{extension}, we know from Lemma \ref{oplip} that $\pi_t(k,v)$ is $4$-Lipschitz, and use the preceding remarks to deduce that $\tilde{\pi}_t(k,v)$ is $10$-Lipschitz. The continuity of $(k,v) \mapsto \tilde{\pi}_t(k,v)$ is straightforward from that of $\pi_t$.
\end{proof}

\subsubsection{Proof of Lemmas 4.5, 4.6, 4.8 and 4.9} \label{subsec:lemmes}

Our proofs of the Lemmas \ref{intertw_dolb} and \ref{int_dolbeault} use the notations of \S \ref{subsec:zoom_RIC}\, B-C.

\begin{proof}[Proof of Lemma \ref{intertw_dolb}]

 In order to make the notations in the statement and proof lighter, we write $A^{(0,q)}$ for $A^{(0,q)_1}$ and $\textbf{ac}: G \to \text{End}(A^{(0,q)})$ for the $G$-action obtained by composing  $\textbf{ac}_1$ and the isomorphism $\varphi_1: G_1 \to G$.
 
 When $g$ is an element of $G$, we write its Mostow decomposition as $g = \kappa(g) \exp_G^{X(g)} \exp_G^{\beta(g)}$, getting Mostow maps $\kappa: G \to K$, $X: G \to \mathfrak{s}$ and $\beta: G \to \mathfrak{l} \cap \pe$. We write $\kappa_t: G_t \to K$, $X_t: G_t \to \mathfrak{s}$ and $\beta_t: G_t \to \mathfrak{l}_t \cap \pe$ for the analogous Mostow maps of $G_t$. We now fix $g$ in $G$ and $t>0$.

\begin{itemize}
\item[$\bullet$] By definition, $\textbf{ac}(g)\cdot f$ is a map from $K \times \mathfrak{s}$ to $E^\sharp \otimes \wedge^q \mathfrak{u}^\star$, \emph{viz.}
\begin{equation*}\textbf{ac}(g)\cdot f: (k,v) \mapsto \varsigma\left(e^{-\beta\left[g^{-1}k\exp_G(v)\right]}\right) f\left(\kappa\left[g^{-1}k\exp_G(v)\right], X\left[g^{-1}k\exp_G(v)\right]\right).\end{equation*}

Thus $\mathbf{C}_t\left[ \textbf{ac}(g) \cdot f\right] $ is the map
\begin{align*} (k, v) & \mapsto \tilde{\mathbf{c}}_t \cdot \left[\textbf{ac}(g)\cdot f\right](k,tv) \\ &=  \tilde{\mathbf{c}}_t \cdot \varsigma\left(\exp_G^{-\beta\left[g^{-1}k\exp_G(tv)\right]}\right) f\left(\kappa\left[g^{-1}k\exp_G^{tv}\right], \quad X\left[g^{-1}k\exp_G^{tv}\right]\right) \\ &=  \varsigma_t \left(\varphi_t^{-1}(\exp_G^{-\beta\left[g^{-1}k\exp_G^{tv}\right])}\right) \tilde{\mathbf{c}}_t \cdot f\left(\kappa\left[g^{-1}k\exp_G^{tv}\right], \quad X\left[g^{-1}k\exp_G^{tv}\right]\right) \\ &=  \varsigma_t \left(\exp_{G_t}^{-\frac{1}{t}\beta\left[g^{-1}k\exp_G^{tv}\right])}\right) \tilde{\mathbf{c}}_t \cdot f\left(\kappa\left[g^{-1}k\exp_G^{tv}\right], \quad X\left[g^{-1}k\exp_G^{tv}\right]\right)
\end{align*}
(in the transition between the second and third line, we use property ($\text{Nat}_{\text{RIC}}$) for the contraction operator $\mathbf{c}_t$: it intertwines the representations $\sigma$ and $\sigma_t \circ \varphi_t^{-1}$ of $L$).

\item[$\bullet$] We need to compare this with $ \textbf{ac}_t\left(\varphi_t^{-1}(g)\right) \cdot \left[ \mathbf{C}_t f\right]$, which is the map 
\begin{align*} 
(k,v) &\mapsto \varsigma_t\left( \exp_{G_t}^{ \beta_t\left(\varphi_t^{-1}(g)^{-1} k \exp_{G_t}(v)\right)}\right) \cdot \left[ \mathbf{C}_t f \right]\left( \kappa_t\left(\varphi_t^{-1}(g)^{-1} k \exp_{G_t}^v\right),\quad  X_t\left(\varphi_t^{-1}(g)^{-1} k \exp_{G_t}^v\right)\right) \\ 
&= \varsigma_t\left( \exp_{G_t}^{\beta_t\left(\varphi_t^{-1}(g)^{-1} k \exp_{G_t}(v)\right)}\right) \cdot \tilde{\mathbf{c}}_t \cdot f\left( \kappa_t\left(\varphi_t^{-1}(g)^{-1} k \exp_{G_t}^v\right), \quad t \cdot X_t\left(\varphi_t^{-1}(g)^{-1} k \exp_{G_t}^v\right)\right).
\end{align*}
If $\gamma = k e^{v} e^{\beta}$ is an element of $G$, then $\varphi_t^{-1}(\gamma) = k \exp_{G_t}^{X/t} \exp_{G_t}^{\beta/t}$, so that $\kappa(\gamma)=\kappa_t(\varphi_t^{-1}(\gamma)$, $X(\gamma) = t \cdot X_t(\varphi_t^{-1}(\gamma))$ and $\beta(\gamma) = t \cdot \beta_t(\varphi_t^{-1}(\gamma))$. Using this to compare the last lines of our descriptions of $\mathbf{C}_t\left[ \textbf{ac}(g) \cdot f\right]$ and $\textbf{ac}_t\left(\varphi_t^{-1}(g)\right) \cdot \left[ \mathbf{C}_t f\right]$, we see that they are equal. \end{itemize}\end{proof}

\begin{proof}[Proof of Lemma \ref{int_dolbeault}] Recall that our Dolbeault operator $ \bar{\partial} $ on $A^{(0,q)}$ is defined by transferring to $A^{(0,q)}$ the Dolbeault operator acting on $\Omega^q_1$ (defined in \cite{Wong99} \S 2), through the Fréchet isomorphism $\mathcal{I}: A^{(0,q)} \to \Omega^q_1$ from \eqref{identif}. For every $t>0$, the operator $ \bar{\partial}_t$ acting on  $A^{(0,q)_t}$ similarly comes from Wong's Dolbeault operator on $\Omega^q_t$, transferred to  $A^{(0,q)_t}$ via the Fréchet isomorphism $\mathcal{I}_t: A^{(0,q)_t} \to \Omega^q_t$. 

We now remark that the linear automorphism $\mathbf{C}_t$ of $A^q$ induces a linear isomorphism between $\Omega^q_1$ and $\Omega^q_t$,  the map $\mathcal{I}_t \circ \mathbf{C}_t \circ \mathcal{I}_1^{-1}$. Inspecting definitions, we notice that for every $F: G_1 \to E^\sharp \otimes \wedge^q \mathfrak{u}^\star$ in $\Omega^q_1$, the map $\mathcal{I}_t \circ \mathbf{C}_t \circ \mathcal{I}_1^{-1}(F)$ is none other than
\begin{align*} \mathcal{C}_t(F): G_t & \to E^\sharp \otimes \wedge^q \mathfrak{u}^\star \\ (k,v) &\mapsto (\mathbf{c}_t \otimes \phi_t^\star) \cdot F(\varphi_t(k,v))\end{align*}
where $\phi_t^\star: \wedge^q\mathfrak{u}^\star \to  \wedge^q\mathfrak{u}_t^\star$ is the linear isomorphism induced by the derivative of $\varphi_t^{-1}$ at the identity.

This means that $\mathcal{C}_t(F)$ is the element $\Omega^q_t$  induced by the group isomorphism $\varphi_1 \circ \varphi_t^{-1}$ between $G_1$ and $G_t$. That isomorphism induces an isomorphism of holomorphic vector bundles between the bundle $\mathcal{V}_1^\sharp$ over $G_1/L_1$ and the bundle $\mathcal{V}^\sharp_t$ over $G_t/L_t$; by naturality of the Dolbeault operator, that bundle map must intertwine the Dolbeault operators acting on  $\Omega^q_1$ and $\Omega^q_t$. So
\[ \mathcal{C}_t^{-1} \bar{\partial}_{\Omega^q_t} \mathcal{C}_t =\bar{\partial}_{\Omega^q_1},\]
which proves Lemma \ref{int_dolbeault}.
  \end{proof} 
\separateur 

Our proofs of Lemmas \ref{dec_fermees} and \ref{nucleaire} use the notations of \S  \ref{subsec:zoom_RIC}\,D.
\begin{proof}[Proof of Lemma \ref{dec_fermees}] Fix $f$ in $F^s$; the contraction $f_0$ of $f$ is $\bar{\partial}_0$-closed, so we can write $f_0$ as $ a+b$, where $a\in H_0^s$ is a harmonic form and $b\in X_0^s$ is a $\bar{\partial}_0$-exact form. Write $b=\bar{\partial}_0 \beta$, where $\beta$ lies in $B^{s-1}$; we can  extend $\beta$ trivially to $K \times \mathfrak{s}$, obtaining an element $\gamma$ of $A^{s-1}$ that is constant in the $\mathfrak{s}$-directions. Set $\omega = \bar{\partial} \gamma$; then the form $\omega$ lies in $X^s$, and $f-\omega$ lies in $H^s$ (because $f_0-\omega_0=a$ is harmonic). We have proved that $f$ can be written as a sum of a form in $H^s$ and another in $X^s$; that is the lemma.  \end{proof}

\begin{proof}[Proof of Lemma \ref{nucleaire}] We need only find an algebraic complement $\Sigma$ for $H^s \cap X^s$ in $H^s$ that is closed for the topology of $H^s$. The fact that a \emph{closed} $\Sigma$ may be chosen follows from the theory of nuclear spaces (see Tr\`eves \cite{Treves}, chapters 50 and 51, and Grothendieck \cite{Grothendieck}): because of Assumption \ref{bonnedefo}, the Fréchet space $A^s$ is nuclear (see \cite{Treves}, Proposition 50.1 and \S 51.8), so the closed subspaces $H^s$, $X^s$ and $H^s \cap X^s$ are both nuclear; because $X^s\cap H^s$ is closed in $H^s$, the identity $H^s \cap X^s\to H^s \cap X^s$ is a nuclear linear map (\cite{Grothendieck}, corollaire 2, p. 101); it admits (\cite{Grothendieck}, théorème 3, p. 103) a nuclear extension $\mathcal{F}: H^s  \to H^s \cap X^s$. We define $\Sigma$ as the kernel of $\mathcal{F}$: it is a closed complement to $H^s \cap X^s$ in $H^s$, and the lemma follows.\end{proof}


\subsection{The case of quasi-split groups and fine $K$-types}\label{sec:contraction_QS}

\noindent In this section, we consider a pair $(G, \pi)$ where
\begin{itemize}
\item[$\bullet$] $G$ is a \emph{quasi-split} linear group with abelian Cartan subgroups,
\item[$\bullet$] $\pi$ is a tempered irreducible representation of $G$ with real infinitesimal character and a \emph{fine} lowest $K$-type~$\mu$,
\end{itemize}
and we set out to build a Fréchet contraction of $\pi$ onto $\mu$.

In the present context, the representation $\pi$ can be realized as an irreducible factor in a (usually nonspherical) \emph{principal series} representation of $G$, as in Chapter 4 in \cite{Vogan81}. Suppose $P=MAN$ is a \emph{minimal} parabolic subgroup in $G$, then in our quasi-split case:
\begin{itemize}
\item[$\bullet$] $M$ is abelian (and compact, but not connected), 
\item[$\bullet$] and there exists a character $\sigma$ of $M$ such that $\mu$ is one of the lowest $K$-types in 
\[ \Pi = \ind_{MAN}^G\left(\sigma \otimes \mathbf{1}\right).\] 
\end{itemize}

The representation $\Pi$ is a finite direct sum of irreducible tempered subrepresentations, each of which has real infinitesimal character and a unique lowest $K$-type; the irreducible factor that contains $\mu$ is equivalent with $\pi$ as a $G$-representation. 

Recall that in the compact picture from \S \ref{subsec:contraction_compact}, the representation $\Pi$ acts on
\begin{equation*} \mathbf{L}^2_\sigma(K):= \left\{ \Phi: K \rightarrow \C \ / \ \forall (u,m) \in K \times M, \ \Phi(um)=\sigma(m)^{-1} \Phi(u) \right\} \end{equation*} through the action obtained by choosing $\lambda=0$ in \eqref{opGt}. The subspace
\begin{equation} \label{sousmodule} \mathbf{H} = \text{irreducible subspace of $\mathbf{L}^2_\sigma(K)$ containing the $K$-type $\mu$} \end{equation}   thus carries a (not very explicit) realization of $\pi$ on a space of complex-valued functions on $K$. 

Because the induced representation $\Pi$ is reducible and has continuous parameter zero, the results of \S \ref{subsec:contraction_compact} cannot be applied here; yet some of the interplay between \S \ref{subsec:contraction_compact} and \S \ref{subsec:contraction_helgason} will be helpful.


\subsubsection{A Helgason-like realization}\label{subsec:quasisplit}

\noindent We now describe another realization, in which $G$ acts on a space of sections of an equivariant vector bundle over $G/K$; it mimics the realization of spherical principal series representations in \S \ref{subsec:contraction_helgason} using Helgason waves.

For the rest of \S \ref{subsec:quasisplit}, we fix an irreducible $K$-module $V^\mu$ of class $\mu$, write $\mathcal{E}$ for the equivariant bundle $G/K \times_K V^\mu$ over $G/K$ and $\Gamma(\mathcal{E})$ for the space of smooth sections of $\mathcal{E}$.\\

\noindent We first describe the special sections of $\mathcal{E}$ which will play in our construction the part Helgason waves~did~in~\S\ref{subsec:contraction_helgason}.

From Frobenius reciprocity and the fact that the $K$-type $\mu$ occurs with multiplicity one in $\mathbf{H}$, we know that the restriction $\mu_{|M}$ contains $\sigma$ with multiplicity one. We fix a vector $v \in V^\mu$ belonging to the one-dimensional $\sigma$-isotypical subspace and introduce
 \begin{equation*}
\begin{array}{ccccc}
 \bar{\gamma} & : & G & \to & V^\mu  \\
 & &  ke^{H}n & \mapsto & e^{-\langle \rho,\ H\rangle} \mu(k)^{-1} v. \\
\end{array}
\end{equation*}

\begin{lem} \label{transfogamma} For all $(g,m, H, n)$ in $G\times M \times \mathfrak{a}\times N$, we have $\bar{\gamma}(gme^{H}n) = e^{-\langle \rho,\ H\rangle}\sigma(m)^{-1} \gamma(g)$.\end{lem}

\emph{Proof.} We use the Iwasawa decomposition \eqref{iwa} to find\vspace{-0.2cm} \begin{align*} \bar{\gamma}(gman)& = \bar{\gamma}\left(\kappa(g) e^{H(g)} \nu(g) man \right) \\ &= \bar{\gamma}\left( \kappa(g) m e^{H+H'} \tilde{n} \right) \text{ for some $\tilde{n} \in N$ \Fin{(recall that $N$ normalizes $MA$ and $M$ centralizes $A$)}} \\ &= e^{-\langle \rho,\ H+H'\rangle} \mu(\kappa(g)m)^{-1} v \\ &= \sigma(m)^{-1} e^{-\langle \rho,\ H\rangle}  \left[e^{-\langle \rho,\ H(g)\rangle} \mu(\kappa(g))^{-1} v\right]\  \Fin{\text{(since $\mu(m)^{-1} v = \sigma(m)^{-1}v$ and $\sigma(m)$ is scalar)}}\\ &= e^{-\langle \rho,\ H\rangle}\sigma(m)^{-1} \gamma(g). \hspace{10cm} \qed\end{align*} 

Now set $\tilde{\gamma}: g \mapsto \bar{\gamma}(g^{-1})$. A simple calculation 
shows that  $\tilde{\gamma}(gk)=\mu(k^{-1}) \gamma(g)$ for every $k$ in $K$ and $g$ in $G$. As a consequence, $\tilde{\gamma}$ induces a section $\gamma$ of the vector bundle $\mathcal{E}$ over $G/K$. We will take up $\gamma$ as an analogue in $\Gamma(\mathcal{E})$ of the Helgason ``wave'' $e_{\lambda, 1_K}$ with (not very wavy) frequency $\lambda=0$.

 For every $b$ in $K$, the map $\tilde{\gamma}_b: g \mapsto \tilde{\gamma}(b^{-1}g)$ induces another element $\gamma_b$ in $\Gamma(\mathcal{E})$, which we will take up as an analogue of the Helgason ``wave'' $e_{0,\,b}$. From Lemma \ref{transfogamma}, we know that $\gamma_{bm} = \sigma(m)^{-1} \gamma_b$ for every $(b,m) \in K \times M$.\\
 
\noindent Using these special sections, we build an invariant  subspace of $\Gamma(\mathcal{E})$ (for the usual action of $G$ on $\Gamma(\mathcal{E})$). We associate a section of  $\mathcal{E}$ with each of the functions on $K$ belonging to the Hilbert space $\mathbf{H}$ from \eqref{sousmodule}: set \begin{align*} \mathcal{T}: \mathbf{H} & \rightarrow \Gamma(\mathcal{E}) \\ \Phi & \mapsto \int_K \gamma_b {\Phi}(b) db;\end{align*}
note from the previous $M$-equivariance properties of $\Phi$ and $b \mapsto \gamma_b$ that the integrand is invariant under the change of variables  $b \leftarrow b \cdot m$. 

\begin{lem}  The map $\mathcal{T}$ is $G$-equivariant and injective. \end{lem}

\begin{proof} Fix $x$ in $G$ and write $x = \kappa(x)e^{\mathbf{a}(x)}\mathbf{n}(x)$ for its Iwasawa decomposition. To prove that $\mathcal{T}$ is $G$-equivariant, we must prove that for every  $\Phi$ in  $\mathbf{H}$, the sections
\begin{align} \label{transfoH}\mathcal{T}(\pi(x) \phi) &= \mathcal{T}\left[b \mapsto e^{-\langle \rho,\ \mathbf{a}(x^{-1} b) \rangle}\cdot\Phi\left(\kappa(x^{-1} b)\right) \right]  \nonumber \\ 
&=  \left(gK \mapsto \int_K   e^{-\langle \rho,\ \mathbf{a}(x^{-1} b) \rangle} \gamma_b(gK)\cdot \Phi\left(\kappa(x^{-1} b)\right) db \right) \nonumber \\
 &=  \left(gK \mapsto \int_K   e^{-\langle \rho,\ \mathbf{a}(x^{-1} b) \rangle} e^{-\langle \rho,\ \mathbf{a}(g^{-1}b) \rangle} \cdot\Phi\left(\kappa(x^{-1} b)\right) \left(\mu\left(\kappa(g^{-1}b)\right) \cdot v \right)db \right) \end{align} 
and \begin{align} \label{transfoE} \lambda(x) \mathcal{T}(\phi) &= \left( g \mapsto \int_K \gamma_b(x^{-1} gK)\cdot {\Phi}(b) db \right) \nonumber \\ 
&=  \left( g \mapsto \int_K e^{\langle -\rho,\ \mathbf{a}(g^{-1} x b)\rangle} \cdot{\Phi}(b) \left(\mu\left(\kappa(g^{-1} x b)\right)\cdot v\right)  db \right)\end{align} 
are equal. Going from \eqref{transfoH} to \eqref{transfoE} by substituting $b \leftarrow \kappa(x^{-1}b)$ in \eqref{transfoE} is made possible by the fact that the Iwasawa maps $\kappa$, $\mathbf{a}$, $\mathbf{n}$ satisfy cocycle relations. The necessary steps are identical with those taken by Camporesi for his proof of Proposition 3.3 in \cite{Camporesi}, so we omit the details.  

The injectivity of  $\mathcal{T}$ comes from the irreducibility of $\pi$: from the equivariance we know that $\mathcal{T}^{-1}(\{0\})$ is a closed $G$-invariant subspace of $\mathbf{H}$. It is either $\{0\}$ or $\mathbf{H}$. If it were equal to $\mathbf{H}$, the map $\mathcal{T}$ would be identically zero. Fixing a $K$-invariant inner product on $V_\sigma$, we notice that the matrix element $\xi_v: k \mapsto \langle{v, \mu(k^{-1}) v}\rangle_{V_\sigma}$ (attached to the special $v$ above) lies in $\mathbf{L}^2_\sigma(K)$;  by the Peter-Weyl theorem  it transforms under $\Pi_{|K}$ according to $\mu$, so $\xi_v$ lies in $\mathbf{H}$. The value of $\mathcal{T}(\xi_v)$ at the origin of $\pe$ is the vector $\int_K (\mu(k^{-1}) v)\cdot {\xi_v}(k) dk$ in $V^\mu$, whose scalar product with $v$ is nonzero because of Schur's relations. This proves that $\mathcal{T}(\xi_v)$ is not zero, so $\mathcal{T}$ must be injective. \end{proof}

\noindent We have obtained a geometric realization for $\pi$ in which $G$ acts on the subspace
\[ \left\{\int_K \gamma_b \cdot \Phi(b)\ db, \ \Phi \in \mathbf{H} \right\}\]
of $\Gamma(G/K, V^\mu)$, through the usual action of $G$ on $\Gamma(G/K, V^\mu)$.


\subsubsection{Contraction to the lowest $K$-type}

\noindent We can now describe a Fréchet contraction of $\pi$ onto $\mu$ using the same zooming-in operators as we did for the discrete series in \S \ref{subsec:serie_discrete} and for the Helgason-wave picture for spherical principal series in \S \ref{subsec:contraction_helgason}. Set
\begin{equation} \tag{A} \mathbf{E} = \mathcal{C}^\infty(\pe, V^\mu). \end{equation}
For every $t>0$, we transfer to functions on $\pe$ the sections of $\Gamma(G_t/K, V^\mu)$ which the construction in \S \ref{subsec:quasisplit} yields for $G_t$. Calling in the Iwasawa maps $\mathfrak{I}_t: \pe \to \mathfrak{a}$ and $\kappa_t: \pe \to K$ from \eqref{iwa_It}, we define
 \begin{equation*}
\begin{array}{ccccc}
 \Gamma^{\ t}_b & : & \pe & \to & V^\mu  \\
 & &  u & \mapsto & e^{\langle t\rho,\ \mathfrak{I}_t(\text{Ad}(b) \cdot v) \rangle} \mu^{-1}(\kappa_t(u)) v. \\
\end{array}
\end{equation*}
In view of the previous subsection, we define
\begin{equation} \tag{B} \mathbf{V}_t =  \left\{\int_K \Gamma^{\ t}_b \cdot \Phi(b) \ db, \ \Phi \in \mathbf{H}_\pi\right\}. \end{equation}
The action of $G_t$ on $\Gamma(G_t/K, V^\mu)$ yields a $G_t$-action on $\mathbf{E}$, \emph{viz.}
\begin{align}
G_t \times \mathbf{E} &\rightarrow \mathbf{E} \nonumber \\
  \left((k,v), F\right)&\mapsto \pi_t(k,v) F = \left( x \mapsto  \mu(k) \cdot F\left[(k,v)\cdot_t x\right] \right)  \tag{B'}
  \end{align}
where $\cdot_t$ is the $G_t$-action on $\pe$ defined in \S \ref{subsec:cadregeom}. The subspace $\mathbf{V}_t$ of $\mathbf{E}$ is then $\pi_t$-stable.

Bringing in the zooming-in map $z_t: u \mapsto \frac{u}{t}$ from $\pe$ to itself, we recall from Lemma \ref{naturalite_SD} that  
\begin{equation} \tag{C} \mathbf{C}_t : f \mapsto f \circ z_t^{-1} \end{equation}
intertwines the $G_1$- and $G_t$-actions on $\mathbf{E}$; from Lemma \ref{zoom_ondes} we deduce that $\mathbf{C}_t$ sends $\Gamma^{1}_b$ to $\Gamma^t_b$ (the main difference with the situation in \S \ref{subsec:contraction_helgason} is that we need no renormalization of spatial frequencies here, because we are dealing with representations with continuous parameter zero). Thus,
\begin{equation} \tag{Vect1} \text{$\mathbf{C}_t$ sends $\mathbf{V}_1$ to $\mathbf{V}_t$.}
\end{equation}
Now if $f =  \left(\int_K \Gamma^1_b \Phi(b) db\right)$ is any element in $\mathbf{V}_1$, then as $t$ goes to zero, $\mathbf{C}_t f$ converges in $\mathbf{E}$ to
\begin{equation} \label{Vect2} \tag{Vect2} f_0 = \text{the constant function with value} \int_{K} (\mu(b) \cdot v) \Phi(b) db.\end{equation}
If $(k,v)$ is a fixed element in $K \times \pe$ and $f$ is an element in $\mathbf{V}_1$, we proved in \S \ref{subsec:serie_discrete} (see \eqref{convop} and Lemma \ref{lipschitz_seriediscrete}) that 
\begin{equation} \tag{Op1} \text{As $t$ goes to zero, } \quad \pi_t(k,v)f_t \quad \text{ goes to } \quad \mu(k) \cdot \left(\text{ the value of }f_0\right).\end{equation}
From the Schur orthogonality relations and formula {(Vect2)}, we know that
\begin{itemize}
\item[$\bullet$] if $\Phi$ lies in a $K$-isotypical subspace of $\mathbf{H}$ for a $K$-type other than $\mu$, then $f_0 = 0$; 
\item[$\bullet$] but the map $f \mapsto f_0$ induces a $K$-equivariant isomorphism between the $\mu$-isotypical subspace of $\mathbf{H}$ (which is spanned, if we fix an orthogonal basis $(v_j)$ for $V^\mu$, by the matrix elements $k \mapsto \langle{v_j, \mu(k)^{-1} v}\rangle$) and the fiber $V^\mu$. \end{itemize}
We obtain the following conclusion: 
\begin{align} \tag{Op2} \text{The outcome space } \mathbf{V}_0=\left\{ f_0, \ f \in \mathbf{H}_1 \right\} \text{ identifies, as a $K$-module, with $V^\mu$ ;} \\ \text{ for every $(k,v)$ in $K \times \pe$, the limit of $\pi_t(k,v)f_t$ then identifies with $\mu(k) \cdot f_0$.}\nonumber \end{align}
This concludes our search for a contraction of $\pi$ in the present case:
\begin{thm} The quadruple $(\mathbf{E}, (\mathbf{V}_t)_{t>0}, (\pi_t)_{t>0}, (\mathbf{C}_t)_{t>0})$ describes a Fréchet contraction of $\pi$ onto $\mu$. Assumption  \ref{bonnedefo} is satisfied by this Fréchet contraction.  \end{thm}
\noindent (Concerning the statement about Assumption \ref{bonnedefo}, parts $i.$ and $ii.$ of  Assumption \ref{bonnedefo} are clearly satisfied here; for part $iii.$, take up the proof of Lemma \ref{lipschitz_seriediscrete} using star-shaped subsets of $\pe$ in the construction of seminorms.)

\section{Contraction of an arbitrary tempered representation}
\label{sec:contraction_generale}

\subsection{Real-infinitesimal-character representations for disconnected groups} \label{subsec:disco}

\noindent Our analysis of real-infinitesimal-character representations in \S \ref{sec:real_inf_char} crucially used Wong's work \cite{Wong99} on the closed-range property for the Dolbeault operator. The setting in \cite{Wong99} is that of a linear connected semisimple group: that was our reason for assuming $G$ to be linear connected semisimple in \S \ref{sec:real_inf_char}.

Now suppose $G$ is a linear connected reductive group. When $(\chi, \mu)$ is an arbitrary Mackey parameter, the representation of $G$ built from $(\chi, \mu)$ in \S \ref{subsec:correspondance} is induced from a parabolic subgroup $P_\chi= M_\chi A_\chi N_\chi$, using the real-infinitesimal-character representation $\sigma$ of $M_\chi$ with lowest $K_\chi$-type $\mu$. 

Of course $M_\chi$ is usually disconnected, although the disconnectedness is limited in that $M_\chi$ satisfies the axioms in \S 1 of \cite{KnappCompositio}. It is possible that the theorem of Wong mentioned in \ref{subsec:dolbeault} is true for linear (disconnected) reductive groups such as $M_\chi$, but I have not been able to find a reference; if we are to use the results of \S \ref{sec:real_inf_char} in a future discussion of parabolic induction from $P_\chi$, we must check that they extend~to~reasonably~disconnected~groups.\\

\noindent For the rest of \S \ref{subsec:disco}, we thus consider a reductive group $G$ satisfying the axioms of section 1 in \cite{KnappCompositio}: $G$ is a matrix group with a finite number of connected components, the identity component $G_e$ of $G$ has compact center and a reductive Lie algebra, and $G \subset G^{\C} \cdot \text{Cent}_{\text{GL}}(G)$, where $\text{Cent}_{\text{GL}}(G)$ is the centralizer of $G$ in the total general linear group of matrices and $G^{\C}$ is the analytic linear group with Lie algebra $\mathfrak{g}_\C$. \\

\noindent  Consider the identity component $G_{\text{e}}$. It is a connected Lie group and can be decomposed as $G_{e} = G_{\text{ss}} \left(Z_{G}\right)_{\text{e}}$, with $G_{ss}$ a connected semisimple Lie group with finite center. The abelian group $\left(Z_{G}\right)_{\text{e}}$ is compact and central in $G_{e}$ (\cite{KnappPrincetonBook}, \S V.5). 

Suppose we start with an irreducible tempered representation of $G_{\text{e}}$ with real infinitesimal character. Then the elements in $\left(Z_{G}\right)_{\text{e}}$ will act as scalars, defining an abelian character of $\left(Z_{G}\right)_{\text{e}}$. The restriction to $G_{\text{ss}}$ of our representation will be irreducible and have real infinitesimal character. 

A representation $\pi_{\text{e}}$ of $G_{\text{e}}$ in the class $\widetilde{(G_\text{e})}_{\text{RIC}}$  is thus uniquely specified by a $\pi_{\text{ss}}$ in the class $\widetilde{(G_{ss})}_{\text{RIC}}$ and an abelian character $\xi$ of $\left(Z_{G}\right)_{\text{e}}$ whose restriction to $G_{ss} \cap \left(Z_{G}\right)_{\text{e}}$ coincides with $\left(\pi_{ss}\right)\big|_{G_{\text{ss}} \cap \left(Z_{G}\right)_{\text{e}}}$. Given $\pi_{\text{ss}}$ and $\xi$, any carrier space for $\pi_{\text{ss}}$ furnishes a carrier space for $\pi_\text{e}$, with $g = g_{\text{ss}} g_{\left(Z_{G}\right)_{\text{e}}}$ acting through   $\xi(g_{\left(Z_{G}\right)_{\text{e}}}) \pi_{\text{ss}}(g_{\text{ss}})$.

Now that we know how to describe the class $\widetilde{(G_{\text{e}})}_{\text{RIC}}$, let us write $G^\sharp$ for the subgroup $G_\text{e}Z_G$ of $G$; because of \cite{KnappPrincetonBook}, Lemma 12.30, $G_\text{e}$ has finite index in $G^\sharp$; in fact there is a finite, abelian subgroup $F$ of $K$ (it is the subgroup called $F(B^-)$ in \cite{KnappPrincetonBook}), lying in the center of $G$ (hence of $G^\sharp$),  such that
\[ G^\sharp = G_{\text{e}} F. \]
The arguments we used for $G_{\text{e}}$ go through here: a representation $\pi^\sharp$ in the class $\widetilde{(G^{\sharp})}_\text{RIC}$ is uniquely specified by a representation $\pi_{\text{e}}$ in $\widetilde{(G_{\text{e}})}_{\text{RIC}}$ and an abelian character $\chi$ of $F$ whose restriction to $G_{\text{e}} \cap F$ coincides with $\left(\pi_{\text{e}}\right)\big|_{G_{\text{e}} \cap F}$. Given $\pi_\text{e}$ and $\chi$, any carrier space for $\pi_\text{e}$ furnishes a carrier space for $\pi^{\sharp}$, with $g = g_{\text{e}} f$ acting through   $\chi(f) \pi_{\text{e}}(g_{\text{e}})$.

To obtain a tempered representation of $G$, we can start from a representation $\pi^\sharp$ in the class $\widetilde{(G^\sharp)}_\text{RIC}$ and set
\[ \pi = \text{Ind}_{G^\sharp}^G \left( \pi^\sharp \right). \]
It turns out that $\pi$ is irreducible, has real infinitesimal character, and that $\pi^\sharp \mapsto \pi$ maps the class  $\widetilde{(G^\sharp)}_{\text{RIC}}$ \emph{bijectively onto}  $\widetilde{G}_{\text{RIC}}$: this comes from Knapp and Zuckerman's work on the discrete series (see e.g. \cite{KnappPrincetonBook}, Proposition 12.32), and from the fact that every representation in $\widetilde{G}_{\text{RIC}}$ occurs in one induced from one in the discrete series of a Levi subgroup). In addition, the restriction of $\pi$ to $G^\sharp$ decomposes as 
\[ \pi\big|_{G^\sharp} = \sum \limits_{\omega \in G/G^\sharp} \omega\pi^\sharp \]
where $\omega\pi^{\sharp}$ is $g\mapsto \pi^\sharp(\omega^{- 1} g \omega)$. Since $G^\sharp$ has finite index in $G$ (see (12.74) in \cite{KnappPrincetonBook}), the sum is finite here.\\

\noindent With the above description in hand, it is a simple matter to reduce the contraction problem for representations in the class $\widetilde{G}_{\text{RIC}}$ to the already solved contraction problem for representations in the class $\widetilde{(G_{\text{ss}})}_{\text{RIC}}$.

Fix a representation $\pi$ in $\widetilde{G}_{\text{RIC}}$; the above discussion says how $\pi$ may be constructed from a triple $(\chi, \xi, \sigma)$, where $\sigma$ is a representation of $G_{\text{ss}}$ in the class $\widetilde{(G_{\text{ss}})}_{\text{RIC}}$ and $\chi$, $\xi$ are characters of  $\left(Z_{G}\right)_{\text{e}}$ and $F$. 

Write $\mu^\flat$ for the lowest $K_{\text{ss}}$-type of $\sigma$ and assume given a Fréchet contraction 
\begin{equation} \label{contr_reduction}\left(E, (V_t)_{t>0},(\sigma_t)_{t>0},(\mathbf{c}_t)_{t>0}  \right) \end{equation}
of $\sigma$ onto $\mu^\flat$. Remark that $\left(Z_{G}\right)_{\text{e}}$ and $F$ are contained in $K$: there is in fact a direct product decomposition 
\begin{equation}K = \left(Z_{G}\right)_{\text{e}} \cdot F \cdot  K_{\text{ss}},\end{equation} 
thus for each $t>0$ the group $G_t^\sharp = \varphi_t^{-1}(G^\sharp)$ decomposes as
\[G_t^\sharp  =\left(Z_{G}\right)_{\text{e}} \cdot F \cdot \varphi_t^{-1}(G_{\text{ss}}).\]

For every $t>0$, a representation of $G_t$ may be obtained by following the above procedure: write $\sigma_t^\sharp$ for the irreducible tempered representation of $G_t^\sharp$ obtained from the representation $\sigma_t$ of $\varphi_t^{-1}(G_{\text{ss}})$, so that every element $ \gamma_{{\left(Z_{G}\right)_{\text{e}}}} \gamma_{F}\gamma_{ss}$ of $G_t^\sharp$ acts on $V_t$ through the operator $\xi(\gamma_{{\left(Z_{G}\right)_{\text{e}}}}) \chi(\gamma_f) \sigma_t(\gamma_{\text{ss}})$. For each $t>0$, the representation 
\begin{equation} \label{pi_t_nonconnexe} \pi_t:= \ind_{G_t^\sharp}^{G_t} \left( \sigma_t^\sharp \right) \end{equation}
is irreducible tempered, has real infinitesimal character; its lowest $K$-type is the lowest $K$-type $\mu$ of $\pi$. 

Turning to $K$-types, there is (applying the above remarks to the reductive group $K$) an analogous description of $\mu$ as induced from a representation of $K^\sharp$: from the representation $\mu^\flat$ of $K_{\text{ss}}$ and the given characters $\xi$, $\chi$, form the representation $\mu^\natural := (\mu^\flat)^\sharp$ of $K^\sharp$ where any $k^\sharp =  k_{{\left(Z_{G}\right)_{\text{e}}}} k_{F}k_{\text{ss}}$ acts through $\xi(k_{{\left(Z_{G}\right)_{\text{e}}}}) \chi(k_F) \sigma_t(k_{\text{ss}})$, then induce to $K$: the result is an irreducible $K$-module, and we do have
\begin{equation} \label{outcome} \mu \simeq \ind_{K^\sharp}^{K} \left( \mu^\natural \right).\end{equation}
We now construct a Fréchet contraction of $\pi$ onto $\mu$. The representation space for $\pi_t$ is a finite direct sum of copies of $V_t$, indexed by $G_t/G_t^\sharp$. Remarking that the inclusion of $K$ in $G_t$ induces a bijection 
\begin{equation} \label{isomclasses}  K / K^\sharp \overset{\sim}{\longrightarrow} G_t / G_t^{\sharp}, \end{equation}
we will thus write the direct sum of copies of $V_t$ as
\begin{equation} \tag{B} \mathbf{V}_{t} = \sum \limits_{\omega \in K/K^\sharp} V_{t,\omega}. \end{equation}
The action of $G_t$ on $ \mathbf{V}_{t}$ may be described by choosing a section of the projection $K \to K/K^\sharp$, that is, a finite collection $(\kappa_\omega)_{\omega \in K/K^\sharp}$ of elements of $K$: for every $\gamma$ in $G_t$, there is a~collection~$(\gamma^\sharp_{a})_{a \in K/K^\sharp}$~of~elements~of $G_t^\sharp$~satisfying
\begin{equation} \label{coc} \forall \omega \in K/K^\sharp, \enskip \gamma \kappa_{\omega} = \kappa_{[\gamma \omega]_t} \gamma^{\sharp}_{[\gamma \omega]_t},\end{equation}
 where the action $\omega \to [\gamma \omega]_t$ of $\gamma$ on $K/K^\sharp$ is that induced from the action of $\gamma$ on $G_t/G_t^\sharp$ through the bijection \eqref{isomclasses}. One can then set 
\begin{equation} \label{beprime} \tag{B'}\pi_t(\gamma) = \text{ \enskip the operator \enskip} \sum \limits_{\omega} x_{\omega} \mapsto \sum \limits_{\omega} \sigma_t^\sharp(\gamma^{\sharp}_{[\gamma\omega]}) x_{[\gamma{\omega}]} \end{equation}
to obtain a realization for the representation \eqref{pi_t_nonconnexe} of $G_t$.

Now, each of the $\mathbf{V}_t$, $t>0$, is contained in a finite direct sum of copies of the Fréchet space $E$ from \eqref{contr_reduction}; we will denote that direct sum by
\begin{equation} \tag{A} \mathbf{E} =  \sum \limits_{\omega \in K/K^\sharp} E_{\omega}, \end{equation}
and equip it with the direct product topology. For each $t>0$, the contraction operator $\mathbf{c}_t$ acts on every summand; if we write $\mathbf{c}^\omega_t$ for the action on the summand with index $\omega$, we obtain  an operator 
\begin{align} \tag{C} \mathbf{C}_t: \mathbf{V}_{1} & \to \mathbf{V}_t \\  
 \left( \sum \limits_{\omega \in K/K^\sharp} f_{\omega} \right) & \mapsto\left( \sum \limits_{\omega \in K/K^\sharp} \mathbf{c}_{t}^{\omega} f_{\omega}\right).\nonumber
 \end{align}
 
 \begin{prop} \label{contraction-nonconnexe} The quadruple $\left(\mathbf{E}, (\mathbf{V}_t)_{t>0},(\pi_t)_{t>0},(\mathbf{C}_t)_{t>0}  \right) $ describes a Fréchet contraction of $\pi$ onto $\mu$. \end{prop}

 We start with property ($\text{Nat}_{\text{RIC}}$):

\begin{lem} The map $\mathbf{C}_{t}$ intertwines $\pi_t \circ \varphi_t^{-1}$ and $\pi_1 \circ \varphi_1^{-1}$ . \end{lem}

\begin{proof}

Fix $g$ in $G$, write  $\gamma_1$ for the element $\varphi_1^{-1}(g)$ of $G_1$ and $\gamma_t$ for the element $\varphi_t^{-1}(g)$  of $G_t$. Suppose $f = \sum\limits_{\omega} f_{\omega}$ is an element of $\mathbf{V}_1$. Taking up (B'), we use the intertwining relation for $\mathbf{c}_t$ to find
\begin{equation}\label{lemme_ct} \mathbf{C}_{t} \left( \pi_1(\gamma_1) f \right)  = \mathbf{C}_t\left( \sum \limits_{\omega} \sigma_t^\sharp(\gamma^{\sharp}_{[\gamma_1\omega]}) f_{[\gamma_1{\omega}]} \right)=  \sum \limits_{\omega} \mathbf{c}^{\omega}_t \left[ \sigma_1^\sharp(\gamma^{\sharp}_{[\gamma_1\omega]}) f_{[\gamma_1{\omega}]_1}\right] =
   \sum \limits_{\omega}  \sigma_1^\sharp(\varphi_t^{-1} \varphi_1(\gamma^{\sharp}_{[\gamma_1\omega]_1}))\cdot \left[ \mathbf{c}^{\omega}_t f_{[\gamma_1{\omega}]_1}\right]. \end{equation}
 For every $\omega$ in $K/K^\sharp$, the isomorphism $\varphi_t \circ \varphi_1^{-1}$ sends the class of $\gamma_1 \kappa_{\omega}$ in $G_1/G_1^\sharp$ to the class of $\gamma_t \kappa_{\omega}$ in $G_t/G_t^\sharp$. As a consequence, $[\gamma_1 \omega]_1$ and $ [\gamma_t \omega]_t$ coincide. Together with  \eqref{coc}, this means that
\[  \varphi_t \varphi_1^{-1}\left(\gamma^\sharp_{[\gamma_1 \omega]_1}\right) = \varphi_t \varphi_1^{-1} \left( \gamma_1 \kappa_{\omega}\kappa_{[\gamma_1 \omega]}^{-1}\right) = \varphi_t \varphi_1^{-1} \gamma_1 \left( \kappa_{\omega}\kappa_{[\gamma_1 \omega]_1}^{-1}\right) = \gamma_t \kappa_{\omega}\kappa_{[\gamma_t \omega]_t}^{-1} = \gamma^\sharp_{[\gamma_t \omega]_t}.\]
Inserting this into \eqref{lemme_ct}, we find $\mathbf{C}_{t} \left( \pi_1(\gamma_1) f \right)  =  \sum \limits_{\omega}  \sigma_t^\sharp(\gamma^{\sharp}_{[\gamma_t\omega]_t})\cdot \left[ \mathbf{c}^\omega_t f_{[\gamma_t{\omega}]} \right] =  \pi_t(\gamma_t) \cdot \left[\mathbf{C}_t f \right]$, as desired.
\end{proof}

For each vector $f$ in $\mathbf{V}_1$, the contraction $f_t:=\mathbf{C}_{t} f$ does of course admit a limit as $t$ goes to zero. The set of possible limits is
\begin{align*} \mathbf{V}_{0} :=  \left\{ f_{0} \ | \ f \in \mathbf{E} \right\}  = \sum \limits_{\omega \in K/K^\sharp} V_{0, \omega},\end{align*}
where $V_{0, \omega} := \left\{ \lim(\mathbf{c}_t f) \ ; \ f \in V_{1, \omega}\right\}$ carries an irreducible representation of $K^\sharp$ with class $\mu^\natural$. \\

Turning to the convergence of operators, fix $F =  \sum\limits_{\omega} f_{0,\omega}$ in $\mathbf{V}_0$ and $f =  \sum\limits_{\omega} f_{\omega}$ in $\mathbf{V}_1$ such that $F=f_0$. Fix $(k,v)$ in $K \times \pe$. To observe $\pi_t(k,v) f_t$, we need to insert $\gamma=\exp_{G_t}(v) k$ in \eqref{beprime}; for every  $\omega$ in $K/K^\sharp$, we remark that $\gamma^\sharp_{[\gamma \omega]_t}$ from \eqref{coc} is equal to $\exp_{G_t}(v) k^\sharp_{k\omega}$, where $k^\sharp_{k\omega}$ is the element of $K^\sharp$ satisfying $k \kappa_{\omega} = \kappa_{k \omega} k^{\sharp}_{k\omega}$ (compare \eqref{coc}; here $k\omega$ is just the class of $k \kappa_{\omega}$ in $K/K^\sharp$). This indicates that 
\begin{align*} \pi_t(k,v) f_t = \sum \limits_{ \omega \in K/K^\sharp} \pi_t^\sharp(k^\sharp_{k\omega}, v) f_{t, \omega}.\end{align*}
As $t$ goes to zero, the fact that we started with a contraction from $\sigma$ onto $\mu^\flat$ means that for every $\omega$, $\pi_t^\sharp(k^\sharp_{k\omega}, v) f_{t, \omega}$ goes to $\mu^\natural(k^\sharp_{k\omega}) f_{0, \omega}$. Thus, 
\begin{align*} \pi_t(k,v) f_t \text{\quad goes to  \quad } \pi_0(k,v) F := \sum \limits_{\omega \in K/K^\chi} \mu^\natural(k^\sharp_{k\omega}) f_{0, \omega}.\end{align*}
In the last formula, we recognize the $K$-action in $\ind_{K^\sharp}^K(\mu^\natural)$; from \eqref{outcome} we deduce that the $G_0$-representation $(\mathbf{V}_0, \pi_0)$ is the extension to $G_0$ of an irreducible representation of $K$ equivalent with $\mu$,~proving~Proposition~\ref{contraction-nonconnexe}.\qed\\

\noindent We note that if the contraction \eqref{contr_reduction}  from $\sigma$ onto $\mu^\flat$ satisfies Lemma \ref{bonnedefo_dolb}, then the contraction from $\pi$ onto $\mu$ we built in Proposotion \ref{contraction-nonconnexe} obviously satisfies Lemma \ref{bonnedefo_dolb}, too.
 
\subsection{Parabolic induction and reduction to real infinitesimal character}

\noindent We return to our linear connected reductive group $G$. Consider an arbitrary Mackey parameter $(\chi, \mu)$ and call in the parabolic subgroup $P_\chi = M_\chi A_\chi N_\chi$ from \eqref{inductionG}, together with the irreducible tempered representation $\sigma=V_{M_\chi}(\mu)$ with real infinitesimal character and lowest $K_\chi$-type $\mu$ (recall that $M_\chi \cap K = K_\chi$). The representation of $G$ that we attached to $(\chi, \mu)$ in  \S \ref{subsec:correspondance} is 
\begin{equation} \pi = \ind_{M_\chi A_\chi N_\chi} \left( \sigma \otimes e^{i\chi}\right), \end{equation}
and the representation of $G_0$ paired with $\pi$ in the Mackey-Higson bijection is 
\begin{equation} \pi_0 = \ind_{K_\chi \ltimes \pe } \left( \mu \otimes e^{i\chi}\right). \end{equation}
We will start from a Fréchet contraction 
\begin{equation} \label{contraction_initiale} \left(E, (V_t)_{t>0}, (\sigma_t)_{t>0}, (\mathbf{c}_t)_{t>0}\right)\end{equation}
of $\sigma$ onto $\mu$, assume that it has the properties in Lemma \ref{bonnedefo_dolb}, and build a Fréchet contraction of $\pi$ onto $\pi_0$.

 For every $t>0$, the space 
\begin{equation}\label{vt} \mathbf{V}_t = \left\{ f : K \overset{\text{continuous}}{\longrightarrow} V_t \quad | \quad \forall u \in K_\chi,  \ \forall k \in K, \ f(ku) = \sigma_t(u^{-1}) f(k) \right\} \end{equation}  
can be equipped with a family of irreducible representations of $G_t$:  recall from Lemma \ref{iwa_gt} that the inverse image of $M_\chi A_\chi N_\chi$ under the isomorphism $\varphi_t: G_t \to G$ is a parabolic subgroup $M_{t, \chi} A_\chi N_{t, \chi}$ of $G_t$; for every $\lambda$ in $\mathfrak{a}_\chi^\star$, we use  the Iwasawa maps 
\begin{equation*} \kappa_t: G_t  \to K, \quad\quad\quad \mathbf{m}_t: G_t  \to M_{t,\chi} \cap \exp_{G_t}(\pe), \quad\quad\quad \mathbf{a}_t: G_t  \to \mathfrak{a}_\chi, \end{equation*}
and the half-sum $\rho_t$ of positive roots in the ordering used to define $N_{t, \chi}$, to define an endomorphism of $\mathbf{V}_t$: we set 
\begin{equation} \label{pi_t_comp} \pi^{t, \text{comp}}_{\lambda, \mu}(\gamma) :=  f \mapsto  \left[ k \mapsto  \exp{\langle - i\lambda - \rho_t, \mathbf{a}_t(\gamma^{- 1} k) \rangle} \sigma(\mathbf{m}_t(\gamma^{- 1}k)) f\left(\kappa_t(\gamma^{- 1} k)\right) \right] \end{equation}
for every $f$ in $\mathbf{V}_t$ and every $\gamma$ in $G_t$. See the discussion of the compact picture in \S \ref{subsec:contraction_compact}.

Embedding these spaces into a fixed Fréchet space is easy: they are all vector subspaces of 
\begin{equation} \tag{A} \mathbf{E} = \left\{ \text{ continuous functions from $K$ to $E$ } \right\}, \end{equation}
a Fréchet space when equipped with the topology of uniform convergence. Now for $t>0$, consider
\begin{align}\tag{B} \mathbf{V}_t & \ \text{from \eqref{vt}} \\ \tag{B'} \pi_t &= \pi^{t, \text{comp}}_{\chi, \mu} \quad \text{from \eqref{pi_t_comp} }\\ \tag{C} \mathbf{C}_t &=  \text{ pointwise composition with the map $\mathbf{c}_t$ from \eqref{contraction_initiale} }. \end{align}
We are ready to complete the task we set ourselves.
\begin{thm} \label{contraction_finale} Suppose $G$ is a linear connected reductive group, $\pi$ is an  irreducible tempered representation of $G$, and $\pi_0$ is the representation of $G_0$ that corresponds to $\pi$ in the Mackey-Higson bijection of Theorem \ref{correspondance}. The quadruple $\left( \mathbf{E}, (\mathbf{V}_t)_{t>0},  (\pi_t)_{t>0},  (\mathbf{C}_t)_{t>0}  \right)$ above describes a Fréchet contraction of $\pi$ onto $\pi_0$. \end{thm}

\noindent It is perhaps simplest to start by establishing the naturality condition ($\text{Nat}_{\text{gen}}$) of \S \ref{subsec:programme}: we need only mimic the case of principal series representations and use the facts encountered in the proof of Proposition \ref{renorm2} concerning the normalization of continuous parameters.
 
For every $\lambda$ in $\mathfrak{a}_\chi^\star$, use  $\varphi_1: G_1 \to G$ to define a representation $\tilde{\pi}^{}_{\lambda}$ of $G$ acting on $\mathbf{V}_1$, the composition
\[ \tilde{\pi}_{\lambda} \ : \  G \overset{\varphi_{1}^{- 1}}{\longrightarrow} G_{1}  \overset{\pi^{1, \text{comp}}_{\lambda, \mu}}{\longrightarrow} \text{End}(\mathbf{V}_1). \]
Consider the representation of $G$ transferred from $\pi_t$ through the isomorphism $\varphi_t: G_t \to G$, 
\[ \pi_t \circ \varphi_t^{-1} \ : \  G \overset{\varphi_{t}^{- 1}}{\longrightarrow} G_{t}  \overset{\pi^{t, \text{comp}}_{\chi, \mu}}{\longrightarrow} \text{End}(\mathbf{V}_t). \]
\begin{lem}  The map $\mathbf{C}_t$ intertwines $\pi_t \circ \varphi_t^{-1}$ and $\tilde{\pi}_{\chi/t}$.  \end{lem} 

\emph{Proof.} Fix $f$ in $\mathbf{V}_1$ and $g$ in $G$. We need to compare  $(\pi_t \circ \varphi_{t}^{- 1})(g) \cdot [\mathbf{C}_t f]$ with $\mathbf{C}_t\left[\tilde{\pi}_{\chi/t} (g)f\right]$.
Lemma \ref{racines_gt} proves that $\rho_t$ is equal with $t \rho_1$, so given the definition of $\mathbf{C}_t$,  
\[ (\pi_t \circ \varphi_{t}^{- 1})(g) \cdot [\mathbf{C}_t f] \enskip\text{is the map }\enskip  k \mapsto  e^{\langle - i\chi - t\rho, \mathbf{a}_{t}( \left[\varphi_{t}^{- 1}g\right]^{- 1} k) \rangle} \sigma_{t}(\mathbf{m}_{t}(\left[\varphi_{t}^{- 1}g\right]^{- 1} k) )\left[\mathbf{c}_t f\right]\left(\kappa_{t}(\left[\varphi_{t}^{- 1}g\right]^{- 1} k)\right) \]
for every $g$ in $G$. Lemma \ref{iwa_gt} identifies the maps $\kappa_t \circ \varphi_t^{-1}$, $\mathbf{m}_t \circ \varphi_t^{-1}$ and $\mathbf{a}_t \circ \varphi_t^{-1}$; inserting, from property ($\text{Nat}_{\text{RIC}}$) of the Fréchet contraction from $\sigma$ onto $\mu$, the fact that $\mathbf{c}_{t}$ intertwines $\sigma_1$ and $\sigma_{t} \circ \varphi_t^{-1}$, some rearranging leads to
\begin{align*} \pi^{t, comp}_{\lambda, \sigma}(\varphi_{t}^{- 1}(g)) \left[ \mathbf{C}_{t} f \right] & =  \left[ k \mapsto  e^{\langle - i\frac{\lambda}{t} - \rho, \ t \cdot \mathbf{a}_{t}( \varphi_{t}^{- 1}\left[g^{- 1} k\right]) \rangle}  \sigma_{t}(\mathbf{m}_{t}(\left[\varphi_{t}^{- 1}g\right]^{- 1} k) ) \left(\mathbf{c}_{t} f\right) \left(\kappa_{t}(\varphi_{t}^{- 1}\left[g^{- 1} k\right])\right) \right] \\
& = \left[ k \mapsto  e^{\langle - i\frac{\lambda}{t} - \rho, \ \mathbf{a}( g^{- 1} k) \rangle}\left\{ \mathbf{c}_{t}  \sigma(g^{- 1}k)\right\} f\left(\kappa(g^{- 1} k)\right) \right] \\
& = \mathbf{C}_t\left[\tilde{\pi}_{\chi/t}(g) f\right], \text{ as desired.} \hspace{7cm} \qed\end{align*}
Among the properties of \S \ref{subsec:programme}, (Vect1) and (Vect2) are once again obvious from the fact that the quadruple $\left(E, (V_t)_{t>0}, (\sigma_t)_{t>0}, (\mathbf{c}_t)_{t>0}\right)$ contracts $\sigma$ onto $\mu$: if $f$ lies in $\mathbf{V}_1$ and we set $f_t = \mathbf{C}_t f$ for every $t>0$, then $f_t$ converges to an element of 
\begin{equation} \mathbf{V}_0 =  \left\{ f : K \overset{\text{continuous}}{\longrightarrow} V_0 \quad | \quad \forall u \in K_\chi,  \ \forall k \in K, \ f(ku) = \sigma_0(u^{-1}) f(k) \right\} \end{equation} 
where $V_0$ is the outcome of the Fréchet  contraction $\left(E, (V_t)_{t>0}, (\sigma_t)_{t>0}, (\mathbf{c}_t)_{t>0}\right)$, an irreducible $K_\chi$-module carrying an irreducible representation $k \mapsto \sigma_0(k)$ of class $\mu$. Every $F$ in $\mathbf{V}_0$ does arises as the limit of $\mathbf{C}_t f$ for some $f$ in $\mathbf{V}_1$.

The space $\mathbf{V}_0$ comes already equipped with an irreducible $G_0$-representation equivalent with $\pi_0$, that described in \S \ref{subsec:correspondance} and Eq. \eqref{actionG0}: for every $F$ in $\mathbf{V}_0$ and $(k,v)$ in $G_0$, we set 
\begin{equation} \label{actionG0_finale} \pi_0(k,v)\cdot F = \left[ u \mapsto e^{i \langle Ad^\star(u)\chi, v\rangle} F(k^{- 1} u)\right].\end{equation}
The only thing that still needs proof is the convergence of operators.
\begin{prop} \label{conv_op_finale} Fix $F$ in $\mathbf{V}_0$, fix $f$ in $\mathbf{V}_1$ such that $F = \lim \limits_{t \to 0} f_t$,  and fix $(k,v)$ in $K \times \pe$. Then  $\pi_t(k,v) \cdot  f_t$ goes to $\pi_0(k,v)\cdot F$ as $t$ goes to zero.\end{prop}

\begin{proof} We first rearrange $ \pi_t(k,v)f_t := \pi_t(\exp_{G_{t}}(v)k) f_t $ as 
\[  u  \mapsto  e^{\langle -  i\lambda - t \rho, \mathbf{a}_{t}\left[k^{- 1} u \exp_{G_{t}}^{- Ad(u^{- 1}) v}\right] \rangle} \sigma_{t}\left(\mathbf{m}_{t}\left[ k^{- 1} u \exp_{G_{t}}^{- Ad(u^{- 1}) v}  \right]\right)^{- 1}f_t\left(\kappa_{t}(k^{- 1} u \exp_{G_{t}}^{- Ad(u^{- 1}) v) }\right).  \]
Calling in the Iwasawa maps $\mathfrak{I}_{t} : \pe \to \mathfrak{a}_{\chi}$ and $\mathfrak{K}_t: \pe \to K$ from \eqref{iwa_It},  notice that  
\[ \kappa_{t}\left(k^{- 1} u \exp_{G_{t}}^{- Ad(u^{- 1}) v} \right) =k^{- 1} u \ \mathfrak{K}_{t}\left[- Ad(u^{- 1})v \right] ,\] which makes it possible to rewrite $\pi_t(k,v) f_t$ as 
 \[ u \mapsto  e^{\langle - i\lambda - t \rho, \mathfrak{I}_{t}(- Ad(u^{- 1}) v) \rangle} \sigma_{t}\left(\mathbf{m}_{t}\left[k^{- 1} u \exp_{G_{t}}^{- Ad(u^{- 1}) v} \right] \right) f_t\left(k^{- 1} u \ \mathfrak{K}_{t}\left[- Ad(u^{- 1})v \right]\right). \]
Now recall that for all $u$ in $K$,
\begin{equation} \label{interm1} \mathbf{m}_{t}\left[k^{- 1} u \exp_{G_{t}}^{- Ad(u^{- 1} v)} \right] = \mathbf{m}_{t}\left[\exp_{G_{t}}^{- Ad(u^{- 1}) v}.\right] 
 \end{equation}
%
%
%
Set $\beta(u,v) = - Ad(u^{- 1}) v$, an element of the compact orbit $\ad(K_\chi) v$ in $\pe$. Then we just saw that
\begin{align} \label{Final1} \hspace{- 0cm} \pi^{t}(k,v) f_t = u \mapsto &  \quad { \exp{\langle - i\chi - t \rho, \mathfrak{I}_{t}(\beta(u,v)) \rangle}} \\ \label{Final2} & \quad  \times {\sigma_{t} \left(\mathbf{m}_{t}\left[\exp_{G_{t}}^{\beta(u,v)}\right]   \right) f_t\left(k^{- 1} u \ \mathfrak{K}_{t}\left[ \beta(u,v) \right]\right)}. \end{align}
When $t$ goes to zero, we already know from Lemma \ref{conv_ondes} that  \eqref{Final1} goes to $e^{i \langle Ad^\star(u)\chi, v\rangle}$, uniformly in $u$ because of Lemma \ref{conv_iwa}. To complete the proof of Proposition \ref{conv_op_finale}, it is enough to establish that \eqref{Final2} goes to $F(k^{-1} u)$ as $t$ goes to zero. For this, we call in Lemma \ref{bonnedefo_dolb} and use the extension of $\sigma_t$ to all of ${E}$ to rearrange  \eqref{Final2}. To simplify the notations, set $\Sigma_t(u,v)=\sigma_{t} \left(\mathbf{m}_{t}\left[\exp_{G_{t}}^{\beta(u,v)}\right]   \right);$ then
\begin{align} 
 \tag{Term 1}  \hspace{-1cm}{\sigma_{t} \left(\mathbf{m}_{t}\left[\exp_{G_{t}}^{\beta(u,v)}\right]   \right) f_t\left(k^{- 1} u \ \mathfrak{K}_{t}\left[ \beta(u,v) \right]\right)}=  &  \ \  f_{t}(k^{- 1} u) \\
   \tag{Term 2} & + \left[\Sigma_{t}(u,v)  - Id_{E} \right]\left[ f_{t}(k^{- 1}u) - f_{0}(k^{- 1}u) \right]  \\ 
 \tag{Term 3}   & + \Sigma_{t}(u,v) \left[  f_{t}\left(k^{- 1} u \ \mathfrak{K}_{t}\left[ \beta(u,v) \right]\right) - f_{t}(k^{- 1} u) \right] \\ 
  \tag{Term 4}  & +  \left[ \Sigma_{t}(u,v)  - Id_{E} \right]\left[ f_{0}(k^{- 1}u) \right] . 
\end{align}
\noindent As $t$ goes to zero,  \begin{itemize}
\item[$\bullet$] \emph{Term 1} goes to $f_0(k^{-1} u) = F(k^{-1} u)$.
\item[$\bullet$] \emph{Term 2} goes to zero because of the Lipschitz estimate in Lemma \ref{bonnedefo_dolb}: since $\beta(u,v)$ lies on a compact orbit in $\pe$, there is a number $C$, independent of $t$ and $u$, such that  $\Sigma_{t}(u,v) $ is $C$-Lipschitz for all $t,u$ (recall that $v$ is fixed here). Since $f_t-f_0$ goes to zero as $t$ goes to zero, so does Term 2.
\item[$\bullet$]  \emph{Term 3} also goes to zero: it can be rewritten as $ \left(\Sigma_{t}(u,v) \circ \mathbf{c}_t \right)\left[  f\left(k^{- 1} u \ \mathfrak{K}_{t}\left[ \beta \right]\right) - (k^{- 1} u) \right]$; from Lemma \ref{conv_iwa} and the compactness of $\left\{ \beta(u,v), u \in K\right\}$, we know that $u \mapsto f\left(k^{- 1} u \ \mathfrak{K}_{t}\left[ \beta(u,v) \right]\right) - (k^{- 1} u)$ goes to zero as $t$ goes to zero, and we use the uniform Lipschitz estimate of Lemma \ref{bonnedefo_dolb} to conclude as we did for Term 2.
\item[$\bullet$]  To prove that \emph{Term 4} goes to zero, we need to prove that  $\Sigma_{t}(u,v)f_0(k^{-1}u)$ goes to $f_0(k^{-1} u)$. Now, $f_0$ is by definition a $\mathbf{c}_t$-invariant element of ${E}$, so for all $u$ in $K$,
\begin{align*}
 \Sigma_{t}(u,v)f_0(k^{-1}u) &= \Sigma_{t}(u,v)\cdot \mathbf{c}_t\cdot  f_0(k^{-1}u) \\ 
 &= \sigma_{t} \left(\mathbf{m}_{t}\left[\exp_{G_{t}}^{\beta(u,v)}\right]   \right)\cdot \mathbf{c}_t\cdot f_0(k^{-1}u) \\
  &= \sigma_1 \left(\varphi_1^{-1} \circ \varphi_t \left\{\mathbf{m}_{t}\left[\exp_{G_{t}}^{\beta(u,v)}\right]\right\}  \right) f_0(k^{-1}u) \\ 
  &=  \sigma_1 \left(\mathbf{m}_{1}\left[\exp_{G_1}^{t \cdot \beta(u,v)}\right] \right) f_0(k^{-1}u).
  \end{align*}
The transition between the second line and the third uses the intertwining relation ($\text{Nat}_{\text{RIC}}$) for $\mathbf{c}_t$, and that between the third line and the fourth uses the fact that $\beta(u,v)$ lies in $\pe$.

As $t$ goes to zero, the element $\mathbf{m}_{1}\left[\exp_{G_1}^{t \cdot \beta(u,v)}\right]$ of $G_1$ goes to the identity of $G_1$; from the continuity of $\sigma_1$ (see Lemma   \ref{bonnedefo_dolb}) we deduce that $\Sigma_{t}(u,v)f_0(k^{-1}u) $ goes to $f_0(k^{-1}u) $, as desired. 
\end{itemize} 
All convergences are uniform in $u$; Proposition \ref{conv_op_finale} and Theorem \ref{contraction_finale} follow. \end{proof}
\newpage

\bibliography{mackey_part_2}
\bibliographystyle{afgou_bib}

\end{document}